\documentclass{article}
\usepackage[english]{babel}
\usepackage{amsmath,amssymb,graphicx,latexsym,pslatex,theorem}
\usepackage{hyperref} 

\newcommand{\Alpha}{\mathrm{A}}
\newcommand{\Mu}{\mathrm{M}}
\newcommand{\Tau}{\mathrm{T}}
\newcommand{\assign}{:=}
\newcommand{\comma}{{,}}
\newcommand{\longdownminus}{{\mbox{\rotatebox[origin=c]{-90}{$\longminus$}}}}
\newcommand{\longminus}{{-\!\!-}}
\newcommand{\mathd}{\mathrm{d}}
\newcommand{\mathi}{\mathrm{i}}
\newcommand{\mathlambda}{\lambda}
\newcommand{\mathpi}{\pi}
\newcommand{\of}{:}
\newcommand{\textdots}{...}
\newcommand{\tmaffiliation}[1]{\\ #1}
\newcommand{\tmmathbf}[1]{\ensuremath{\boldsymbol{#1}}}
\newcommand{\tmop}[1]{\ensuremath{\operatorname{#1}}}
\newenvironment{proof}{\noindent\textbf{Proof\ }}{\hspace*{\fill}$\Box$\medskip}
\newtheorem{corollary}{Corollary}
\newtheorem{definition}{Definition}
{\theorembodyfont{\rmfamily}\newtheorem{example}{Example}}
\newtheorem{lemma}{Lemma}
\newtheorem{notation}{Notation}
\newtheorem{proposition}{Proposition}
{\theorembodyfont{\rmfamily}\newtheorem{remark}{Remark}}
\newtheorem{theorem}{Theorem}

\begin{document}

\title{Asymptotics of Harish-Chandra transform and infinitesimal freeness}

\author{
  Alexey Bufetov, Panagiotis Zografos
  \tmaffiliation{Leipzig University, Institute of Mathematics}
}

\date{}
\maketitle

\begin{abstract}

In the last ten years a technique of Schur generating functions and Harish-Chandra transforms was developed for the study of the asymptotic behavior of discrete particle systems and random matrices. In the current paper we extend this toolbox in several directions. We establish general results which allow to access not only the Law of Large Numbers, but also next terms of the asymptotic expansion of averaged empirical measures. In particular, this allows to obtain an analog of a discrete Baik-Ben Arous-Peche phase transition. A connection with infinitesimal free probability is shown and a quantized version of infinitesimal free probability is introduced. Also, we establish the Law of Large Numbers for several new regimes of growth of a Harish-Chandra transform.
  
\end{abstract}

\tableofcontents

\section{Introduction}
\label{sec:intro}

\subsection*{Overview}

Let $A$ be a random Hermitian $N \times N$ matrix with (possibly random) eigenvalues $\{\lambda_1 (A) \leq \cdots \leq \lambda_N (A)\}$. Its \textit{Harish-Chandra transform} (also known as a multivariate Bessel generating function) is defined by 
\begin{equation}
  \mathbb E [H \text{} C (x_1, \ldots, x_N ; \lambda_1 (A), \ldots, \lambda_N (A))] \assign \mathbb E \int_{U (N)} \exp (\tmop{Tr} (A \text{} U \text{} B
  \text{} U^{\ast})) \tmmathbf{m}_N (d \text{} U), \qquad x_i \in \mathbb C, \label{intro-Haarmeasure}
\end{equation}
where $B$ is a deterministic diagonal matrix with eigenvalues $x_1, x_2, \dots, x_N$, and $U$ is integrated with respect to the Haar measure on unitary $N \times N$ matrices. It is well-known by now that the asymptotic behavior of function \
$\mathbb E [H \text{} C (x_1, \ldots, x_N ; \lambda_1 (A), \ldots, \lambda_N (A))] $ can be used for the analysis of the asymptotic behavior of eigenvalues of $A$. In particular, it was established in \cite{BG} (see also \cite{GM}, \cite{GP}, \cite{MN}; we omit minor technical assumptions) that the following limit of the Harish-Chandra transform implies the weak Law of Large Numbers convergence of the empirical measure: 
\begin{equation}
\label{eq:Intro-free-conv}
\lim_{N \rightarrow \infty} \frac{1}{N} \log \text{} \mathbb{E} [H \text{}
    C (x_1, \ldots, x_r, 0^{N - r} ; \lambda_1 (A), \ldots, \lambda_N (A))] =
    \sum_{i = 1}^r \Psi (x_i) \ \ \Rightarrow \ \lim_{N \rightarrow \infty} \frac{1}{N} \sum_{i=1}^N \delta \left( \frac{\lambda_i (A)}{N} \right) = \mu,
\end{equation} 
where the former convergence is uniform in a complex neighborhood of $0^r:= (\underbrace{0,0,\dots,0}_r)$ and should hold for arbitrary fixed $r$, and the function  $\Psi' (x)$ is the R-transform of a probability measure $\mu$. In \cite{BG} this claim was also extended to the case  of Schur generating functions and applied to problems coming from random tilings and asymptotic representation theory.

For deterministic $\{\lambda_i (A)\}$, in a somewhat different direction, \cite[Theorem 4.1]{OV} can be written in the following form: 
\begin{equation}
\label{eq:Intro-erg-conv}
\lim_{N \rightarrow \infty} \log \text{} \mathbb{E} [H \text{}
    C (x_1, \ldots, x_r, 0^{N - r} ; \lambda_1 (A), \ldots, \lambda_N (A))] =
    \sum_{i = 1}^r \Phi (x_i) \ \ \Rightarrow \ \lim_{N \rightarrow \infty} \sum_{i=1}^N \delta \left( \frac{\lambda_i}{N} \right) = \nu,
\end{equation} 
where the former convergence is again uniform in a complex neighborhood of $0^r$ and should hold for arbitrary fixed $r$, while the latter convergence is the convergence in the sense of moments applied to measures of growing weight (this implies that for such a convergence to take place most of the weight should be around 0). This theorem plays a crucial role in the classification of infinite ergodic unitarily invariant Hermitian matrices, see \cite{OV}. 

The striking similarity between \eqref{eq:Intro-free-conv} and \eqref{eq:Intro-erg-conv} is the starting point of the current paper. These two results play very important role in two quite different settings with different sets of applications. The goal of this paper is to establish other asymptotic results of this form and to start to explore their applications. In particular, we address in detail perturbations (or corrections) to \eqref{eq:Intro-free-conv}.

The main results of this paper are
\begin{itemize}

\item in Theorem \ref{HALFGODHALFSOUVLAKI} from Section \ref{sec:interm} we establish the \textit{intermediate regime} which interpolates between \eqref{eq:Intro-free-conv} and \eqref{eq:Intro-erg-conv}. In Theorem \ref{BEBT!} we prove the implication \eqref{eq:Intro-erg-conv} for \textit{random} eigenvalues $\{\lambda_i (A)\}$. Also, in Theorem \ref{Propolemiko} we study the behavior of eigenvalues in the case if the Harish-Chandra transform grows faster than in \eqref{eq:Intro-free-conv}. 

\item in Theorem \ref{Einaiiagapi} from Section \ref{sec:inf-free} we generalize \eqref{eq:Intro-free-conv} by computing the first two leading terms of the asymptotic expansion for the expectation of the moments of the empirical measure. We also connect it to the notion of the infinitesimal freeness from free probability. 

\item in Theorem \ref{Maria!!} we establish an analog of Theorem \ref{Einaiiagapi} for a related setup of Schur generating functions (see details below) and connect it to the quantized free convolution; by doing this, we introduce the \textit{quantized} infinitesimal cumulants. As applications, we calculate the outliers in a perturbation of the model of uniformly random domino tilings of the Aztec diamond, see Figure \ref{fig:til} and Example \ref{ex:Aztec}, and demonstrate a BBP-type phase transition in asymptotic representation theory in Example \ref{ex:disc-BBP}.

\item in Theorems \ref{th:sec-order-main} and \ref{fredi} we extend Theorem \ref{Einaiiagapi} to the next terms of asymptotic expansion of the moments of the empirical measure, and connect them with the second and higher order infinitesimal freeness, respectively. We also provide several examples, in particular, in Example \ref{ex:BBP-highOrder} we demonstrate a version of a higher order BBP phase transition. 

\end{itemize}

The main focus of this paper is on establishing general theorems in various growth regimes of the Harish-Chandra transform and on establishing precise connections with (quantized) infinitesimal freeness. However, we also provide a number of examples. 
Examples from Section \ref{sec:inf-free} can be calculated (or were already calculated) with the use of the existing techniques of infinitesimal freeness, we included them in order to better illustrate our general results. Examples from Sections \ref{sec:Schur-free}, \ref{sec:second-order} and \ref{sec:higher-order} seem to be new. 

Section \ref{sec:prelim} contains required preliminaries. In Section \ref{sec:previous-result} we provide a technically improved and very detailed proof of (a degeneration of) \cite[Theorem 5.1]{BG}, which serves to us as a reference point for further progress. In the remainder of the introduction we discuss our results in more detail. 

\subsection*{Intermediate regime}

The main result of Section \ref{sec:interm} is the following.

\begin{theorem} (Theorem \ref{HALFGODHALFSOUVLAKI})
  \label{th:intro-HALFGODHALFSOUVLAKI}Let $A$ be a random Hermitian matrix of size $N$
  and $0 < \theta < 1$. Assume that for every finite $r$ we have
  \begin{equation}
    \lim_{N \rightarrow \infty} \frac{1}{N^{\theta}} \log \mathbb{E} [H
    \text{} C (x_1, \ldots, x_r, 0^{N - r} ; \lambda_1 (A), \ldots, \lambda_N
    (A))] = \sum_{i = 1}^r \Psi (x_i), \label{intro-manytoomany}
  \end{equation}
  where $\Psi$ is a smooth function in a complex neighborhood of $0$ and the above
  convergence is uniform in a complex neighborhood of $0^r$. Then, for every $k \in
  \mathbb{N}$, the $k$-th moment of the random measure $N^{- \theta} \sum_{i =
  1}^N \delta \left( N^{- 1} \lambda_i (A) \right)$ converges in
  probability to $\dfrac{\Psi^{(k)} (0)}{(k - 1)!}$, as $N \rightarrow \infty$.
\end{theorem}

This theorem shows that the random empirical measure $N^{- \theta} \sum_{i =
  1}^N \delta \left( N^{- 1} \lambda_i (A) \right)$ converges in the sense of moments also in this new limit regime, which interpolates between $\theta=1$ case \eqref{eq:Intro-free-conv} and $\theta=0$ case \eqref{eq:Intro-erg-conv}. The asymptotic behavior of a Harish-Chandra transform encoded by function $\Psi (x)$ can be translated into the information about the limit of the empirical measure: As visible from the statement of the theorem, the moment generating function of the limiting measure is $\Psi' (x)$.
  
It is interesting to note that this intermediate regime shares properties with both of the $\theta=0$ and $\theta=1$ cases. Similarly to the $\theta=1$ case, the limit (in terms of moments) of $N^{- \theta} \sum_{i=1}^N \delta \left( N^{- 1} \lambda_i (A) \right)$ can be essentially an arbitrary probability measure. On the other hand, this measure is related to the limit of the Harish-Chandra transform in the same way as in the $\theta=0$ case. 

We prove two more results in Section \ref{sec:interm}. In Theorem \ref{BEBT!} we prove the implication \eqref{eq:Intro-erg-conv} for \textit{random} eigenvalues $\{\lambda_i (A)\}$ (in \cite{OV} only deterministic ones were considered). This provides a new proof of \cite[Theorem 4.1]{OV}, and can potentially lead to new applications in such a limit regime, which is the most natural one from the point of view of asymptotic representation theory.

Also, for the sake of completeness, in Theorem \ref{Propolemiko} we study the behavior of eigenvalues in the case if the value of $\theta$ in \eqref{intro-manytoomany} is greater than 1. In such a situation, we show that the limit becomes degenerate. 

\subsection*{Random matrices: Infinitesimal freeness}

Equation \eqref{eq:Intro-free-conv} implies the Law of Large Numbers for the empirical (random) measure:
$$
\lim_{N \to \infty} \frac{1}{N} \sum_{i=1}^N \delta \left( \frac{\lambda_i}{N} \right) = \mu,
$$
where $\mu$ is a deterministic probability measure. There are (at least) two natural sources of a more detailed stochastic information about the empirical measure in this setup. One of them is the Central Limit Theorem (=CLT); it was studied (with the use of more detailed assumptions on asymtotics of the Harish-Chandra transform) in \cite{BG2}, \cite{BK}, \cite{BG3}, \cite{GS}, \cite{H}, \cite{BL}, among others. The Central Limit Theorem in all of these papers is studied on the scale $\frac{1}{N}$ (if we assume that the Law of Large Numbers is on a constant scale), since most of the main applications in random matrix theory, random tilings, and asymptotic representation theory have fluctuations of such size. 

However, there is another contribution to this scale: The second leading term in the asymptotics of the expected value of observables of the measure. We will often refer to it as a \textit{correction} (to the Law of Large Numbers). Since it lives on the same scale, one can argue that it is of comparable to CLT importance for the stochastic models under study. For example, it is known that this is the source of \textit{outliers} in perturbed random matrix ensembles. From the free probability side, this scaling is studied in the framework of the \textit{infinitesimal freeness}, see, e.g., \cite{Au}, \cite{BS}, \cite{BGN}, \cite{FN}, \cite{Min}, \cite{Sh}. 

In the current paper, we establish a general result which allows to obtain the information about such a scaling from the asymptotics of a Harish-Chandra transform. 

\begin{theorem} (Theorem \ref{Einaiiagapi})
  \label{th:intro-Einaiiagapi}Let $A$ be a random Hermitian matrix of size $N$. Assume that for every finite $r$ one has
  \begin{equation}
    \lim_{N \rightarrow \infty} N \left( \frac{1}{N} \log \mathbb{E}[H \text{}
    C (x_1, \ldots, x_r, 0^{N - r} ; \lambda_1 (A), \ldots, \lambda_N (A))] -
    \sum_{i = 1}^r \Psi (x_i) \right) = \sum_{i = 1}^r \Phi (x_i), \label{intro-ola}
  \end{equation}
  where $\Psi, \Phi$ are smooth functions in a (complex) neighborhood of $0$ and the
  above convergence is uniform in a (complex) neighborhood of $0^r$. Then one has the following     limit:
  \begin{equation}
    \lim_{N \rightarrow \infty} N \left( \frac{1}{N^{k + 1}} \mathbb{E} \left[
    \sum_{i = 1}^N \lambda_i^k (A) \right] -\tmmathbf{\mu}_k \right)
    =\tmmathbf{\mu}'_k, \label{intro-rin}
  \end{equation}
  where 
  \begin{equation}
    \tmmathbf{\mu}_n = \sum_{m = 0}^{n - 1} \frac{n!}{m! (m + 1) ! (n - m) !}
    \frac{\mathd^m}{\mathd x^m} ((\Psi' (x))^{n - m}) \longdownminus_{x =
    0}, \label{intro-topg}
  \end{equation}
  \begin{equation}
    \tmmathbf{\mu}'_k \assign \sum_{m = 0}^{k - 1} \frac{k!}{m! (m + 1) !
    (k - m - 1) !}  \left. \frac{\mathd^m}{\mathd x^m} ((\Psi' (x))^{k - m
    - 1} \Phi' (x)) \right|_{x = 0} . \label{intro-G}
  \end{equation}
\end{theorem}

After proving this result in Section \ref{sec:inf-free}, we explain how the formula \eqref{intro-G} is connected with the known notions of infinitesimal freeness and cumulants, and provide several examples. 

\subsection*{Schur generating functions: Infinitesimal freeness}

In Section \ref{sec:Schur-free} we study a related setup of Schur generating functions instead of Harish-Chandra transforms. Let us briefly recall necessary definitions. 

A signature is an $N$-tuple of integers
$\mathlambda_1 \geq \mathlambda_2 \geq \cdots \geq \mathlambda_N$. We denote
by $\hat{U} (N)$ the set of all signatures.
Information about a signature $\lambda$ can be encoded by a discrete
probability measure on $\mathbb{R}$
\begin{equation}
  \tmmathbf{m}_N [\mathlambda] \assign \frac{1}{N} \sum_{i = 1}^N
  \delta \left( \frac{\mathlambda_i + N - i}{N} \right) . \label{intro-nono}
\end{equation}
For $\mathlambda \in \hat{U} (N)$ chosen at random with respect to a
probability measure $\varrho$ on $\hat{U} (N)$, we are interested in the
asymptotic behavior of the random measure (\ref{intro-nono}), which is a discrete analog of an empirical measure.

A \textit{Schur function} is defined by
\begin{equation}
  \tmmathbf{\chi}^{\mathlambda} (u_1, \ldots, u_N) := \frac{\det
  (u_i^{\mathlambda_j + N - j})}{\det \left( u_i^{N - j} \right) } = \frac{\det
  \left( u_i^{\mathlambda_j + N - j} \right) }{\prod_{i < j} (u_i - u_j)}, \qquad \mathlambda \in \hat{U} (N), \ u_i \in \mathbb C. \label{intro-Maroko} 
\end{equation}  
A \textit{Schur generating function} $S_{\varrho}^{U (N)}$ is a symmetric Laurent
  power series in $(u_1, \ldots, u_N)$ given by
  \[ S_{\varrho}^{U (N)} (u_1, \ldots, u_N) := \sum_{\mathlambda \in \hat{U}
     (N)} \varrho (\mathlambda) \frac{\tmmathbf{\chi}^{\mathlambda} (u_1,
     \ldots, u_N)}{\tmmathbf{\chi}^{\mathlambda} (1^N)},  \]
where $1^N$ stands for $\underbrace{(1,1,\dots,1)}_N$. It can be viewed as a discrete generalization of a Harish-Chandra transform, see, e.g., \cite[Proposition 1.5]{BG}.

In the current paper we establish a general result which allows to obtain the information about two leading terms of the expectation of the empirical measure \eqref{intro-nono} from the asymptotics of its Schur generating function. 

\begin{theorem} (Theorem \ref{Maria!!})
  \label{th:intro-Maria!!}Let $\varrho (N)$ be a sequence of probability measures on
  $\hat{U} (N)$ such that for every finite $r$ one has
  \begin{equation}
    \lim_{N \rightarrow \infty} N \left( \frac{1}{N} \log S_{\varrho (N)}^{U
    (N)} (u_1, \ldots, u_r, 1^{N - r}) - \sum_{i = 1}^r \Psi (u_i) \right) =
    \sum_{i = 1}^r \Phi (u_i), \label{intro-thatonegamiswtonKinezo}
  \end{equation}
  where $\Psi, \Phi$ are analytic functions in a complex neighborhood of $1$ and the
  above convergence is uniform in a complex neighborhood of $1^r$. Then for every $k
  \in \mathbb{N}$ we have
  \begin{equation}
    \lim_{N \rightarrow \infty} N \left( \frac{1}{N} \sum_{\mathlambda \in
    \hat{U} (N)} \varrho (N) [\mathlambda] \sum_{i = 1}^N \left(
    \frac{\mathlambda_i + N - i}{N} \right)^k -\mathbf{m}_k \right)
    =\mathbf{m}'_k, \label{intro-magapaskiegwxezw}
  \end{equation}
  where 
  \begin{equation}
    \mathbf{m}_k = \sum_{m = 0}^k
    \frac{k!}{m! (m + 1) ! (k - m) !}  \left. \frac{\mathd^m}{\mathd x^m} (x^k
    (\Psi' (x))^{k - m}) \right|_{x = 1}, \label{intro-mporw?}
  \end{equation}
  \begin{equation}
    \mathbf{m}'_k = \sum_{m = 0}^{k - 1} \frac{k!}{m! (m + 1) ! (k - m - 1) !}
    \left. \frac{\mathd^m}{\mathd x^m} \left( x^k \left( \Phi' (x) -
    \frac{1}{2 x} \right) (\Psi' (x))^{k - m - 1} \right) \right|_{x = 1} .
    \label{intro-poueinai}
  \end{equation}
\end{theorem}

This discrete setup of \textit{quantized free probability} is closely but non-trivially related to free probability, see \cite{BG}. We extend this connection to the level of infinitesimal freeness by introducing the infinitesimal quantized free cumulants in Theorem \ref{th:cumulants}, see also Remark \ref{rem:quant-class-free}.

It is known that several classes of discrete stochastic particle systems can be described with the help of Schur generating functions, see, e.g., \cite{BG}, \cite{BG2}, \cite{BK}, \cite{BL}. We provide two applications of Theorem \ref{th:intro-Maria!!} which can be of independent interest.  


\begin{figure}
\center
\includegraphics[height=9cm]{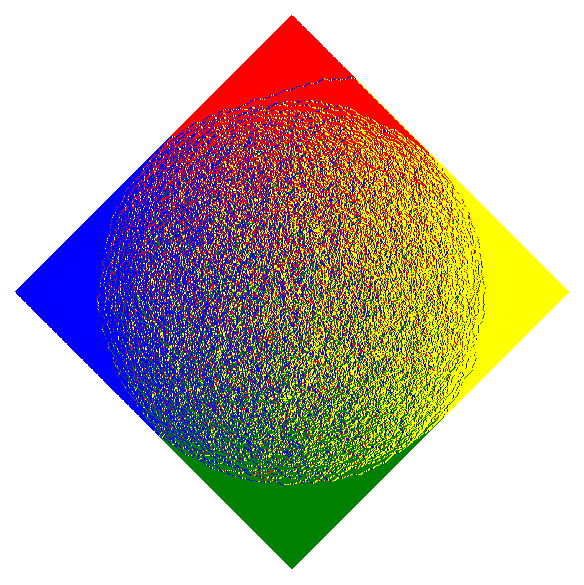}
\caption{Uniformly random domino tilings of Aztec diamond with a perturbation of weight 10 along the up-right edge. See Example \ref{ex:Aztec} for details.}
\label{fig:til}
\end{figure}


\subsection*{Examples in a discrete setup}

In Examples \ref{ex:disc-BBP} and \ref{ex:Aztec} we provide new results about two specific discrete systems. We formulate the results briefly here in the introduction; see the main part of the paper for more details. 

Let $\gamma>1$ and $\alpha>0$ be fixed real numbers. Let $\varrho_{dBBP} (N)$ be a probability measure on $\hat U(N)$ corresponding to the Schur generating function 
$$
S_{\varrho_{dBBP} (N)}^{U (N)} (u_1, \ldots, u_N) := \prod_{i=1}^N \exp \left( \gamma N ( u_i-1) \right) \frac{1}{1-\alpha (u_i-1)}.
$$
It is known that such a probability measure exists, since the expression above comes from an extreme character of the infinite-dimensional unitary group (see Section \ref{sec:Schur-free}). For any $k \in \mathbb{N}$, in Example \ref{ex:disc-BBP} we establish the following asymptotic expansion
\begin{equation}
\label{intro-eq:asym-dBBP}
\frac{1}{N} \sum_{\mathlambda \in
    \hat{U} (N)} \varrho_{dBBP} (N) [\mathlambda] \sum_{i = 1}^N \left(
    \frac{\mathlambda_i + N - i}{N} \right)^k 
    = \mathbf{m}_{k; dBBP} + \frac{1}{N} \mathbf{m}'_{k; dBBP} + o \left( \frac{1}{N} \right), \qquad N \to \infty,
\end{equation}
where 
$$
\mathbf{m}_{k; dBBP} = \int_{\mathbb{R}} t^k d \mu_{dBBP}, \qquad \mathbf{m}'_{k; dBBP} = \int_{\mathbb{R}} t^k d \mu'_{dBBP},
$$
while a probability measure $\mu_{dBBP}$ and a signed measure $\mu'_{dBBP}$ are given by

\[ d \mu_{dBBP} = \frac{1}{\mathpi} \mathbf{1}_{t \in [\gamma + 1 - 2 \sqrt{\gamma}, \gamma + 1
     + 2 \sqrt{\gamma}]} \arccos \left( \frac{t - 1 + \gamma}{2
     \sqrt{\gamma t}} \right) dt, 
 \]

\begin{multline}
\label{intro-eq:BBP-answer}
d \mu'_{dBBP} = \delta \left( \alpha + 1 + \gamma + \frac{\gamma}{\alpha} \right) \mathbf{1}_{\alpha > \sqrt{\gamma}} + \frac12 \delta \left( \alpha + 1 + \gamma + \frac{\gamma}{\alpha} \right) \mathbf{1}_{\alpha = \sqrt{\gamma}} \\ + \left( \frac{ \alpha ( t - 2 \alpha - \gamma - 1)}{\gamma + \alpha^2 +
     \alpha \gamma + \alpha - \alpha t} - \frac{t + 1 - \gamma}{2t} \right) \frac{1}{2 \pi \sqrt{(\gamma + 1 + 2
     \sqrt{\gamma} - t) (t - \gamma - 1 + 2 \sqrt{\gamma})}} \mathbf{1}_{ t \in [(\sqrt{\gamma}-1)^2;(\sqrt{\gamma}+1)^2]} dt. 
\end{multline}

The new result is the second term in the right-hand side of \eqref{intro-eq:asym-dBBP}, given by an explicit formula \eqref{intro-eq:BBP-answer}; the first term is essentially known from \cite{Bia2}, see also \cite{BBO}. An important feature of the result is the presence of delta measures in \eqref{intro-eq:BBP-answer} -- in fact, this can be viewed as a discrete analog of a well-known Baik-Ben Arous-Peche phase transition for random matrices. It appears in the following setup. Consider the sum of a GUE matrix and a rank-one matrix. Then, depending on the value of a parameter in the rank-one matrix, the outlier in the empirical measure might or might not appear (we recall the exact details in Example \ref{ex:one-pertub}).

The probability measure $\varrho_{dBBP}$ can be thought of in a similar way -- the $\gamma N$ parameter comes from the one-sided Plancherel character, which is a discrete version of a GUE matrix, while $\alpha$ parameter plays the role of a rank-one matrix. The product of these two characters corresponds to the tensor product of representations, which is a discrete analog of the summation of matrices. We see that depending on whether $\alpha > \sqrt{\gamma}$, the outlier does or does not appear in the model. 

Another example that we consider is a deformation of a model of uniformly random domino tilings of the Aztec diamond. It is well-known (we refer to \cite[Section 2]{BK} for a detailed discussion) that the domino tilings are in bijection with sequences of signatures $(\lambda^1, \mu^1, \dots, \lambda^M)$ satisfying the following \textit{interlacing conditions}:
$$
\varnothing \prec \lambda^1 \succ_v \mu^1 \prec \lambda^2 \succ_v \mu^2 \prec \dots \prec \lambda^M \succ_v \varnothing.
$$ 
Therefore, the uniform measure on the domino tilings of the Aztec diamond is equivalent to choosing uniformly at random such a sequence; also, geometric properties of the tilings are encoded by the signatures. Let $A>1$ be a fixed real number. We consider a deformation of the model by attaching the probability $Z A^{|\lambda^M|}$ to each sequence, where $Z$ is the normalizing factor, and $|\lambda^M|$ is the sum of all coordinates of $\lambda^M$. Geometrically, this means that we consider dominoes adjacent to the up-right edge of the Aztec diamond (drawn as in Figure \ref{fig:til}), count how many of them are horizontal (they are colored in yellow in Figure \ref{fig:til}), and the probability of a domino tiling is proportional to $A$ to the power of this amount. 

Let $\varrho_{A} (N)$ be a distribution of $\lambda^N$ under this measure, where $1 \le N \le M$. We assume that $N = \alpha M$, where $\frac12 < \alpha <1$ is a fixed real number. For any $k \in \mathbb{N}$, we obtain in Example \ref{ex:Aztec} the asymptotic expansion
\begin{equation}
\label{intro-eq:asym-Aztec}
\frac{1}{N} \sum_{\mathlambda \in
    \hat{U} (N)} \varrho_{A} (N) [\mathlambda] \sum_{i = 1}^N \left(
    \frac{\mathlambda_i + N - i}{N} \right)^k 
    = \mathbf{m}_{k; A} + \frac{1}{N} \mathbf{m}'_{k; A} + o \left( \frac{1}{N} \right), \qquad N \to \infty,
\end{equation}
where 
$$
\mathbf{m}_{k; A} = \int_{\mathbb{R}} t^k d \mu_{A}, \qquad \mathbf{m}'_{k; A} = \int_{\mathbb{R}} t^k d \mu'_{A},
$$
and
\begin{equation}
\label{intro-eq:Aztec-LLN}
d \mu_{A} = \mathbf{1}_{(1-2\alpha t)^2 \le 1 - (1-2 \alpha)^2 } \frac{1}{\pi} \arccos \left( \frac{1 - 2 \alpha}{\sqrt{1 - (2 \alpha t -1)^2}} \right) dt + \mathbf{1}_{1 \ge (1-2\alpha t)^2 > 1 - (1-2 \alpha)^2 } dt,
\end{equation}

\begin{multline}
\label{intro-eq:Aztec-correction}
d \mu'_{A} = - \mathbf{1}_{\alpha > \frac{(A+1)^2}{2(A^2+1)}} \delta \left( \frac{-1 + 2\alpha A -A}{\alpha ( A^2 -1 )} \right) - \frac12 \mathbf{1}_{\alpha = \frac{(A+1)^2}{2(A^2+1)}} \delta \left( \frac{-1 + 2\alpha A -A}{\alpha ( A^2 -1 )} \right) \\ + \left( \alpha+\frac{-2 \alpha+1}{4t}-\frac{\alpha (2 \alpha-1)}{4( \alpha t-1)}+\frac{\alpha (2 A^2 \alpha-A^2-2A+2 \alpha-1)}{2(-\alpha t +A^2 \alpha t-2 \alpha A+1+A)} \right) \frac{\mathbf{1}_{1-(2 \alpha-1)^2-(2 \alpha t-1)^2 >0}}{\pi \sqrt{1-(2 \alpha-1)^2-(2 \alpha t-1)^2}} dt.
\end{multline}

The new result is the second term in the right-hand side of \eqref{intro-eq:asym-Aztec}, given by an explicit formula \eqref{intro-eq:Aztec-correction}. The first term is given by the arctic circle theorem, see \cite{JPS}, \cite{CKP}, \cite{KO}, \cite{J}. 

Let us explain how this formula corresponds to geometric properties of a tiling. Since we scale in Theorem \ref{th:intro-Maria!!} by $N$, while geometrically we scale by $M$ in both directions (so that the Aztec diamond remains a square after the rescaling), we need to introduce variable $\tilde y = \alpha y$ for a horizontal direction. In coordinates $( \tilde y, \alpha)$ the last summand of the correction measure \eqref{intro-eq:Aztec-correction} provides the density of the correction inside the arctic circle $1-(2 \alpha-1)^2-(2 \tilde y-1)^2=0$. Next, we see that the outlier is present if and only if $\alpha > \frac{(A+1)^2}{2(A^2+1)}$ at point $\tilde y = \frac{-1 + 2\alpha A -A}{( A^2 -1 )}$. Therefore, the outliers form a segment on a line $\frac{1-A^2}{2A} \alpha + \tilde y = \frac{A+1}{2A}$, which is tangent to the arctic circle and intersects with it at the point $ \left( \frac{1}{A^2+1}, \frac{(A+1)^2}{2(A^2+1)} \right)$. Thus, the result matches the effect visible in Figure \ref{fig:til}.

It is interesting to note that the perturbation described above is similar to the one considered in the Tangent Method, see \cite{CS}, \cite{A}, \cite{DG1}, \cite{DG2}, \cite{KDR}, \cite{R}. In this method, one fixes positions of dominoes along the up-right edge of the Aztec diamond and considers uniformly random domino tilings of the remaining domain. As shown by the computation above, our introduction of a parameter $A$ leads to the same asymptotic behavior of outliers and the dominoes along the up-right edge as prescribed by the Tangent Method.
However, in principle these are two distinct perturbations; for example, it is not obvious whether the correction measures inside the arctic circle should also coincide for these two perturbations. 


\subsection*{Higher order freeness}

Given the connection of the Harish-Chandra transform and infinitesimal freeness from Theorem \ref{th:intro-Einaiiagapi}, it is natural to attempt to access further information about the averaged empirical measure. 
Our further results allow to extract next terms of its asymptotic expansion from the information about the Harish-Chandra transform. We establish two results in this direction for two slightly different asymptotic regimes.

The main result of Section \ref{sec:second-order} is the following.

\begin{theorem} (Theorem \ref{th:sec-order-main})
\label{intro-th:sec-order-main}
  Let $A$ be a random Hermitian matrix of size $N$ and assume that for every
  finite $r$ one has
  \begin{equation} 
  \lim_{N \rightarrow \infty} N^2 \left( \frac{1}{N} \log \mathbb{E}[H
     \text{} C (x_1, \ldots, x_r, 0^{N - r} ; \lambda_1 (A), \ldots, \lambda_N
     (A))] - \sum_{i = 1}^r \Psi (x_i) - \frac{1}{N} \sum_{i = 1}^r \Phi (x_i)
     \right) = \sum_{i = 1}^r \Tau (x_i), \label{intro-audikos}
     \end{equation} 
  where $\Psi, \Phi, \Tau$ are smooth functions in a complex neighborhood of $0$ and
  the above convergence is uniform in a complex neighborhood of $0^r$. Then the $1 /
  N^2$ correction of the average empirical distribution of $N^{- 1} A$
  is given by
  \[ \lim_{N \rightarrow \infty} N^2 \left( \frac{1}{N^{k + 1}}
     \mathbb{E}[\tmop{Tr} (A^k)] -\tmmathbf{\mu}_k - \frac{1}{N}
     \tmmathbf{\mu}'_k \right) =\tmmathbf{\mu}''_k
     +\tmmathbf{\nu}''_k, \]
  where $(\tmmathbf{\mu}_k)_{k \in \mathbb{N}}, (\tmmathbf{\mu}'_k)_{k \in
  \mathbb{N}}$ are given by (\ref{intro-topg}),(\ref{intro-G}) respectively. The sequences
  $(\tmmathbf{\mu}''_k)_{k \in \mathbb{N}}, (\tmmathbf{\nu}''_k)_{k \in
  \mathbb{N}}$ are given by
  \[ \tmmathbf{\mu}''_k = \sum_{m = 0}^{k - 1} \frac{k!}{m! (m + 1) ! (k -
     m - 1) !}  \left. \frac{\mathd^m}{\mathd x^m} ((\Psi' (x))^{k - m - 1}
     \Tau' (x)) \right|_{x = 0} \]
  \begin{equation}
    \text{{\hspace{10em}}} + \frac{1}{2} \sum_{m = 0}^{k - 2} \frac{k!}{m! (m
    + 1) ! (k - m - 2) !}  \left. \frac{\mathd^m}{\mathd x^m} ((\Psi'
    (x))^{k - m - 2} (\Phi' (x))^2) \right|_{x = 0} \label{intro-Ss}
  \end{equation}
  and
  \[ \tmmathbf{\nu}''_k = \frac{1}{24} \sum_{m = 0}^{k - 3} \frac{k!}{m! (m
     + 1) ! (k - m - 3) !}  \left. \frac{\mathd^m}{\mathd x^m} ((\Psi'
     (x))^{k - m - 3} \Psi''' (x)) \right|_{x = 0} \]
  \[ \text{{\hspace{4em}}} + \frac{1}{12} \sum_{m = 0}^{k - 3} \frac{k!}{m! (m
     + 2) ! (k - m - 3) !}  \left. \frac{\mathd^m}{\mathd x^m} ((\Psi'
     (x))^{k - m - 3} \Psi''' (x)) \right|_{x = 0} \]
  \begin{equation}
    \text{{\hspace{10em}}} + \frac{1}{12} \sum_{m = 0}^{k - 4} \frac{k!}{m! (m
    + 2) ! (k - m - 4) !}  \left. \frac{\mathd^m}{\mathd x^m} ((\Psi'
    (x))^{k - m - 4} (\Psi'' (x))^2) \right|_{x = 0} . \label{intro-sagan}
  \end{equation}
\end{theorem}

As visible from the result, new effects appear in the analysis of the third asymptotic term. Despite the fact that condition \eqref{intro-audikos} is similar to condition \eqref{intro-ola}, the structure of the answer is more complicated than in Theorem \ref{th:intro-Einaiiagapi}. As explained in Lemma \ref{lem:mu-doublePrime}, the term \eqref{intro-Ss} comes from the second order infinitesimal freeness. However, the term \eqref{intro-sagan} is more specific to the scaling under consideration, and does not have an analog for the first order correction. It would be interesting to establish a connection of explicit formulas \eqref{intro-Ss}, \eqref{intro-sagan} with the description for higher order free cumulants given in \cite{BCGLS}; we do not address this question.

The main result of Section \ref{sec:higher-order} is the following.

\begin{theorem} (Theorem \ref{fredi})
  \label{intro-fredi}
  Let $n$ be a fixed integer. Let $A$ be a random Hermitian matrix of size $N$ and assume
  that for every finite $r$ one has
  \begin{equation}
    \lim_{N \rightarrow \infty} N^{n \varepsilon} \left( \frac{1}{N} \log
    \mathbb{E}[H \text{} C (x_1, \ldots, x_r, 0^{N - r} ; \lambda_1 (A),
    \ldots, \lambda_n (A))] - \sum_{j = 0}^{n - 1} \sum_{i = 0}^r
    \frac{1}{N^{j \varepsilon}} \Psi_j (x_i) \right) = \sum_{i = 1}^r \Psi_n
    (x_i), \label{intro-100}
  \end{equation}
  where $\Psi_0, \ldots, \Psi_n$ are smooth functions in a complex neighborhood of
  $0$, the above convergence is uniform in a complex neighborhood of $0^r$, and $0 <
  \varepsilon < n^{- 1}$. Then the $1 / N^{n \varepsilon}$ correction of the
  average empirical distribution of $N^{- 1} A$ is given by
  \begin{equation}
    \lim_{N \rightarrow \infty} N^{n \varepsilon} \left( \frac{1}{N^{k + 1}}
    \mathbb{E}[\tmop{Tr} (A^k)] - \sum_{j = 0}^{n - 1} \frac{1}{N^{j
    \varepsilon}} \tmmathbf{\mu}^{(j)}_k \right) =\tmmathbf{\mu}_k^{(n)},
    \label{intro-PORDOKOFTIS}
  \end{equation}
  where for every $j \in \{0, \ldots, n\}$
  \[ \tmmathbf{\mu}^{(j)}_k \assign \sum_{b_1 + 2 b_2 + \cdots + j \text{} b_j
     = j} \sum_{m = 0}^{k - b_1 - \cdots - b_j} \frac{k!}{m! (m + 1) ! (k - m
     - b_1 - \cdots - b_j) !b_1 ! \ldots b_j !} \text{{\hspace{8em}}} \]
  \begin{equation}
    \text{{\hspace{14em}}} \times \left. \frac{\mathd^m}{\mathd x^m} ((\Psi'_0
    (x))^{k - m - b_1 - \cdots - b_j} (\Psi'_1 (x))^{b_1} \ldots (\Psi'_j
    (x))^{b_j}) \right|_{x = 0} . \label{intro-tavradiamouleipis}
  \end{equation}
\end{theorem}

The introduction of parameter $\epsilon$ and the limit regime \eqref{intro-100} simplify the structure of the answer. 
We connect formula \eqref{intro-tavradiamouleipis} with higher order infinitesimal freeness introduced in \cite{F}; nevertheless, the explicit formula \eqref{intro-tavradiamouleipis} seems to be new. We would like to emphasize that the main feature of Theorem \ref{intro-fredi}, as well as other key results of the current paper, is that condition \eqref{intro-100} is of general nature --- we do not assume the explicit form of the Harish-Chandra transform, only the asymptotic information about it. 

In Sections \ref{sec:second-order} and \ref{sec:higher-order} we also provide several examples with detailed calculations. In particular, in Example \ref{ex:BBP-highOrder} we study a version of a higher order BBP phase transition.


\subsection*{Acknowledgments}
We are grateful to Alexei Borodin,Vadim Gorin and Leonid Petrov for valuable comments. We used the code by Sunil Chhita in order to produce Figure \ref{fig:til}. Both authors were partially supported by the European Research Council (ERC), Grant Agreement No. 101041499.

\section{Preliminaries}
\label{sec:prelim}

Voiculescu's free probability theory (\cite{V1}, \cite{V2}, \cite{NS}) provides a natural framework to study
families of random matrices as their size goes to infinity. The reason is that
many classes of large random matrices behave like freely independent random
variables. A typical situation where this phenomenon occurs is the
summation of random matrices. For example, consider the case of two
independent $N \times N$ Hermitian random matrices $A_1, A_2$ with
deterministic eigenvalues $\{\lambda_1 (A_1) \leq \cdots \leq \lambda_N
(A_1)\}$, $\{\lambda_1 (A_2) \leq \cdots \leq \lambda_N (A_2)\}$
and eigenvectors chosen uniformly (=Haar distributed) at random. Under convergence assumptions for the empirical distributions $N^{- 1} \sum_{i = 1}^N \delta \left( \lambda_i (A_j) \right)$,
$j = 1, 2$, one can deduce the convergence for the (random) empirical distribution
of $A_1 + A_2$. The limit of $A_1 + A_2$ (meaning the limit of its empirical
distribution) can be uniquely determined in terms of the limits of $A_1, A_2$
by quantities called free cumulants. The summation $A_1+A_2$ translates into the 
summation of the free cumulants that correspond to the limits of $A_1$ and $A_2$.

Infinitesimal free probability theory goes a step further and studies not only
the limit of $A_1 + A_2$ but also its $1 / N$ \textbf{correction}, that is, the second leading term in the asymptotic expansion of the expectation of observables of $A_1 + A_2$. This question attracted additional attention due to its connection with finite-rank perturbations of
random matrices \cite{BBP}, \cite{BS}, \cite{Sh}. In this context,
let $A_3$ be a $N \times N$ finite-rank (i.e., its rank does not grow with $N$), Hermitian and deterministic matrix.
Its empirical distribution converges to $\delta \left( 0 \right)$ and its $1 / N$ correction
is given by
\[ \sum_{\lambda \in \text{ non-zero eigenvalues of } A_3} (\delta \left( \lambda \right) -
   \delta \left( 0 \right) ) . \]
Assuming the convergence for the empirical measure of $A_1$, the leading order limits of $A_1$ and $A_1 + A_3$ are
equal to each other. However, the existence of $A_3$ in $A_1 + A_3$ affects the $1 / N$
correction. This can be obtained from the fact that free independence of
$A_1, A_3$ does not only give a rule for the limit of $A_1 + A_3$, but also
for its $1 / N$ correction. The description of this rule can be presented with the use of
quantities that generalize the notion of free cumulants; they are known as
infinitesimal free cumulants. They were introduced and studied in \cite{BS}, \cite{FN}.

In the current paper we also study arbitrary order corrections to the limit of
the empirical distribution of certain random matrix models. We make explicit
calculations for higher order corrections on different scales. Specifically, we
compute the $n$-th order correction on the scale $1 / N^{\varepsilon}$, with $0 <
\varepsilon \leq 1$, for a $N \times N$ Hermitian random matrix $A$, by
obtaining an asymptotic expansion
\begin{equation}
  \frac{1}{N} \mathbb{E} [\tmop{Tr} (A^k)] = \sum_{i = 0}^n M_{i, k, N}
  \frac{1}{N^{i \varepsilon}} + \omicron (N^{- n \varepsilon}) \text{, \quad
  for every } k \in \mathbb{N}, \label{1}
\end{equation}
where $(M_{i, k, N})_{N \in \mathbb{N}}$ are convergent sequences of real
numbers. We prove the existence of an expansion of the form (\ref{1}) under general assumptions on the asymptotic behavior of Harish-Chandra transforms, see Section \ref{sec:intro}.


Let us now turn to formal definitions. We begin by recalling the basic notions of non-commutative/free probability and some of its basic tools.

\begin{definition}[Non-commutative probability]
\label{def:ncP}
  A non-commutative probability space consists of a unital algebra
  $\mathcal{A}$ over $\mathbb{C}$ and a linear functional $\varphi :
  \mathcal{A} \rightarrow \mathbb{C}$, such that $\varphi (1_{\mathcal{A}}) =
  1$. The elements of $\mathcal{A}$ are called non-commutative random
  variables. The non-commutative distribution of $a \in \mathcal{A}$ is the
  linear functional $\mu_a : \mathbb{C} \langle \mathbf{x} \rangle \rightarrow
  \mathbb{C}$, defined by
  \[ \mu_a (P) \assign \varphi (P (a)) \text{, \quad for every ploynomial } P \in
     \mathbb{C} \langle \mathbf{x} \rangle . \]
  For $(a_1, \ldots, a_n) \in \mathcal{A}^n$ the values $\varphi (a_1 \ldots
  a_n)$ are called moments.
\end{definition}

The framework described in Definition \ref{def:ncP} can be seen as a
generalization of the classical probabilistic setting. For this, one
considers the algebra of (integrable) random variables as the fundamental object, instead of
a probability space $(\Omega, \mathcal{F}, \mathbb{P})$. In that case, the
expectation $\mathbb{E}$ will play the role of the linear functional. In the
current paper, we are interested in the following non-commutative probability
space.
\begin{example}
  Let $\Mu_N (L^{\infty -} (\Omega, \mathcal{F}, \mathbb{P}))$ be the algebra
  of $N \times N$ matrices, whose entries are complex-valued random variables
  such that all their moments exist. The tuple $(\Mu_N (L^{\infty -} (\Omega,
  \mathcal{F}, \mathbb{P})), N^{- 1}  \mathbb{E}[\tmop{Tr} (\cdot)])$ is a
  non-commutative probability space.
\end{example}

\textbf{Free probability} deals with the notion of \textbf{free independence} in non-commutative spaces. A combinatorial description of free independence was developed in
\cite{Sp}; it is based on cumulant functionals, known as free cumulants. A partition $\pi$ of $\{1, \ldots, k\}$ is called non-crossing if for all
blocks $V = \{v_1 < \cdots < v_n \}$ and $W = \{w_1 < \cdots < w_m \}$ we have
$v_r < w_1 < v_{r + 1}$ for some $r = 1, \ldots, n - 1$ if and only if
$v_r < w_m < v_{r + 1}$. We denote the set of non-crossing partitions of $\{1,
\ldots, k\}$ by $\tmop{NC} (k)$. Non-crossing partitions can be visualized in
the following way: If we connect with ''bridges'' the points of $\{1, \ldots, k\}$
that belong to the same block of $\pi$, then $\pi$ is non-crossing if and only
if these bridges do not cross.


\begin{definition}[Free cumulants]
  Let $(\mathcal{A}, \varphi)$ be a non-commutative probability space. The
  multilinear functionals $(\kappa_n : \mathcal{A}^n \rightarrow
  \mathbb{C})_{n \in \mathbb{N}}$ uniquely determined by the relation
  \begin{equation}
    \varphi (a_1 \ldots a_k) = \sum_{\pi \in \tmop{NC} (k)} \prod_{\{i_1 <
    \cdots < i_s \} \in \pi} \kappa_s (a_{i_1}, \ldots, a_{i_s}) \text{, \quad
    for every $k \in \mathbb{N}$ and } (a_1, \ldots, a_k) \in \mathcal{A}^k,
    \label{tavradiamouleipeisKoulii}
  \end{equation}
  are called free cumulants of $(\mathcal{A}, \varphi)$.
\end{definition}

The relation between free independence and free cumulants comes from the fact
that two non-commutative random variables are freely independent if and only
if certain free cumulants vanish (see, e.g., \cite{NS}). Therefore, for freely independent $a, b \in (\mathcal{A}, \varphi)$, one has
\[ \kappa_n (a + b, \ldots, a + b) = \kappa_n (a, \ldots, a) + \kappa_n (b,
   \ldots, b) \text{, \quad for every } n \in \mathbb{N}. \]
The main connection between random matrices and free probability comes from
the fact that free probability provides a natural framework for many classes of random
matrices as their size goes to infinity, see, e.g., \cite{NS}, \cite{AGZ}. In particular, it describes large Hermitian random matrices with
fixed spectrum and eigenvectors chosen uniformly at random. Two independent
such matrices $A_1, A_2$, under asymptotic conditions for their spectra, become
asymptotically freely independent. In more detail, one has the following result.


\begin{theorem}[Voiculescu]
  \label{Voiculescu}Let $A_1 = U_1 \text{} \tmop{diag} (a^{(1)}_1, \ldots,
  a^{(1)}_N) U_1^{\ast}$ and $A_2 = U_2 \tmop{diag} (a^{(2)}_1, \ldots,
  a^{(2)}_N) U_2^{\ast}$ be Hermitian random matrices of size $N$, where $U_1,
  U_2$ are chosen independently at random with respect to the Haar measure
  on the unitary group $U (N)$ and $\{ a_i^{(j)} \}_{i, j}$ are fixed and such
  that for $j = 1, 2$ the empirical distribution of $A_j$ converges weakly to
  a probability measure $\tmmathbf{\mu}_j$. Then the random empirical
  distribution of $A_1 + A_2$ converges weakly, in probability to a
  deterministic probability measure $\tmmathbf{\mu}_1 \boxplus
  \tmmathbf{\mu}_2$.
\end{theorem}

The probability measure $\tmmathbf{\mu}_1 \boxplus \tmmathbf{\mu}_2$ is known
as the free convolution of $\tmmathbf{\mu}_1$ and $\tmmathbf{\mu}_2$. It can be described via
corresponding free cumulant sequences $(\kappa_n (\tmmathbf{\mu}_1))_{n \in
\mathbb{N}}, (\kappa_n (\tmmathbf{\mu}_2))_{n \in \mathbb{N}}$, which are uniquely
determined by the relations
\[ \int_{\mathbb{R}} t^k \tmmathbf{\mu}_j (d \text{} t) = \sum_{\pi \in
   \tmop{NC} (k)} \prod_{V \in \pi} \kappa_{|V|} (\tmmathbf{\mu}_j) \text{,
   \quad for every } k \in \mathbb{N} \text{ and } j = 1, 2. \]
Namely, one has:
\[ \kappa_n (\tmmathbf{\mu}_1 \boxplus \tmmathbf{\mu}_2) = \kappa_n
   (\tmmathbf{\mu}_1) + \kappa_n (\tmmathbf{\mu}_2) \text{, \quad for every }
   n \in \mathbb{N}. \]
In the next sections we are interested in the \textbf{corrections} to this limit. We
want to understand these corrections through quantities that are
generalizations of free cumulants and encode the notion of higher order
infinitesimal freeness. The abstract setting is the following.

\begin{definition}
  An infinitesimal non-commutative probability space of order $\nu$ is a $(\nu
  + 1)$-tuple $(\mathcal{A}, \varphi^{(0)}, \ldots, \varphi^{(\nu)})$, where
  $(\mathcal{A}, \varphi^{(0)})$ is a non-commutative probability space and
  $\varphi^{(i)} : \mathcal{A} \rightarrow \mathbb{C}$ are linear functionals
  such that $\varphi^{(i)} (1_{\mathcal{A}}) = 0$, for every $i = 1, \ldots,
  \nu$. The infinitesimal non-commutative distribution of order $n$ of $a \in
  \mathcal{A}$ is the $(n + 1)$-tuple $(\mu_a^{(i)} : \mathbb{C} \langle
  \mathbf{x} \rangle \rightarrow \mathbb{C})_{i = 0}^n$ of linear functionals,
  defined by
  \[ \mu_a^{(i)} (P) = \varphi^{(i)} (P (a)) \text{, \quad for every } P \in
     \mathbb{C} \langle \mathbf{x} \rangle \text{ and } i = 0, \ldots, \nu .
  \]
\end{definition}

We are interested in infinitesimal non-commutative probability spaces
$(\mathcal{A}, \varphi^{(0)}, \ldots, \varphi^{(\nu)})$ that emerge as the
infinitesimal limit of a family $(\mathcal{A}, \varphi_t)_{t \in \mathbb{R}}$ of
non-commutative probability spaces. This means that for $a \in \mathcal{A}$, $\varphi^{(0)}
(a), \ldots, \varphi^{(\nu)} (a)$ are uniquely determined by
\begin{equation}
  \frac{\varphi^{(\nu)} (a)}{\nu !} = \lim_{t \rightarrow 0} \frac{1}{t^{\nu}}
  \left( \varphi_t (a) - \sum_{i = 0}^{\nu - 1} \varphi^{(i)} (a)
  \frac{t^i}{i!} \right) . \label{getingolder}
\end{equation}
Thus, we have $\lim_{t \rightarrow 0} \varphi_t (a) = \varphi^{(0)} (a)$,
for every $a \in \mathcal{A}$ and for $i = 1, \ldots, \nu$ the ``derivative''
$\varphi^{(i)}$ plays the role of the $i$-th order correction to the above
limit. For $\nu = 1$ this is the framework of the infinitesimal
free probability, introduced in \cite{BS}. They
introduced a notion of infinitesimal freeness which provides a rule for
computing moments of random variables in $(\mathcal{A}, \varphi^{(0)},
\varphi^{(1)})$, assuming that they are freely independent in $(\mathcal{A},
\varphi_t)$ for every $t \in \mathbb{R}$. In the same direction, the notion of infinitesimal free cumulants was introduced in \cite{FN}. The moments of random variables in $(\mathcal{A}, \varphi^{(0)}, \varphi^{(1)})$ can be written in terms of these cumulants in a similar to (\ref{tavradiamouleipeisKoulii}) way. In more detail,
the infinitesimal free cumulants $(\kappa^{(0)}_n, \kappa^{(1)}_n :
\mathcal{A}^n \rightarrow \mathbb{C})_{n \in \mathbb{N}}$ are defined in the
following way: $(\kappa_n^{(0)})_{n \in \mathbb{N}}$ are the free cumulants of
$(\mathcal{A}, \varphi^{(0)})$ and $(\kappa_n^{(1)})_{n \in \mathbb{N}}$ are
uniquely determined by
\begin{equation}
  \varphi^{(1)} (a_1 \ldots a_k) = \sum_{\pi \in \tmop{NC} (k)}  \sum_{V =
  \{i_1 < \cdots < i_s \} \in \pi} \kappa_s^{(1)} (a_{i_1}, \ldots, a_{i_s})
  \prod_{\underset{W \neq V}{W = \{j_1 < \cdots < j_l \} \in \pi}}
  \kappa_l^{(0)} (a_{j_1}, \ldots, a_{j_l}), \label{thelongestway}
\end{equation}
for every $k \in \mathbb{N}$ and $(a_1, \ldots, a_k) \in \mathcal{A}^k$. Formula \eqref{thelongestway} can be thought of as that there is an (informal) differentiation procedure to pass from $\varphi^{(0)}$ to $\varphi^{(1)}$, starting from the
moment-cumulant relations (\ref{tavradiamouleipeisKoulii}). Namely, if
$\kappa_n^{(1)} (a_1, \ldots, a_n)$ is considered as the derivative of
$\kappa_n^{(0)} (a_1, \ldots, a_n)$ and $\varphi^{(1)} (a_1 \ldots a_n)$ as
the derivative of $\varphi^{(0)} (a_1 \ldots a_n)$, applying derivatives in
(\ref{tavradiamouleipeisKoulii}) and using Leibniz rule, we get
(\ref{thelongestway}). In a similar vein, the infinitesimal free cumulants
$(\kappa_n^{(0)})_{n \in \mathbb{N}}$ and $(\kappa_n^{(1)})_{n \in
\mathbb{N}}$ of $(\mathcal{A}, \varphi^{(0)}, \varphi^{(1)})$ satisfy the
relations
\[ \kappa_n^{(0)} (a_1, \ldots, a_n) = \lim_{t \rightarrow 0} \kappa_n^{(t)}
   (a_1, \ldots, a_n) \text{\quad and\quad} \kappa_n^{(1)} (a_1, \ldots, a_n)
   = \left. \frac{\mathd}{\mathd t} \kappa_n^{(t)} (a_1, \ldots, a_n)
   \right|_{t = 0}, \]
where $(\kappa_n^{(t)})_{n \in \mathbb{N}}$ are the free cumulants of
$(\mathcal{A} \comma \varphi_t)$.

\begin{remark}
  Similarly with the free independence, the infinitesimal freeness can be
  characterized by infinitesimal free cumulants. Two random variables in $(\mathcal{A}, \varphi^{(0)},
  \varphi^{(1)})$ are infinitesimally free if and only if certain
  infinitesimal free cumulants vanish (see \cite{FN}). As a corollary, infinitesimally free variables $a, b \in (\mathcal{A},
  \varphi^{(0)}, \varphi^{(1)})$ satisfy
  \[ \kappa_n^{(i)} (a + b, \ldots, a + b) = \kappa_n^{(i)} (a, \ldots, a) +
     \kappa_n^{(i)} (b, \ldots, b), \text{\quad for every } n \in \mathbb{N}
     \text{ and } i = 0, 1. \]
\end{remark}

Motivated by the construction of infinitesimal free cumulants for
$(\mathcal{A}, \varphi^{(0)}, \varphi^{(1)})$, one can extend the notion of
infinitesimal free cumulants for $(\mathcal{A}, \varphi^{(0)}, \varphi^{(1)},
\ldots, \varphi^{(\nu)})$. Since $\varphi^{(i)}$ plays the role of the $i$-th
derivative of $\varphi^{(0)}$, the infinitesimal free cumulants of higher
order $\kappa_n^{(i)}$ can be seen as the $i$-th derivative of
$\kappa_n^{(0)}$. In that way, starting from the relation
(\ref{tavradiamouleipeisKoulii}) between $\varphi^{(0)}$, $(\kappa_n^{(0)})_{n
\in \mathbb{N}}$ and differentiating step by step, we can determine
$(\kappa_n^{(i)})_{n \in \mathbb{N}}$ with the use of $\varphi^{(i)},
(\kappa_n^{(0)})_{n \in \mathbb{N}}, \ldots, (\kappa_n^{(i - 1)})_{n \in
\mathbb{N}}$. In order to write the exact formula that describes this
relation, we introduce some notation.

\begin{notation}
  Let $(\psi_n : \mathcal{A}^n \rightarrow \mathbb{C})_{n \in \mathbb{N}}$ be
  a family of multilinear functionals and $\pi$ be a partition of $\{1,
  \ldots, k\}$. Given a block $V = \{i_1 < \cdots < i_s \}$ of $\pi$ and
  $(a_1, \ldots, a_k) \in \mathcal{A}^k$, we define
  \[ \psi_{|V|} (a_1, \ldots, a_k |V) \assign \psi_s (a_{i_1}, \ldots,
     a_{i_s}) . \]
\end{notation}

\begin{definition}
\label{def:inf-free-mom-cum}
  \label{Xisi}Let $(\mathcal{A}, \varphi^{(0)}, \varphi^{(1)}, \ldots,
  \varphi^{(\nu)})$ be an infinitesimal non-commutative probability space of
  order $\nu$. The infinitesimal free cumulants of order $\nu$ are
  multilinear maps $(\kappa_n^{(0)}, \ldots, \kappa_n^{(\nu)} : \mathcal{A}^n
  \rightarrow \mathbb{C})$ uniquely determined by the following relation: For
  every $k \in \mathbb{N}$, $i = 0, \ldots, \nu$ and $(a_1, \ldots, a_k) \in
  \mathcal{A}^k$,
  \begin{equation}
    \varphi^{(i)} (a_1 \ldots a_k) = \sum_{\pi = \{V_1, \ldots, V_p \} \in
    \tmop{NC} (k)} \sum_{\lambda_1 + \cdots + \lambda_p = i}
    \frac{i!}{\lambda_1 ! \ldots \lambda_p !} \prod_{i = 1}^p \kappa_{|V_i
    |}^{(\lambda_i)} (a_1, \ldots, a_k |V_i) . \label{Goldtouch}
  \end{equation}
\end{definition}

Note that Definition \ref{Xisi} is consistent with our description of the
infinitesimal free cumulants of higher order since the latter sum in the right hand side of
(\ref{Goldtouch}) plays the role of the $i$-th derivative of $\prod_{V \in \pi}
\kappa^{(0)}_{|V|} (a_1, \ldots, a_k |V)$. Formula \eqref{Goldtouch} for the
infinitesimal free cumulants of higher order was introduced in \cite{F}.
Analogously to the connection of the (infinitesimal) freeness (of order $1$) and
(infinitesimal) free cumulants (of order $1$), in \cite{F} the notion of higher
order freeness is introduced as a rule for vanishing certain higher order
infinitesimal free cumulants. This is consistent with the cases $\nu = 0, 1$, and
for two infinitesimally free of order $\nu$ random variables $a, b \in
(\mathcal{A}, \varphi^{(0)}, \ldots, \varphi^{(\nu)})$ we have
\[ \kappa_n^{(i)} (a + b, \ldots, a + b) = \kappa_n^{(i)} (a, \ldots, a) +
   \kappa_n^{(i)} (b, \ldots, b) \text{, \quad for every } n \in \mathbb{N}
   \text{ and } i = 0, \ldots, \nu . \]
In the present paper, our interest is in infinitesimal non-commutative probability spaces of higher order that emerge as the infinitesimal limit of the non-commutative
distribution of random matrix models. In more detail, given a Hermitian random
matrix $A$ of size $N$, the linear functional that will play the role of
$\varphi_t$ is
\[ P \in \mathbb{C} \langle \mathbf{x} \rangle \mapsto \frac{1}{N} \mathbb{E}
   [\tmop{Tr} \text{} P (A)], \]
where ``$t = \frac{1}{N}$''. In this setting, in order to guarantee the
existence of the infinitesimal limit of the form (\ref{getingolder}), we will
focus on particular random matrix models. The simplest example is a GUE
matrix, i.e. a Hermitian random matrix $A = (a_{i, j})_{i, j = 1}^N$ such that
$(a_{i, j})_{i \leq j}$ are independent centered complex Gaussian variables,
with independent real and imaginary parts and covariances $\mathbb{E} [a_{i,
j} a_{k, l}] = N^{- 1} \delta_{i, l} \delta_{j, k}$, for every $1 \leq i, j, k,
l \leq N$. 
It is well known (see, e.g., \cite{AGZ}) that for such a matrix, for
every $k \in \mathbb{N}$ one has an asymptotic expansion
\begin{equation}
\frac{1}{N} \mathbb{E} [\tmop{Tr} (A^k)] = \sum_{n \geq 0} \frac{1}{N^{2 n}}
M_{n, k}, \label{topology}
\end{equation}
where $M_{n, k}$ are real numbers that do not depend on $N$. Relation
(\ref{topology}) is called the topological expansion because the numbers $M_{n,
k}$ are related to the enumeration of maps.

In the following sections we study random matrix models that have
expansions similar to (\ref{topology}). Let $A$ be a Hermitian random matrix of
size $N$ with real eigenvalues $\lambda_1 (A), \ldots, \lambda_N (A)$. Our
approach for computing the moments of the probability measure $N^{- 1}
\mathbb{E} \left[ \sum_{i = 1}^N \delta \left( \lambda_i (A) \right) \right]$ is based on a
differentiation procedure of the characteristic/moment generating function of
$A$. This is analogous to the fact that via differentiating the
characteristic/moment generating function of a random variable one can get its
moments. Therefore, we focus on classes of random matrices whose
characteristic function can be controlled to some extent.

One class of such matrices is formed by ergodic unitarily invariant matrices. 
In the current paper they will be the main source of our examples.
Ergodic random matrices were classified and studied in \cite{OV} (see also \cite{P}). We recall that given a
probability measure $\tmmathbf{M}$ on the space of all infinite Hermitian
matrices $H$, its characteristic function is defined via
\[ f_{\tmmathbf{M}} (A) := \int_H \exp (\mathi \tmop{Tr} (A \text{} B))
   \tmmathbf{M} (d \text{} B) \text{, \quad for every } A \in H (\infty), \]
where $H (\infty) \subset H$ is the space of infinite Hermitian matrices with
finitely many non-zero entries. Let $U (\infty) \subseteq H$ be the group of
inifinite unitary matrices $U = (u_{i, j})$ such that $u_{i, j} = \delta_{i,
j}$ when $i + j$ is large enough and $D (\infty) \subseteq H (\infty)$ be the
subspace of diagonal matrices in $H (\infty)$. For a $U (\infty)$-invariant
Borel probability measure $\tmmathbf{M}$, the value $f_{\tmmathbf{M}} (A)$
depends only from the spectrum of $A \in H (\infty)$, since any matrix in $H
(\infty)$ can be diagonalized under the action of $U (\infty)$. Then the
Multiplicativity Theorem (see, e.g., \cite[Theorem 2.1]{OV}) states that $\tmmathbf{M}$ is
ergodic if and only if for every $k \in \mathbb{N}$ the symmetric function
$(a_1, \ldots, a_k) \mapsto f_{\tmmathbf{M}} (\tmop{diag} (a_1, \ldots, a_k,
0, \ldots))$ is multiplicative, in the sense that there exists a one-variable
function $F_{\tmmathbf{M}}$, such that
\[ f_{\tmmathbf{M}} (\tmop{diag} (a_1, \ldots, a_k, 0, \ldots)) = \prod_{i =
   1}^k F_{\tmmathbf{M}} (a_i) \text{, \quad for every } k \in \mathbb{N}
   \text{ and } a_1, \ldots, a_k \in \mathbb{R}. \]
The function $F_{\tmmathbf{M}}$ is determined by $F_{\tmmathbf{M}} (a) =
f_{\tmmathbf{M}} (\tmop{diag} (a, 0, \ldots))$, for every $a \in \mathbb{R}$.
Let $\mathcal{F}$ denote the class of all these functions $F_{\tmmathbf{M}}$.
The description of $\mathcal{F}$ leads to the classification of the ergodic
measures $\tmmathbf{M}$.

\begin{theorem} (\cite{OV}, \cite{P})
  The class $\mathcal{F}$ consists of all the functions $F_{\gamma_1, \gamma_2,
  x}$ of the form
  \begin{equation}
  \label{eq:OV-classification}
  F_{\gamma_1, \gamma_2, x} (a) = \exp \left( \mathi \gamma_1 a -
     \frac{\gamma_2}{2} a^2 \right) \prod_{n = 1}^{\infty} \frac{\exp (-
     \mathi x_n a)}{1 - \mathi x_n a},
  \end{equation}
  where $(x_n) \in \mathbb{N}$ is a sequence of real numbers such that
  $\sum_{n \geq 1} x_n^2 < \infty$ and $\gamma_1 \in \mathbb{R}$, $\gamma_2
  \geq 0$.
\end{theorem}

Let us recall that for two $N \times N$ Hermitian matrices
$A, B$ with eigenvalues $\lambda_1 (A) \leq \cdots \leq \lambda_N (A)$ and
$\lambda_1 (B) \leq \cdots \leq \lambda_N (B)$, the Harish-Chandra integral
(also known as Itzykson-Zuber integral) is defined by
\begin{equation}
  H \text{} C (\lambda_1 (A), \ldots, \lambda_N (A) ; \lambda_1 (B), \ldots,
  \lambda_N (B)) \assign \int_{U (N)} \exp (\tmop{Tr} (A \text{} U \text{} B
  \text{} U^{\ast})) \tmmathbf{m}_N (d \text{} U), \label{Haarmeasure}
\end{equation}
where $\tmmathbf{m}_N$ denotes the Haar measure on the unitary group $U (N)$ (see \cite{HC}).
Note that the right hand side of (\ref{Haarmeasure}) depends only on
${\{\lambda_i (A), \lambda_i (B)\}_{i = 1}^N} $. The above integral can be
computed explicitly and it is equal to
\begin{equation}
  c_N \frac{\det (\exp (\lambda_i (A) \lambda_j (B)))_{i, j = 1}^N}{\prod_{1
  \leq i < j \leq N} (\lambda_i (A) - \lambda_j (B)) \prod_{1 \leq i < j \leq
  N} (\lambda_i (B) - \lambda_j (B))}, \label{Jmn}
\end{equation}
where $c_N = \prod_{i = 1}^{N - 1} i!$.

\section{Free probability scaling}
\label{sec:previous-result}

In this section we essentially recall the proof of \cite[Theorem 5.1]{BG}, see Theorem \ref{SouvlakiA} below. We do certain technical improvements along the way, which will allow us to use this section as a reference point for the proofs in the rest of the paper. We also present it here in the language of random matrices and Harish-Chandra transform, rather than a somewhat more general setup of discrete particle systems and Schur functions of \cite[Theorem 5.1]{BG}, which we address in Section \ref{sec:Schur-free} below. 


For every $k \in \mathbb{N}$ let us introduce a
differential operator acting on smooth functions $f$ of $N$ variables $x_1, x_2, \dots, x_n$:
\begin{equation}
  \mathcal{D}_k (f) \assign \left( \prod_{i < j} (x_i - x_j) \right)^{- 1}
  \sum_{i = 1}^N \partial_i^k \left( \prod_{i < j} (x_i - x_j) \cdot f \right)
  . \label{gamisi}
\end{equation}
\begin{proposition}
\label{prop:1}
  For the function $f_{\underset{\lambda}{\rightarrow}} (x_1, \ldots, x_N)
  \assign H \text{} C (x_1, \ldots, x_N ; \lambda_1, \ldots, \lambda_N)$ we
  have
  \begin{equation}
    \mathcal{D}_k \left( f_{\underset{\lambda}{\rightarrow}} \right) (x_1,
    \ldots, x_N) = \sum_{i = 1}^N \lambda_i^k
    f_{\underset{\lambda}{\rightarrow}} (x_1, \ldots, x_N) .
    \label{pordokoftis}
  \end{equation}
\end{proposition}

\begin{proof}
  By the relation \eqref{Jmn} and the Leibniz rule we have
  \begin{eqnarray*}
    \mathcal{D}_k \left( f_{\underset{\lambda}{\rightarrow}} \right) & = &
    \left( \prod_{i < j} (x_i - x_j) (\lambda_i - \lambda_j) \right)^{- 1}
    \sum_{\pi \in S_N} \tmop{sgn} (\pi) \left( \sum_{i = 1}^N \lambda_{\pi
    (i)}^k \right) \prod_{i = 1}^N \exp (x_i \lambda_{\pi (i)})\\
    & = & \left( \sum_{i = 1}^N \lambda_i^k \right) \left( \prod_{i < j} (x_i
    - x_j) (\lambda_i - \lambda_j) \right)^{- 1} \det (\exp (x_i
    \lambda_j))_{i, j = 1}^N
    = \sum_{i = 1}^N \lambda_i^k f_{\underset{\lambda}{\rightarrow}} .
  \end{eqnarray*}
  
\end{proof}

Before stating the main theorem of this section, we give
several lemmas which help us to understand how the differential operator
$\mathcal{D}_k$ acts on smooth functions.

\begin{lemma}
\label{lem:2}
  For a smooth function $f$ of $N$ variables we have
  \begin{equation}
    \mathcal{D}_k (f) = \sum_{m = 0}^k \sum_{\underset{l_i \neq l_j \text{ for
    } i \neq j}{l_0, \ldots, l_m = 1}}^N \binom{k}{m} \frac{\partial_{l_0}^{k
    - m} f}{(x_{l_0} - x_{l_1}) \ldots (x_{l_0} - x_{l_m})} .
    \label{alaniariko}
  \end{equation}
\end{lemma}

\begin{proof}
  For $l \in \{1, \ldots, N\}$ and $k \in \mathbb{N}$, by Leibniz rule, we
  have
  \begin{equation}
    \partial_l^k \left( \prod_{i < j} (x_i - x_j) \cdot f \right) = (- 1)^{l -
    1} \sum_{k_1 + \cdots + k_N = k} \frac{k!}{k_1 ! \ldots k_N !} 
    \prod_{\underset{i \neq l}{i = 1}}^N \partial_l^{k_i} (x_l - x_i)
    \partial_l^{k_l} f, \label{mplimmplom}
  \end{equation}
  where in the product in the left hand side we omit terms of $\prod_{i <
  j} (x_i - x_j)$ that do not depend on $x_l$. Consider the case, where $k_l
  = k - m$ and $k_{l_1} = \cdots = k_{l_m} = 1$, for some $l_1, \ldots, l_m \in
  \{1, \ldots, N\} \backslash \{l\}$, with $l_i \neq l_j$, for all $i \neq j$.
  Then, if we divide the corresponding summand of the left hand side of
  (\ref{mplimmplom}) to $\prod_{i < j} (x_i - x_j)$, we get
  \[ \frac{k!}{(k - m) !}  \frac{\partial_l^{k - m} f}{(x_l - x_{l_1}) \ldots
     (x_l - x_{l_m})}. \]
  Considering all the different choices of variables to differentiate, we
  obtain that the left hand side of (\ref{mplimmplom}) divided by $\prod_{i < j}
  (x_i - x_j)$ is equal to
  \[ \sum_{\underset{l_i \neq l_j \text{ for } i \neq j}{l_1, \ldots, l_m \in
     \{1, \ldots, N\} \backslash \{l\}}} \binom{k}{m} \frac{\partial_l^{k - m}
     f}{(x_l - x_{l_1}) \ldots (x_l - x_{l_m})}, \]
  where in the above sum the binomial coefficient appears because for a fixed
  $m$-tuple $(l_1, \ldots, l_m)$ the summand $\frac{k!}{(k - m) !} 
  \frac{\partial_l^{k - m} f}{(x_l - x_{l_1}) \ldots (x_l - x_{l_m})}$ appears
  $m!$ times. If we take the sum with respect to $l = 1, \ldots, N$ and $m =
  1, \ldots, k$, we arrive at the claim.
\end{proof}

The next lemma is a standard result which will help us to evaluate the differential operator at zero, see e.g., \cite[Lemma 5.5]{BG}.

\begin{lemma}
  Let $n \geq 2$ and a function $f$ smooth in a neighborhood of $0$. Then
  \begin{multline}
    \lim_{x_1, \ldots, x_n \rightarrow 0} \left( \frac{f (x_1)}{(x_1 - x_2)
    \ldots (x_1 - x_n)} + \frac{f (x_2)}{(x_2 - x_1) (x_2 - x_3) \ldots (x_2 -
    x_n)} + \cdots + \frac{f (x_n)}{(x_n - x_1) \ldots (x_n - x_{n - 1})}
    \right) \\ = \frac{f^{(n - 1)} (0)}{(n - 1) !}. \label{malakas}
  \end{multline}
\end{lemma}
The next lemma is \cite[Lemma 5.4]{BG}.

\begin{lemma}
  \label{SnikfeatFy}Let $n$ be a positive integer and $P \subseteq \{(a, b)
  \in \mathbb{N}^2 \of 1 \leq a < b \leq n\}$. Moreover, let $f (z_1, \ldots,
  z_n)$ be a function and consider its symmetrization with respect to $P$,
  \[ f_P (z_1, \ldots, z_n) \assign \frac{1}{n!} \sum_{\pi \in S (n)} \frac{f
     (z_{\pi (1)}, \ldots, z_{\pi (n)})}{\prod_{(a, b) \in P} (z_{\pi (a)} -
     z_{\pi (b)})} . \]
  Then, the following holds:
  \begin{enumerate}
    {\item If $f$ is an analytic function in a neighborhood of $0^n$,
    then $f_P$ is also an analytic function in a neighborhood of $0^n$.}{\item
    If $(f^{(m)} (z_1, \ldots, z_n))_{m \in \mathbb{N}}$ is a sequence of
    analytic functions converging to zero uniformly in a neighborhood of
    $0^n$, then so is the sequence $f_P^{(m)}$.}
  \end{enumerate}
\end{lemma}

Now we state the main theorem of this section. It is a degeneration of \cite[Theorem 5.1]{BG}.  

\begin{theorem}
  \label{SouvlakiA}Let $A=A(N)$ be a random Hermitian matrix of size $N$ and assume
  that for every finite $r$ one has
  \begin{equation}
    \lim_{N \rightarrow \infty} \frac{1}{N} \log \text{} \mathbb{E} [H \text{}
    C (x_1, \ldots, x_r, 0^{N - r} ; \lambda_1 (A), \ldots, \lambda_N (A))] =
    \sum_{i = 1}^r \Psi (x_i), \label{xereis}
  \end{equation}
  where $\Psi$ is a smooth function in a complex neighborhood of $0$ and the above
  convergence is uniform in a complex neighborhood of $0^r$. Then the random measure
  $N^{- 1} \sum_{i = 1}^N \delta \left( N^{- 1} \lambda_i (A) \right)$ converges, as $N
  \rightarrow \infty$, in probability, in the sense of moments to a
  deterministic measure $\tmmathbf{\mu}$ on $\mathbb{R}$ whose moments are
  given by
  \begin{equation}
    \int_{\mathbb{R}} t^k \tmmathbf{\mu} (d \text{} t) = \sum_{m = 0}^k
    \frac{k!}{m! (m + 1) ! (k - m) !}  \left. \frac{\mathd^m}{\mathd x^m}
    ((\Psi' (x))^{k - m}) \right|_{x = 0} . \label{xese}
  \end{equation}
\end{theorem}

\begin{proof}
  For a notation simplicity we write $f_N (x_1, \ldots, x_N) \assign \mathbb{E}
  \left( H \text{} C (x_1, \ldots, x_N ; \lambda_1 (A), \ldots, \lambda_N (A))
  \right)$. First, we will show the convergence in expectation, in the sense that
  for every $k \in \mathbb{N}$
  \begin{equation}
    \lim_{N \rightarrow \infty} \frac{1}{N^{k + 1}} \mathbb{E} \left[ \sum_{i
    = 1}^N \lambda_i^k (A) \right] = \int_{\mathbb{R}} t^k \tmmathbf{\mu} (d
    \text{} t) \label{flwros} .
  \end{equation}
  Using Proposition \ref{prop:1}, we have
  \begin{equation}
    \mathbb{E} \left[ \left( \sum_{i = 1}^N \lambda_i^k (A) \right) H \text{}
    C (x_1, \ldots, x_N ; \lambda_1 (A), \ldots, \lambda_N (A)) \right]
    =\mathcal{D}_k f_N (x_1, \ldots, x_N). \label{pipi}
  \end{equation}
  It would be convenient to use the equality $f_N = \exp \left( N \cdot \frac{1}{N} \log \text{} f_N \right)$.
  For $m = 0, \ldots, k$,
  applying the chain rule for every $i = 1, \ldots, N$, one obtains
  \begin{equation}
    \partial_i^{k - m} f_N = \sum \frac{(k - m) !}{l_1 !1!^{l_1} l_2 !2!^{l_2}
    \ldots l_{k - m} ! (k - m) !^{l_{k - m}}} N^{l_1 + \cdots + l_{k - m}} f_N
    \prod_{j = 1}^{k - m} \left( \partial_i^j \left( \frac{1}{N} \log \text{}
    f_N \right) \right)^{l_j}, \label{kefi}
  \end{equation}
  where we consider the sum with respect to $l_1, \ldots, l_{k - m}$ such that
  $l_1 + 2 l_2 + \cdots + (k - m) l_{k - m} = k - m$. Due to Lemma \ref{SnikfeatFy},
  $\mathcal{D}_k f_N (x_1, \ldots, x_N)$ can be written as a linear
  combination of symmetric terms of the form
  \[ \frac{\left( \partial_{b_0} \left( \frac{1}{N} \log \text{} f_N \right)
     \right)^{l_1} \ldots \left( \partial_{b_0}^{k - m} \left( \frac{1}{N}
     \log \text{} f_N \right) \right)^{l_{k - m}}}{(x_{b_0} - x_{b_1}) \ldots
     (x_{b_0} - x_{b_m})} + \frac{\left( \partial_{b_1} \left( \frac{1}{N}
     \log \text{} f_N \right) \right)^{l_1} \ldots \left( \partial_{b_1}^{k -
     m} \left( \frac{1}{N} \log \text{} f_N \right) \right)^{l_{k -
     m}}}{(x_{b_1} - x_{b_0}) (x_{b_1} - x_{b_2}) \ldots (x_{b_1} - x_{b_m})}
     + \cdots \]
  \begin{equation}
    + \frac{\left( \partial_{b_m} \left( \frac{1}{N} \log \text{} f_N
    \right) \right)^{l_1} \ldots \left( \partial_{b_m}^{k - m} \left(
    \frac{1}{N} \log \text{} f_N \right) \right)^{l_{k - m}}}{(x_{b_m} -
    x_{b_0}) \ldots (x_{b_m} - x_{b_{m - 1}})}, \label{tralala}
  \end{equation}
  where $b_0, \ldots, b_m \in \{1, \ldots, N\}$. In (\ref{pipi}) one needs 
  to send $x_1, \ldots, x_N$ to zero in order to get the moments of the empirical distribution. 
  Let us determine the limiting behavior of terms (\ref{tralala}), as $x_1, \ldots, x_N
 \rightarrow 0$ and $N \rightarrow \infty$. For every $i =
  1, \ldots, N$, we consider the functions
  \[ g_i (x_1, \ldots, x_N) \assign \left[ \left( \partial_i \left(
     \frac{1}{N} \log \text{} f_N \right) \right)^{l_1} \left( \partial_i^2
     \left( \frac{1}{N} \log \text{} f_N \right) \right)^{l_2} \ldots \left(
     \partial_i^{k - m} \left( \frac{1}{N} \log \text{} f_N \right)
     \right)^{l_{k - m}} \right] (x_1, \ldots, x_N), \]
  and $F_i (\varepsilon) \assign g_i (\varepsilon, 2 \varepsilon, \ldots, N
  \varepsilon)$. In (\ref{pipi}) we set $x_i = i \varepsilon$ for every $i =
  1, \ldots, N$ and we will send $\varepsilon$ to $0$. For the functions $F_i$
  we consider the Taylor expansions
  \[ F_i (\varepsilon) = F_i (0) + F_i' (0) \varepsilon + \cdots +
     \frac{\varepsilon^m}{m!} F_i^{(m)} (0) + \varepsilon^m h_i (\varepsilon),
  \]
  where $\lim_{\varepsilon \rightarrow 0} h_i (\varepsilon) = 0$, and we want
  to understand how the summands
  \begin{equation}
    \frac{F_{b_0}^{(n)} (0) \varepsilon^{n - m}}{(b_0 - b_1) \ldots (b_0 -
    b_m)} + \frac{F_{b_1}^{(n)} (0) \varepsilon^{n - m}}{(b_1 - b_0) (b_1 -
    b_2) \ldots (b_1 - b_m)} + \cdots + \frac{F_{b_m}^{(n)} (0) \varepsilon^{n
    - m}}{(b_m - b_0) \ldots (b_m - b_{m - 1})} \label{w}
  \end{equation}
  contribute to (\ref{tralala}). For $n \in \{0, \ldots, m\}$ and $i \in \{1,
  \ldots, N\}$, we have
  \begin{equation}
    F^{(n)}_i (0) = \sum_{i_1, \ldots, i_n = 1}^N i_1 \ldots i_n
    \partial_{i_n} \ldots \partial_{i_1} g_i (0^N). \label{kamputsiki}
  \end{equation}
  Since $f_N$ is symmetric, for two $(n + 1)$-tuples $(\alpha_1, \ldots,
  \alpha_{n + 1})$ and $(\beta_1, \ldots, \beta_{n + 1})$ of elements of $\{1,
  \ldots, N\}$,
  \begin{equation}
    \partial_{\alpha_1} \ldots \partial_{\alpha_n} g_{\alpha_{n + 1}} (0^N) =
    \partial_{\beta_1} \ldots \partial_{\beta_n} g_{\beta_{n + 1}} (0^N),
    \label{seixeles}
  \end{equation}
  where $\alpha_i = \alpha_j \Leftrightarrow \beta_i = \beta_j$, for every $i,
  j$. But, by induction on $r$, for every $k_1, \ldots, k_r \in \mathbb{N}$, a
  sum of the form
  \[ \sum_{\underset{x_i \neq x_j \text{ for } i \neq j}{x_1, \ldots, x_r \in
     \{1, \ldots, N\} \backslash \{x\}}}  x_1^{k_1} \ldots x^{k_r}_r \]
  is a polynomial in $x$ of degree $k_1 + \cdots + k_r$, where its leading
  coefficient does not depend on $N$. Thus, in the sum (\ref{kamputsiki}),
  considering all the cases separately for some of the $i, i_1, \ldots, i_n$
  to be equal, we see that $F^{(n)}_i (0)$ is a polynomial in $i$ of degree
  $n$ and its leading coefficient is a linear combination of derivatives
  $\partial_{i_n} \ldots \partial_{i_1} g_i (0^N)$. Since
  \[ \frac{b_0^n}{(b_0 - b_1) \ldots (b_0 - b_m)} + \cdots + \frac{b_m^n}{(b_m
     - b_0) \ldots (b_m - b_{m - 1})} = \left\{\begin{array}{l}
       0 \text{, for } n = 0, \ldots, m - 1\\
       1 \text{, for } n = m,
     \end{array}\right. \]
  we deduce that only for $n = m$ the sum (\ref{w}) contributes to
  (\ref{tralala}). Thus, (\ref{tralala}) converges as $\varepsilon \rightarrow
  0$ and the limit does not depend on $b_0, \ldots, b_m$. We denote it by
  $c_N (m)$. The dependence of $c_N (m)$ on $N$ comes from the fact that it
  is a linear combination of derivatives
  \begin{equation}
    \partial_{i_m} \ldots \partial_{i_1} \left[ \left( \partial_i \left(
    \frac{1}{N} \log \text{} f_N \right) \right)^{l_1} \ldots \left(
    \partial_i \left( \frac{1}{N} \log \text{} f_N \right) \right)^{l_{k - m}}
    \right] (0^N), \label{antk}
  \end{equation}
  where $i, i_1, \ldots, i_m \in \{1, \ldots, m + 1\}$. Then, assumption
  (\ref{xereis}) implies that only the derivative with respect to one variable
  (which corresponds to the case $i = i_1 = \cdots = i_m$) contributes as $N
  \rightarrow \infty$, i.e.
  \[ \lim_{N \rightarrow \infty} c_N (m) = \left. \frac{\mathd^m}{\mathd x^m}
     ((\Psi' (x))^{l_1} (\Psi'' (x))^{l_2} \ldots (\Psi^{(k - m)}
     (x))^{l_{k - m}}) \right|_{x = 0} . \]
  In order to prove (\ref{flwros}) we also have to understand the dependence
  of $\lim_{\varepsilon \rightarrow 0} \mathcal{D}_k f_N \longdownminus_{x_i =
  i \varepsilon}$ from $N$. Since $c_N (m)$ does not depend on $b_0, \ldots,
  b_m$, we have
  \[ \left. \sum_{\underset{b_i \neq b_j \text{ for } i \neq j}{b_0, \ldots,
     b_m = 1}}^N \left( \frac{F_{b_0} (\varepsilon)}{(b_0 - b_1) \ldots (b_0 -
     b_m) \varepsilon^m} + \cdots + \frac{F_{b_m} (\varepsilon)}{(b_m - b_0)
     \ldots (b_m - b_{m - 1})} \right) \right|_{\varepsilon = 0} = m!
     \binom{N}{m + 1} c_N (m), \]
  and it is true that the dependence of $\lim_{\varepsilon \rightarrow 0}
  \mathcal{D}_k f_N \longdownminus_{x_i = i \varepsilon}$ on $N$ emerges
  from $(c_N (m))_{m = 0}^k$ and powers of $N$. The largest power of $N$ that
  appears is $N^{k + 1}$, which corresponds to the case $m \in \{0,
  \ldots, k\}$, $l_1 = k - m$ and $l_2 = \cdots = l_{k - m} = 0$. Thus, by the
  above we deduce that
  \[ \frac{1}{N^{k + 1}} \mathbb{E} \left[ \sum_{i = 1}^N \lambda_i^k (A)
     \right] = \sum_{m = 0}^k \frac{k!}{m! (m + 1) ! (k - m) !} \partial_1^m
     \left( \left( \partial_1 \left( \frac{1}{N} \log \text{} f_N \right)
     \right)^{k - m} \right) (0^N) + \omicron (1) . \]
  This proves (\ref{flwros}). Due to (\ref{flwros}), in order to show
  the convergence in probability it suffices to show that for every $k \in
  \mathbb{N}$ one has
  \begin{equation}
    \lim_{N \rightarrow \infty} \mathbb{E} \left[ \sum_{i = 1}^N
    \frac{\lambda_i^k (A)}{N^{k + 1}} \right]^2 = \left( \int_{\mathbb{R}} t^k
    \tmmathbf{\mu}(d \text{} t) \right)^2 . \label{comewithmenow}
  \end{equation}
  Since
  \[ \left. \mathcal{D}_k^2 \left( \mathbb{E} \left[ H \text{} C (x_1, \ldots,
     x_N ; \lambda_1 (A), \ldots, \lambda_N (A)) \right] \right) \right|_{x_1
     = \cdots = x_N = 0} =\mathbb{E} \left[ \sum_{i = 1}^N \lambda_i^k (A)
     \right]^2, \]
  we want to better understand how the operator $\mathcal{D}_k^2$ acts on
  smooth functions. For every smooth function $g$ and $k \in \mathbb{N}$ we
  have
  \[ \mathcal{D}_k^2 g =\mathcal{D}_{2 k} g + \left( \prod_{i < j} (x_i - x_j)
     \right)^{- 1} \sum_{\underset{m \neq n}{m, n = 1}}^N \partial_n^k
     \partial_m^k \left( \prod_{i < j} (x_i - x_j) g \right), \]
  and we have shown that
  \[ \lim_{N \rightarrow \infty} \frac{1}{N^{2 k + 2}} \mathcal{D}_{2 k} f_N
     \longdownminus_{x_1 = \cdots = x_N = 0} = 0. \]
  Thus, we have to control the above sum when $x_1, \ldots, x_N \rightarrow 0$
  and $N \rightarrow \infty$. By Leibniz rule
  \begin{equation}
    \sum_{\underset{n \neq m}{n, m = 1}}^N \partial_n^k \partial_m^k \left(
    \prod_{i < j} (x_i - x_j) f_N \right) = \sum_{\nu, \mu = 0}^k
    \sum_{\underset{n \neq m}{n, m = 1}}^N \binom{k}{\nu} \binom{k}{\mu}
    \partial_n^{k - \nu} \partial_m^{k - \mu} f_N \partial_n^{\nu}
    \partial_m^{\mu} \left( \prod_{i < j} (x_i - x_j) \right)
    \label{TisoukanwLilamou} .
  \end{equation}
  Dividing both sides by $\prod_{i < j} (x_i - x_j)$, we see that the
  summands in the right hand side of (\ref{TisoukanwLilamou}) are linear
  combinations of terms of the form
  \[ \frac{\partial_n^{k - \nu} \partial_m^{k - \mu} f_N}{(x_n - x_{a_1})
     \ldots (x_n - x_{a_{\nu}}) \ldots (x_m - x_{b_1}) \ldots (x_m -
     x_{b_{\mu}})}, \]
  where $a_i \neq n, b_i \neq m, a_i \neq a_j, b_i \neq b_j$ and $|\{m\} \cap
  \{a_1, \ldots, a_{\nu} \}| + |\{n\} \cap \{b_1, \ldots, b_{\mu} \}| \leq 1$.
  Taking into account condition (\ref{xereis}), we write $f_N = \exp \left( N
  \cdot \frac{1}{N} \log \text{} f_N \right)$ and we use the chain rule in
  order to write $\partial_n^{k - \nu} \partial_m^{k - \mu} f_N$ as a large sum. 
  For each summand the factors that will depend on $x_1, \ldots,
  x_N$ will be derivatives $\partial_n^{\rho} \partial_m^{\lambda} \left(
  \frac{1}{N} \log \text{} f_N \right)$ or $f_N$ and the factors that will not
  depend on $x_1, \ldots, x_N$ will be monomials in $N$. Moreover, the summand
  that corresponds to the monomial of the highest order will be
  \begin{equation}
    N^{2 k - \mu - \nu} f_N  \left( \partial_n \left( \frac{1}{N} \log \text{}
    f_N \right) \right)^{k - \nu} \left( \partial_m \left( \frac{1}{N} \log
    \text{} f_N \right) \right)^{k - \mu} \label{SEXYLADY} .
  \end{equation}
  For a notation homogeneity we write $n = a_0$ and $m = b_0$. Thus, since $f_N$
  is a symmetric function, in order to control the right hand side of
  (\ref{TisoukanwLilamou}) as $x_1, \ldots, x_N \rightarrow 0$, we can
  consider the sum of summands of the form
  \[ \frac{\prod_{i = 1}^t \partial_{a_0}^{\alpha_i} \partial_{b_0}^{\beta_i}
     \left( \frac{1}{N} \log \text{} f_N \right)}{(x_{a_0} - x_{a_1}) \ldots
     (x_{a_0} - x_{a_{\nu}}) (x_{b_0} - x_{b_1}) \ldots (x_{b_0} -
     x_{b_{\mu}})} \]
  in order to obtain a symmetrization of $\prod_{i = 1}^t \left[
  \partial_n^{\alpha_i} \partial_m^{\beta_i} \left( \frac{1}{N} \log \text{}
  f_N \right) \right] (x_1, \ldots, x_N)$ with respect to some subset of
  $\{(a, b) \in \mathbb{N}^2 \of 1 \leq a < b \leq N\}$, where $n, m \in \{1,
  2\}$. This subset will have the form 
  $$
  \{(1, l_1), \ldots, (1, l_{\rho}), (2, m_1), \ldots, (2, m_{\lambda})\}.
  $$ 
  If some of the $l_i$'s are equal to $2$,
  this corresponds to the case where some of the $a_1, \ldots, a_{\nu}$ are
  equal to $b_0$ or some of the $b_1, \ldots, b_{\mu}$ are equal to $a_0$. By
  Lemma \ref{SnikfeatFy} these symmetrizations will converge in the limit $x_1, \ldots,
  x_N \rightarrow 0$. Note that in order to compute the limit we can first
  send to zero the variables that will not appear in the denominator. Then the
  limit can be written as a sum of limits, where these summands are
  symmetrizations of the function $\prod_{i = 1}^t
  \left[ \partial_n^{\alpha_i} \partial_m^{\beta_i} \left( \frac{1}{N} \log
  \text{} f_N \right) \right] (x_1, \ldots, x_r, 0^{N - r})$ evaluated at $0^r$, where $r = |\{a_0, \ldots, a_{\nu}, b_0,
  \ldots, b_{\mu} \}|$. In order to show (\ref{comewithmenow}), we are
  interested in the limit $N \rightarrow \infty$. Due to the second part of Lemma
  \ref{SnikfeatFy} and assumption (\ref{xereis}), the symmetrization of the above
  function, with finite number of variables, evaluated at $0^r$, will converge to
  the corresponding symmetrization of
  \[ g_{\alpha, \beta, t} (x_1, \ldots, x_r) \assign \prod_{i = 1}^t
     \partial_n^{\alpha_i} \partial_m^{\beta_i} \left( \sum_{j = 1}^r \Psi
     (x_j) \right), \]
  evaluated at $0^r$. Assume that the indices $a_0, \ldots, a_{\nu}, b_0, \ldots,
  b_{\mu}$ are all distinct. Then we see that the number of different ways
  that we can choose $a_0, \ldots, a_{\nu}, b_0, \ldots, b_{\mu} \in \{1,
  \ldots, N\}$ at the sum in the right hand side of (\ref{TisoukanwLilamou}),
  divided by $\prod_{i < j} (x_i - x_j)$, in order to obtain the above
  symmetrization evaluated at zero, is equal to $O (N^{\nu + \mu + 2})$. On the
  other hand, if some of the $a_i$'s is equal to some of the $b_j$'s, then the
  corresponding symmetrization, evaluated at zero, will appear $\omicron (N^{\nu + \mu + 2})$ times in the right hand
  side of (\ref{TisoukanwLilamou}). Thus,
  since we have to divide $\mathcal{D}_k^2 f_N \longdownminus_{x_i = 0}$ by
  $N^{2 k + 2}$ before we consider the limit $N \rightarrow \infty$, we deduce that only
  the function (\ref{SEXYLADY}) will contribute to $\lim_{N \rightarrow
  \infty} \frac{1}{N^{2 k + 2}} \mathcal{D}_k^2 f_N \longdownminus_{x_i = 0}$.
  But the symmetrization of this function, evaluated at zero, contributes as $N
  \rightarrow \infty$ with terms
  \[ \left( \frac{(\Psi' (x_{b_0}))^{k - \mu}}{(x_{b_0} - x_{b_1}) \ldots
     (x_{b_0} - x_{b_{\mu}})} + \cdots + \frac{(\Psi' (x_{b_{\mu}}))^{k -
     \mu}}{(x_{b_{\mu}} - x_{b_0}) \ldots (x_{b_{\mu}} - x_{b_{\mu - 1}})}
     \right) \text{{\hspace{7em}}} \]
  \[ \text{{\hspace{9em}}} \times \left. \left( \frac{(\Psi' (x_{a_0}))^{k -
     \nu}}{(x_{a_0} - x_{a_1}) \ldots (x_{a_0} - x_{a_{\nu}})} + \cdots +
     \frac{(\Psi' (x_{a_{\nu}}))^{k - \nu}}{(x_{a_{\nu}} - x_{a_0}) \ldots
     (x_{a_{\nu}} - x_{a_{\nu - 1}})} \right) \right|_{x_i = 0} \]
  \begin{equation}
    = \left. \frac{\mathd^{\mu}}{\mathd x^{\mu}} \left( \frac{(\Psi' (x))^{k -
    \mu}}{\mu !} \right) \frac{\mathd^{\nu}}{\mathd x^{\nu}} \left(
    \frac{(\Psi' (x))^{k - \nu}}{\nu !} \right) \right|_{x = 0} \label{HH} .
  \end{equation}
  Taking into account the possible values that we can consider for $a_0,
  \ldots, a_{\nu}, b_0, \ldots, b_{\mu}$, we see that the terms (\ref{HH})
  will contribute $\mu ! \nu ! \binom{N}{\mu + 1} \binom{N - \mu - 1}{\nu +
  1}$ times to $\mathcal{D}_k^2 f_N \longdownminus_{x_i = 0}$. The factors
  $\mu !, \nu !$ exist due to the number of different ways that we can
  differentiate $x_{a_0} - x_{a_1}, \ldots, x_{a_0} - x_{a_{\nu}}$ and
  $x_{b_0} - x_{b_1}, \ldots, x_{b_0} - x_{b_{\mu}}$ respectively, in
  $\partial_{a_0}^{\nu} \partial_{b_0}^{\mu} \left( \prod_{i < j} (x_i - x_j)
  \right)$, in order to obtain the desirable denominator. Thus, we deduce that
  \[ \frac{1}{N^{2 k + 2}} \mathbb{E} \left[ \sum_{i = 1}^N \lambda_i^k (A)
     \right]^2 = \left( \sum_{m = 0}^{k - 1} \frac{k!}{m! (m + 1) ! (k - m) !}
     \left. \frac{\mathd^m}{\mathd x^m} ((\Psi' (x))^{k - m}) \right|_{x = 0}
     \right)^2 + \omicron (1), \]
  and the claim holds.
\end{proof}

In order to emphasize the connection with free probability, we recall the following technical lemma. 

\begin{lemma}
\label{lem:mom-cum}
  Let $\tmmathbf{\mu}$ be a probability measure on $\mathbb{R}$ such that it's
  moments $(\tmmathbf{\mu}_n)_{n \in \mathbb{N}}$ are given by
  \begin{equation}
    \tmmathbf{\mu}_n = \sum_{m = 0}^{n - 1} \frac{n!}{m! (m + 1) ! (n - m) !}
    \frac{\mathd^m}{\mathd x^m} ((\Psi' (x))^{n - m}) \longdownminus_{x =
    0}, \label{topg}
  \end{equation}
  where $\Psi$ is a smooth function in a neighborhood of $0$. Then the free
  cumulants $(\tmmathbf{\kappa}_n)_{n \in \mathbb{N}}$ of $\tmmathbf{\mu}$ are
  given by
  \[ \tmmathbf{\kappa}_n = \frac{\Psi^{(n)} (0)}{(n - 1) !}, \text{\qquad for
     all } \ n \in \mathbb{N}, \]
  i.e., $\Psi'$ is the $R$-transform of $\tmmathbf{\mu}$.
\end{lemma}

\begin{proof}
  We will show that for the sequence $(c_n)_{n \in \mathbb{N}}$, where $c_n
  \assign \Psi^{(n)} (0) / (n - 1) !$, the moment-cumulant relations are
  satisfied, i.e.,
  \begin{equation}
    \tmmathbf{\mu}_n = \sum_{\pi \in \tmop{NC} (n)} \prod_{V \in \pi} c_{|V|},
    \text{\quad for all } \ n \in \mathbb{N}. \label{korifi}
  \end{equation}
  By Leibniz rule we have
  \begin{equation}
    \tmmathbf{\mu}_n = \sum_{m = 0}^{n - 1} \frac{n!}{(m + 1) ! (n - m) !}
    \sum_{l_1 + \cdots + l_{n - m} = m} \frac{1}{l_1 ! \ldots l_{n - m} !}
    \Psi^{(l_1 + 1)} (0) \ldots \Psi^{(l_{n - m} + 1)} (0). \label{dm}
  \end{equation}
  For $m = 0, \ldots, n - 1$ consider $\lambda_1, \ldots, \lambda_{n - m}$
  such that $\lambda_1 + \cdots + \lambda_{n - m} = m$, and assume that
  $\{\lambda_1, \ldots, \lambda_{n - m} \} = \{\nu_1, \ldots, \nu_a \}$, where
  $\nu_i \neq \nu_j$, for every $i \neq j$. We also assume that for every $i =
  1, \ldots, a$ the element $\nu_i$ appears $r_i$ times in the set
  $\{\lambda_1, \ldots, \lambda_{n - m} \}$, which implies that
  \[ r_1 \nu_1 + \cdots + r_a \nu_a = \lambda_1 + \cdots + \lambda_{n - m} = m
     \text{\quad and\quad} r_1 + \cdots + r_a = n - m. \]
  Therefore, the number of times that the summand $\frac{1}{\lambda_1 ! \ldots
  \lambda_{n - m} !} \Psi^{(\lambda_1 + 1)} (0) \ldots \Psi^{(\lambda_{n - m}
  + 1)} (0)$ will appear in the sum
  \[ \sum_{l_1 + \cdots + l_{n - m} = m} \frac{1}{l_1 ! \ldots l_{n - m} !}
     \Psi^{(l_1 + 1)} (0) \ldots \Psi^{(l_{n - m} + 1)} (0) \]
  is equal to the number of different ways that we can cover $n - m$ points on
  a line segment with $r_1$ elements $\nu_1$, {\textdots}, $r_a$
  elements $\nu_a$; this number equals
  \[ \binom{n - m}{r_1} \binom{n - m - r_1}{r_2} \ldots \binom{n - m - r_1 -
     \cdots - r_{a - 1}}{r_a} = \frac{(n - m) !}{r_1 ! \ldots r_a !} . \]
  In order to relate the right hand side of (\ref{korifi}) with the right hand
  side of (\ref{dm}), we will relate such a $(n - m)$-tuple $(\lambda_1,
  \ldots, \lambda_{n - m})$ with the partitions $\pi = \{V_1, \ldots, V_{n -
  m} \} \in \tmop{NC} (n)$ which have $r_1$ blocks with $\nu_1 + 1$
  elements,{\textdots}, $r_a$ blocks with $\nu_a + 1$ elements. The
  number of these non-crossing partition is equal to
  \[ \frac{n!}{(n + 1 - \sum_{i = 1}^a r_i) ! \prod_{i = 1}^a r_i !} =
     \frac{n!}{(m + 1) !r_1 ! \ldots r_a !} \]
  (see, e.g., \cite{NS}); thus, we have
  \[ \sum_{\underset{\text{with } r_i \text{ blocks with } \nu_i + 1 \text{
     elements}}{\pi \in \tmop{NC} (n)}} \prod_{V \in \pi} c_{|V|} =
     \frac{n!}{(m + 1) !r_1 ! \ldots r_a !}  \frac{1}{\lambda_1 ! \ldots
     \lambda_{n - m} !} \Psi^{(\lambda_1 + 1)} (0) \ldots \Psi^{(\lambda_{n -
     m} + 1)} (0) . \]
  This proves (\ref{korifi}).
\end{proof}

\section{A variety of scalings}
\label{sec:interm}

In this section we prove statements that are similar to \eqref{eq:Intro-free-conv} and \eqref{eq:Intro-erg-conv} for various regimes of growth of a Harish-Chandra transform.

\begin{theorem}
  \label{Propolemiko}
  Let $l > 1$ be a real number. Let $A$ be a random Hermitian matrix of size $N$ and
  assume that for every $r \in \mathbb{N}$ one has
  \begin{equation}
    \lim_{N \rightarrow \infty} \frac{1}{N^l} \log \mathbb{E} [H \text{} C
    (x_1, \ldots, x_r, 0^{N - r} ; \lambda_1 (A), \ldots, \lambda_N (A))] =
    \sum_{i = 1}^r \Psi (x_i), \label{tapaidiatouzevedeou}
  \end{equation}
  where $\Psi$ is a smooth function in a complex neighborhood of $0$ and the above
  convergence is uniform in a complex neighborhood of $0^r$. Then the random measure $N^{- 1} \sum_{i = 1}^N \delta \left( N^{- l}
  \lambda_i (A) \right)$ converges, as $N \rightarrow \infty$, in probability, in the
  sense of moments to the Dirac measure $\delta \left( \Psi' (0) \right)$. In more detail, this means that any moment of the former measure converges to the corresponding moment of the latter measure in probability.
\end{theorem}

\begin{proof}
  Using the same idea as in Theorem \ref{SouvlakiA}, we will show that we can
  get an expansion $N^{l \text{} k + 1} M_{0, k, N} (\Psi) + N^{l \text{} k}
  M_{1, k, N} (\Psi) + \ldots$ for the $k$-th moment of $N^{- 1} \mathbb{E}
  \left[ \sum_{i = 1}^N \delta \left( N^{- l} \lambda_i (A) \right) \right]$, where the
  sequences $(M_{i, k, N} (\Psi))_{N \in \mathbb{N}}$ converge. This can be
  done by applying the differential operator $\mathcal{D}_k$ to 
  $$
  f_N (x_1,
  \ldots, x_N) \assign \mathbb{E} [H \text{} C (x_1, \ldots, x_N ; \lambda_1
  (A), \ldots, \lambda_N (A))],$$ 
  since $\mathbb{E} [\tmop{Tr} (A^k)] =
  \lim_{\varepsilon \rightarrow 0} \mathcal{D}_k f_N \longdownminus_{x_i = i
  \varepsilon}$. Taking into account condition (\ref{tapaidiatouzevedeou}), in
  order to obtain the desirable expansion and determine $M_{0, k, N} (\Psi)$,
  by the chain rule for $f_N = \exp \left( N^l \cdot \frac{1}{N^l} \log
  \text{} f_N \right)$, we can write the derivatives $\partial_i^{k - m} f_N$
  that are involved in $\mathcal{D}_k f_N$ in the following form:
  \begin{equation}
    \partial_i^{k - m} f_N = \sum \frac{(k - m) !}{l_1 !1!^{l_1} l_2 !2!^{l_2}
    \ldots l_{k - m} ! (k - m) !^{l_{k - m}}} N^{l (l_1 + \cdots + l_{k - m})}
    f_N \prod_{j = 1}^{k - m} \left( \partial_i^j \left( \frac{1}{N^l} \log
    \text{} f_N \right) \right)^{l_j}, \label{eskasetoportofoliaptapenintarika}
  \end{equation}
  where the sum is over non-negative integers $l_1, \ldots,
  l_{k - m}$ such that $l_1 + 2 l_2 + \cdots + (k - m) l_{k - m} = k - m$.
  Then, similarly with the proof of Theorem \ref{SouvlakiA}, considering the
  Taylor expansion for the functions
  \[ F_{i, l} (\varepsilon) \assign \left[ \left( \partial_i \left(
     \frac{1}{N^l} \log \text{} f_N \right) \right)^{l_1} \left( \partial_i^2
     \left( \frac{1}{N^l} \log \text{} f_N \right) \right)^{l_2} \ldots \left(
     \partial_i^{k - m} \left( \frac{1}{N^l} \log \text{} f_N \right)
     \right)^{l_{k - m}} \right] (\varepsilon, 2 \varepsilon, \ldots, N
     \varepsilon), \]
  we obtain that for $x_i = i \varepsilon$, for every $i = 1, \ldots, N$, the
  symmetrizations
  \[ \frac{\left( \partial_{b_0} \left( \frac{1}{N^l} \log \text{} f_N \right)
     \right)^{l_1} \ldots \left( \partial_{b_0}^{k - m} \left( \frac{1}{N^l}
     \log \text{} f_N \right) \right)^{l_{k - m}}}{(x_{b_0} - x_{b_1}) \ldots
     (x_{b_0} - x_{b_m})} + \frac{\left( \partial_{b_1} \left( \frac{1}{N^l}
     \log \text{} f_N \right) \right)^{l_1} \ldots \left( \partial_{b_1}^{k -
     m} \left( \frac{1}{N^l} \log \text{} f_N \right) \right)^{l_{k -
     m}}}{(x_{b_1} - x_{b_0}) (x_{b_1} - x_{b_2}) \ldots (x_{b_1} - x_{b_m})}
     + \]
  \begin{equation}
    \text{\qquad} \ldots + \frac{\left( \partial_{b_m} \left( \frac{1}{N^l}
    \log \text{} f_N \right) \right)^{l_1} \ldots \left( \partial_{b_m}^{k -
    m} \left( \frac{1}{N^l} \log \text{} f_N \right) \right)^{l_{k -
    m}}}{(x_{b_m} - x_{b_0}) \ldots (x_{b_m} - x_{b_{m - 1}})}
    \label{Mithridakis}
  \end{equation}
  converge as $\varepsilon \rightarrow 0$. The limit does not depend on $b_0,
  \ldots, b_m$ and it is a linear combination of derivatives
  \begin{equation}
    \partial_{i_m} \ldots \partial_{i_1} \left[ \left( \partial_i \left(
    \frac{1}{N^l} \log \text{} f_N \right) \right)^{l_1} \ldots \left(
    \partial_i^{k - m} \left( \frac{1}{N^l} \log \text{} f_N \right)
    \right)^{l_{k - m}} \right] (0^N), \label{potedeneimounakalosstosxoleio}
  \end{equation}
  where $i, i_1, \ldots, i_m \in \{1, \ldots, m + 1\}$ and the coefficients
  do not depend on $N$. As a corollary, $\lim_{\varepsilon \rightarrow 0}
  \mathcal{D}_k f_N \longdownminus_{x_i = i \varepsilon}$ is a sum of products, in which
  the factors in each summand are of the form
  (\ref{potedeneimounakalosstosxoleio}) or monomials in $N$. These monomials arise from the
  differentiation (\ref{eskasetoportofoliaptapenintarika}) and the fact that
  when $x_i = i \varepsilon$, for every $i = 1, \ldots, N$ and $\varepsilon
  \rightarrow 0$, (\ref{Mithridakis}) does not depend on $b_0, \ldots, b_m \in
  \{1, \ldots, N\}$. Thus, the summand that corresponds to the monomial of
  the highest degree is
  \[ N^{l \text{} k + 1} \left( \partial_1 \left( \frac{1}{N^l} \log \text{}
     f_N \right) (0^N) \right)^k, \]
  and it is obtained for $m = 0$ and $l_1 = k$. Therefore, we deduce that
  \[ \lim_{N \rightarrow \infty} \frac{1}{N^{l \text{} k + 1}} \mathbb{E}
     \left[ \sum_{i = 1}^N \lambda_i^k (A) \right] = (\Psi' (0))^k . \]
  Hence, by Chebyshev's inequality, in order to prove the claim it suffices to
  show that
  \[ \lim_{N \rightarrow \infty} \frac{1}{N^{2 l \text{} k + 2}} \mathbb{E}
     \left[ \sum_{i = 1}^N \lambda_i^k (A) \right]^2 = (\Psi' (0))^{2 k} . \]
  Applying the operator $\mathcal{D}_k^2$ to
  $f_N$ and considering $x_1, \ldots, x_N \rightarrow 0$, we will get an
  expansion $$N^{2 l \text{} k + 2} M'_{0, k, N} (\Psi) + N^{2 l \text{} k + 1}
  M'_{1, k, N} (\Psi) + \cdots$$ for $\mathbb{E} [\lambda_1^k (A) + \cdots +
  \lambda_N^k (A)]^2$, where the sequences $(M'_{i, k, N} (\Psi))_{N \in
  \mathbb{N}}$ converge. Since $\lim_{N \rightarrow \infty} \frac{1}{N^{2 l
  \text{} k + 2}} \mathcal{D}_k f_N \longdownminus_{x_i = 0} = 0$, by the
  above, the term $N^{2 l \text{} k + 2} M'_{0, k, N} (\Psi)$ will arise from
  \begin{equation}
    \left. \frac{1}{\prod_{i < j} (x_i - x_j)} \sum_{\underset{n \neq m}{n, m
    = 1}}^N \partial_n^k \partial_m^k \left( \prod_{i < j} (x_i - x_j) f_N
    \right) \right|_{x_i = 0} \label{Varoufakis} .
  \end{equation}
  Similarly with Theorem \ref{SouvlakiA}, (\ref{Varoufakis}) can be controlled
  by taking the sum of specific summands
  \begin{equation}
    \frac{\prod_{i = 1}^t \partial_{a_0}^{\alpha_i} \partial_{b_0}^{\beta_i}
    \left( \frac{1}{N^l} \log \text{} f_N \right)}{(x_{a_0} - x_{a_1}) \ldots
    (x_{a_0} - x_{a_{\nu}}) (x_{b_0} - x_{b_1}) \ldots (x_{b_{\mu}})}
    \label{VaroufakiS}
  \end{equation}
  in order to obtain symmetrizations of $\prod_{i = 1}^t \left[
  \partial_n^{\alpha_i} \partial_m^{\beta_i} \left( \frac{1}{N^l} \log \text{}
  f_N \right) \right] (x_1, \ldots, x_r, 0^{N - r})$ with respect to the sets
  described in Theorem \ref{SouvlakiA}, evaluated at zero, where $r = |\{a_0,
  \ldots, a_{\nu}, b_0, \ldots, b_{\mu} \}|$ and $n, m \in \{1, 2\}$. The
  terms of the form (\ref{VaroufakiS}) appear from (\ref{Varoufakis}) if we write $f_N = \exp
  (N^l \cdot \frac{1}{N^l} \log \text{} f_N)$ and use the chain rule in
  order to differentiate $f_N$. In such a way, using the same arguments as in
  Theorem \ref{SouvlakiA}, we can get the desirable expansion, where the term
  $N^{2 l \text{} k + 2} M'_{0, k, N} (\Psi)$ is equal to
  \[ N^{2 l \text{} k + 2} \left( \partial_1^k \left( \frac{1}{N^l} \log
     \text{} f_N \right) (0^N) \right)^2 \]
  and emerges from the summand
  \[ \sum_{\underset{n \neq m}{n, m = 1}}^N \partial_n^k \partial_m^k f_N \]
  of (\ref{Varoufakis}). This proves the claim.
\end{proof}

Now we present a different limit regime, which is related to ergodic unitarily invariant measures on infinite Hermitian matrices. We prove the implication \eqref{eq:Intro-erg-conv}, generalizing the result of \cite[Theorem 4.1]{OV}.

\begin{theorem}
  \label{BEBT!}Let $A$ be a random Hermitian matrix of size $N$ and assume
  that for every $r \in \mathbb{N}$ one has
  \begin{equation}
    \lim_{N \rightarrow \infty} \log \mathbb{E} [H \text{} C (x_1, \ldots,
    x_r, 0^{N - r} ; \lambda_1 (A), \ldots, \lambda_N (A))] = \sum_{i = 1}^r
    \Psi (x_i), \label{rightontime}
  \end{equation}
  where $\Psi$ is a smooth function in a neighborhood of $0$ and the above
  convergence is uniform in a neighborhood of $0^r$. Then for every $k \in
  \mathbb{N}$, the $k$-th moment of the random measure $\sum_{i = 1}^N
  \delta \left( N^{- 1} \lambda_i (A) \right)$ converges as $N \rightarrow \infty$ in
  probability to $\Psi^{(k)} (0) / (k - 1) !$.
\end{theorem}

\begin{proof}
  First we show that for every $k \in \mathbb{N}$
  \begin{equation}
    \lim_{N \rightarrow \infty} \mathbb{E} \left[ \sum_{i = 1}^N
    \frac{\lambda_i^k (A)}{N^k} \right] = \frac{\Psi^{(k)} (0)}{(k - 1) !}
    \label{PiPitopapi} .
  \end{equation}
  This can be done in a similar way as in Theorem \ref{SouvlakiA}, using that
  $\mathbb{E} [\tmop{Tr} (A^k)] = \lim_{\varepsilon \rightarrow 0}
  \mathcal{D}_k f_N \longdownminus_{x_i = i \varepsilon}$, where $f_N (x_1,
  \ldots, x_N) \assign \mathbb{E} [H \text{} C (x_1, \ldots, x_N ; \lambda_1
  (A), \ldots, \lambda_N (A))]$. Unlike condition (\ref{xereis}), assumption
  (\ref{rightontime}) leads to an expansion $N^k m_{0, k, N} (\Psi) + N^{k -
  1} m_{1, k, N} (\Psi) + \cdots$ for $\mathbb{E} [\tmop{Tr} (A^k)]$, where
  the sequences $(m_{i, k, N} (\Psi))_{N \in \mathbb{N}}$ converge. This holds
  because $\mathcal{D}_k f_N$ is a linear combination of terms
  \begin{equation}
    \frac{\left( \partial_{b_0} (\log \text{} f_N) \right)^{l_1} \ldots \left(
    \partial^{k - m}_{b_0} \left( \log \text{} f_N \right) \right)^{l_{k -
    m}}}{(x_{b_0} - x_{b_1}) \ldots (x_{b_0} - x_{b_m})} + \cdots +
    \frac{(\partial_{b_m} (\log \text{} f_N))^{l_1} \ldots \left(
    \partial_{b_m}^{k - m} (\log \text{} f_N) \right)^{l_{k - m}}}{(x_{b_m} -
    x_{b_0}) \ldots (x_{b_m} - x_{b_{m - 1}})}, \label{lil}
  \end{equation}
  and using the same arguments as before, we obtain that these terms do
  not depend on $b_0, \ldots, b_m \in \{1, \ldots, N\}$ and that they converge when
  $x_i = i \varepsilon$, $\varepsilon \rightarrow 0$ and $N \rightarrow
  \infty$. Thus, the limit $\lim_{\varepsilon \rightarrow 0} \mathcal{D}_k f_N
  \longdownminus_{x_i = i \varepsilon}$ is a sum, in which the summand with factor
  (\ref{lil}), for $x_i = i \varepsilon$ and $\varepsilon \rightarrow 0$, will
  also have as a factor a polynomial in $N$. Its degree is $m + 1$, i.e., the
  number of distinct variables in the denominator of (\ref{lil}). In
  comparison with assumption (\ref{xereis}), in the context of Theorem \ref{SouvlakiA} in
  order to get the similar expansion for $\mathbb{E} [\tmop{Tr} (A^k)]$ that
  leads to the limit result that we proved, we have to write $f_N = \exp
  \left( N \cdot \frac{1}{N} \log \text{} f_N \right)$ before we compute the
  derivatives $\partial_i^{k - m} f_N$ that involved in $\mathcal{D}_k f_N$
  according to (\ref{alaniariko}). Thus, in that case the corresponding
  polynomials in $N$ depended also on the differentiation of $f_N$. Concluding,
  we see that we can get the desirable expansion and the term $N^k m_{0, k, N}
  (\Psi)$ will arise from the summand
  \[ \frac{k}{\prod_{i < j} (x_i - x_j)}  \sum_{i = 1}^N \partial_i f_N
     \partial_i^{k - 1} \left( \prod_{i < j} (x_i - x_j) \right) \]
  of $\mathcal{D}_k f_N$. Using the same arguments as before, we obtain
  that the terms (\ref{lil}) for $b_0, \ldots, b_{m \in} \{1, \ldots, N\}$
  will contribute to $\mathbb{E} [\tmop{Tr} (A^k)]$ when $m = k - 1$, $x_i =
  i \varepsilon$ and $\varepsilon \rightarrow 0$ with
  \[ \binom{N}{k} \binom{k}{k - 1} \left( \partial_1^k (\log \text{} f_N)
     (0^N) + \omicron (1) \right), \]
  which implies that $m_{0, k, N} (\Psi) = (\partial_1^k (\log \text{} f_N)
  (0^N) + \omicron (1)) / (k - 1) !$ and (\ref{PiPitopapi}) holds. In order to
  prove the claim it suffices to show that for every $k \in \mathbb{N}$ one has
  \begin{equation}
    \lim_{N \rightarrow \infty} \mathbb{E} \left[ \sum_{i = 1}^N
    \frac{\lambda_i^k (A)}{N^k} \right]^2 = \left( \frac{\Psi^{(k)} (0)}{(k -
    1) !} \right)^2 \label{Rr} .
  \end{equation}
  
  We will show that $\frac{1}{N^{2 k}} \mathcal{D}_k^2 f_N
  \longdownminus_{x_i = 0}$ converges as $N \rightarrow \infty$ to the left
  hand side of (\ref{Rr}). The difference compared to Theorem \ref{SouvlakiA}
  is that now the term $\frac{1}{N^{2 k}} \mathcal{D}_{2 k} f_N
  \longdownminus_{x_i = 0}$ will contribute to the limit $\lim_{N \rightarrow
  \infty} \frac{1}{N^{2 k}} \mathcal{D}_k^2 f_N \longdownminus_{x_i = 0}$, due
  to relation (\ref{PiPitopapi}). Thus, we have to show that
  \begin{equation}
    \lim_{N \rightarrow \infty} \left( \frac{1}{N^{2 k} \prod_{i < j} (x_i -
    x_j)} \sum_{\underset{n \neq m}{n, m = 1}}^N \left. \partial_n^k
    \partial_m^k \left( \prod_{i < j} (x_i - x_j) f_N \right) \right|_{x_i =
    0} \right) = \left( \frac{\Psi^{(k)} (0)}{(k - 1) !} \right)^2 -
    \frac{\Psi^{(2 k)} (0)}{(2 k - 1) !} . \label{Yyy}
  \end{equation}
  For the same reasons as in Theorem \ref{SouvlakiA}, the term
  \begin{equation}
    \frac{1}{\prod_{i < j} (x_i - x_j)} \sum_{\nu, \mu = 0}^k
    \sum_{\underset{n \neq m}{n, m = 1}}^N \binom{k}{\nu} \binom{k}{\mu}
    \partial_n^{k - \nu} \partial_m^{k - \mu} f_N \partial_n^{\nu}
    \partial_m^{\mu} \left( \prod_{i < j} (x_i - x_j) \right) \label{Sarah}
  \end{equation}
  can be controlled when $x_1, \ldots, x_N \rightarrow 0$ if we consider the sum
  of specific summands of the form
  \begin{equation}
    \binom{k}{\nu} \binom{k}{\mu} \frac{\partial_{a_0}^{k - \nu}
    \partial_{b_0}^{k - \mu} f_N}{(x_{a_0} - x_{a_1}) \ldots (x_{a_0} -
    x_{a_{\nu}}) (x_{b_0} - x_{b_1}) \ldots (x_{b_0} - x_{b_{\mu}})}
    \label{numbern}
  \end{equation}
  in order to make symmetrizations. We recall that for the indices $a_0,
  \ldots, a_{\nu}, b_0, \ldots, b_{\mu}$ it is required that $a_0 \neq b_0$,
  $a_i \neq a_j$, $b_i \neq b_j$ and $|\{a_0 \} \cap \{b_1, \ldots, b_{\mu}
  \}| + |\{b_0 \} \cap \{a_1, \ldots, a_{\nu} \}| \leq 1$. Then, writing $f_N
  = \exp (\log \text{} f_N)$ and using the chain rule in order to express the
  derivatives $\partial_{a_0}^{k - \nu} \partial_{b_0}^{k - \mu} f_N$,
  assumption (\ref{rightontime}) implies that we can get an
  expansion $N^{2 k} M_{0, k, N} (\Psi) + N^{2 k - 1} M_{1, k, N} (\Psi) +
  \cdots$ for (\ref{Sarah}), where the sequences $(M_{i, k, N} (\Psi))_{N \in
  \mathbb{N}}$ converge. The coefficient of $N^i$ arises from terms
  (\ref{numbern}) where in the denominator there are at least $i$ variables.
  The reason that in our expansion there is no summand with factor $N^{2 k +
  2}$ is that for every $k \in \mathbb{N}$ one has
  \begin{equation}
    \sum_{\underset{n \neq m}{n, m = 1}}^N \partial_n^k \partial_m^k \left(
    \prod_{i < j} (x_i - x_j) \right) = - \sum_{n = 1}^N \partial_n^{2 k}
    \left( \prod_{i < j} (x_i - x_j) \right) = 0 \label{apopseimekodasou} .
  \end{equation}
  Thus, for $\nu = \mu = k$, the summands (\ref{numbern}) do not contribute to
  (\ref{Sarah}). Similarly, if $(\nu, \mu) = (k - 1, k)$ and all the indices
  $a_0, \ldots, a_{k - 1}, b_0, \ldots, b_k$ are distinct, then the
  corresponding terms (\ref{numbern}) cancel out, since
  \begin{equation}
    \frac{k \partial_{a_0} f_N}{(x_{a_0} - x_{a_1}) \ldots (x_{a_0} - x_{a_{k
    - 1}})} \left( \frac{1}{(x_{b_0} - x_{b_1}) \ldots (x_{b_0} - x_{b_k})} +
    \cdots + \frac{1}{(x_{b_k} - x_{b_0}) \ldots (x_{b_k} - x_{b_{k - 1}})}
    \right) = 0 \label{12345} .
  \end{equation}
  Clearly, the same holds if $(\nu, \mu) = (k, k - 1)$. That's why in our
  expansion there is no summand with factor $N^{2 k + 1}$. In order to
  determine $M_{0, k, N} (\Psi)$, we have to consider terms of the form (\ref{numbern}),
  where either $(\nu, \mu) \in \{(k - 1, k), (k, k - 1)\}$ and $|\{a_0,
  \ldots, a_{\nu}, b_0, \ldots, b_{\mu} \}| = 2 k$, or $(\nu, \mu) = (k - 1, k
  - 1)$ and all $a_0, \ldots, a_{k - 1}, b_0, \ldots, b_{k - 1}$ are
  distinct. We start with the first case, where $(\nu, \mu) = (k - 1, k)$. We
  have that one of the $a_0, \ldots, a_{k - 1}$ is equal to one of the $b_0,
  \ldots, b_k$. If $a_i = b_j$ for some $i \geq 1$ and $j \geq 0$, then the
  corresponding terms (\ref{numbern}) cancel out since the equality
  (\ref{12345}) will also hold in this case. It remains to examine the case
  where $b_i = a_0$, for some $i \geq 1$. Let $\alpha_0, \ldots, \alpha_{k -
  1}, \beta_0, \ldots, \beta_{k - 1} \in \{1, \ldots, N\}$ be fixed and pairwise distinct. Then, for $(\nu, \mu) \in \{(k, k - 1), (k - 1,
  k)\}$, we get a contribution
  \[ \frac{2 kk! (k - 1) ! \partial_{\alpha_0} f_N}{(x_{\alpha_0} -
     x_{\alpha_1}) \ldots (x_{\alpha_0} - x_{\alpha_{k - 1}})} \left(
     \frac{1}{(x_{\beta_0} - x_{\alpha_0}) (x_{\beta_0} - x_{\beta_1}) \ldots
     (x_{\beta_0} - x_{\beta_{k - 1}})} + \right. \]
  \[ \left. \frac{1}{(x_{\beta_1} - x_{\alpha_0}) (x_{\beta_1} - x_{\beta_0})
     \ldots (x_{\beta_1} - x_{\beta_{k - 1}})} + \cdots +
     \frac{1}{(x_{\beta_{k - 1}} - x_{\alpha_0}) (x_{\beta_{k - 1}} -
     x_{\beta_0}) \ldots (x_{\beta_{k - 1}} - x_{\beta_{k - 2}})} \right) \]
  \begin{equation}
    = \frac{- 2 kk! (k - 1) ! \partial_{\alpha_0} f_N}{(x_{\alpha_0} -
    x_{\alpha_1}) \ldots (x_{\alpha_0} - x_{\alpha_{k - 1}}) (x_{\alpha_0} -
    x_{\beta_0}) \ldots (x_{\alpha_0} - x_{\beta_{k - 1}})} \label{tserlio}
  \end{equation}
  to (\ref{Sarah}). The factor $k! (k - 1) !$ appears in the numerator because
  it is the number of different ways that we can differentiate $x_{\alpha_0} -
  x_{\alpha_1}, \ldots, x_{\alpha_0} - x_{\alpha_{k - 1}}, x_{\beta_0} -
  x_{\alpha_0}, x_{\beta_0} - x_{\beta_1}, \ldots, x_{\beta_0} - x_{\beta_{k -
  1}}$ in $\partial_{\alpha_0}^{k - 1} \partial_{\beta_0}^k \left( \prod_{i <
  j} (x_i - x_j) \right)$ in order to get the summand with denominator
  $(x_{\alpha_0} - x_{\alpha_1}) \ldots (x_{\alpha_0} - x_{\alpha_{k - 1}})
  (x_{\beta_0} - x_{\alpha_0}) (x_{\beta_0} - x_{\beta_1}) \ldots (x_{\beta_0}
  - x_{\beta_{k - 1}})$ when we divide with $\prod_{i < j} (x_i - x_j)$. The
  term (\ref{tserlio}) will contribute $\binom{2 k - 1}{k - 1}$ times to
  (\ref{Sarah}). Indeed, we have to choose $k - 1$ terms $x_{\alpha_0} - x_i$,
  where $i = \alpha_1, \ldots, \alpha_{k - 1}, \beta_0, \ldots, \beta_{k - 1}$
  to differentiate with respect to $x_{\alpha_0}$. Then the terms that we have
  to add in order to get (\ref{tserlio}) are fixed. Replacing the
  leading variable $x_{\alpha_0}$ with $x_{\alpha_1}, \ldots, x_{\alpha_{k -
  1}}, x_{\beta_0}, \ldots, x_{\beta_{k - 1}}$ and writing $x_{\alpha_{k + i}}
  \assign x_{\beta_i}$ for every $i = 0, \ldots, k - 1$, we have that the
  contribution of the terms (\ref{numbern}), stated above, for the
  indices $\alpha_0, \ldots, \alpha_{2 k - 1} \in \{1, \ldots, N\}$, is
  \begin{equation}
    \frac{- (2 k) ! \partial_{\alpha_0} f_N}{(x_{\alpha_0} - x_{\alpha_1})
    \ldots (x_{\alpha_0} - x_{\alpha_{2 k - 1}})} + \cdots + \frac{- (2 k) !
    \partial_{\alpha_{2 k - 1}} f_N}{(x_{\alpha_{2 k - 1}} - x_{\alpha_0})
    \ldots (x_{\alpha_{2 k - 1}} - x_{\alpha_{2 k - 2}})} . \label{Romantic}
  \end{equation}
  But when $x_1, \ldots, x_N \rightarrow 0$ and $N \rightarrow \infty$,
  (\ref{Romantic}) is equal to
  \[ - \frac{(2 k) ! \Psi^{(2 k)} (0)}{(2 k - 1) !} . \]
  Since we have $\binom{N}{2 k}$ options for choosing a $2 k$-tuple
  $(\alpha_0, \ldots, \alpha_{2 k - 1}) \in \{1, \ldots, N\}^{2 k}$, we deduce
  that the contribution of the terms (\ref{numbern}) to
  the left hand side of (\ref{Yyy}) is $- \Psi^{(2 k)} (0) / (2 k - 1) !$.
  For the case where $\nu = \mu = k - 1$ and $\alpha_0, \ldots, \alpha_{k -
  1}, \beta_0, \ldots, \beta_{k - 1} \in \{1, \ldots, N\}$ are all distinct,
  we obtain the terms
  \begin{equation}
    \frac{2 (k!)^2 \partial_{\alpha_0} \partial_{\beta_0} f_N}{(x_{\alpha_0} -
    x_{\alpha_1}) \ldots (x_{\alpha_0} - x_{\alpha_{k - 1}}) (x_{\beta_0} -
    x_{\beta_1}) \ldots (x_{\beta_0} - x_{\beta_{k - 1}})}. \label{thaxestw}
  \end{equation}
  The limit of these terms as $x_1, \ldots, x_N \rightarrow 0$ and $N
  \rightarrow \infty$ is
  \[ \frac{1}{2} \binom{2 k}{k} 2 (k!)^2 \lim_{x_1, \ldots, x_N \rightarrow 0}
     \left( \frac{\Psi' (x_{\alpha_0})}{(x_{\alpha_0} - x_{\alpha_1}) \ldots
     (x_{\alpha_0} - x_{\alpha_{k - 1}})} + \cdots + \frac{\Psi' (x_{\alpha_{k
     - 1}})}{(x_{\alpha_{k - 1}} - x_{\alpha_0}) \ldots (x_{\alpha_{k - 1}} -
     x_{\alpha_{k - 2}})} \right) \]
  \[ \times \left( \frac{\Psi' (x_{\beta_0})}{(x_{\beta_0} - x_{\beta_1})
     \ldots (x_{\beta_0} - x_{\beta_{k - 1}})} + \cdots + \frac{\Psi'
     (x_{\beta_{k - 1}})}{(x_{\beta_{k - 1}} - x_{\beta_{k - 2}})} \right) =
     (2 k) ! \left( \frac{\Psi^{(k)} (0)}{(k - 1) !} \right)^2 . \]
  The factor $2^{- 1} \binom{2 k}{k}$ exists in the above product because it
  is equal to the number of different ways that we can split $(\alpha_0,
  \ldots, \alpha_{k - 1}, \beta_0, \ldots, \beta_{k - 1})$ to two $k$-tuples
  such that the sum of terms corresponding to the case (\ref{thaxestw}) gives the above limit,
  as $x_1, \ldots, x_N \rightarrow 0$ and $N \rightarrow \infty$. Hence, the
  contribution of terms (\ref{thaxestw}) to the left hand side of (\ref{Yyy})
  is $(\Psi^{(k)} (0) / (k - 1) !)^2$. This implies that (\ref{Yyy}) holds and
  the claim has been proven.
\end{proof}

Note that the limit regimes of Theorem \ref{SouvlakiA} and Theorem \ref{BEBT!}
are quite similar. Thus, it is reasonable to ask for an intermediate limit
regime. It is addressed by the next theorem.

\begin{theorem}
  \label{HALFGODHALFSOUVLAKI}Let $A$ be a random Hermitian matrix of size $N$
  and $0 < \theta < 1$. Assume that for every finite $r$ one has
  \begin{equation}
    \lim_{N \rightarrow \infty} \frac{1}{N^{\theta}} \log \mathbb{E} [H
    \text{} C (x_1, \ldots, x_r, 0^{N - r} ; \lambda_1 (A), \ldots, \lambda_N
    (A))] = \sum_{i = 1}^r \Psi (x_i), \label{manytoomany}
  \end{equation}
  where $\Psi$ is a smooth function in a complex neighborhood of $0$ and the above
  convergence is uniform in a complex neighborhood of $0^r$. Then, for every $k \in
  \mathbb{N}$, the $k$-th moment of the random measure $N^{- \theta} \sum_{i =
  1}^N \delta \left( N^{- 1} \lambda_i (A) \right)$ converges, as $N \rightarrow \infty$, in
  probability to $\Psi^{(k)} (0) / (k - 1) !$.
\end{theorem}

\begin{proof}
  First, in order to show that for every $k \in \mathbb{N}$
  \begin{equation}
    \lim_{N \rightarrow \infty} \frac{1}{N^{\theta + k}} \mathbb{E} \left[
    \sum_{i = 1}^N \lambda_i^k (A) \right] = \frac{\Psi^{(k)} (0)}{(k - 1) !},
    \label{salamalekum}
  \end{equation}
  we use the same technique as in Theorem \ref{SouvlakiA} and Theorem
  \ref{BEBT!} which gives us an expansion of $\lim_{\varepsilon
  \rightarrow 0} \mathcal{D}_k f_N \longdownminus_{x_i = i \varepsilon}$. In this expansion the leading summand will
  be of the form $N^{\theta + k} m_{0, k, N} (\Psi)$, where the sequence $(m_{0, k, N}
  (\Psi))_{N \in \mathbb{N}}$ converges and the remaining summands will be
  $\omicron (N^{\theta + k})$. This can be done if we use the formula
  (\ref{eskasetoportofoliaptapenintarika}) for the derivatives $\partial_i^{k
  - m} f_N$ of $f_N (x_1, \ldots, x_N) =\mathbb{E} [H \text{} C (x_1, \ldots,
  x_N ; \lambda_1 (A), \ldots, \lambda_N (A))]$, where now $l$ is replaced by
  $\theta$, since (\ref{manytoomany}) and Lemma \ref{SnikfeatFy} imply that
  the terms (\ref{Mithridakis}) converge as $x_1, \ldots, x_N \rightarrow 0$
  and $N \rightarrow \infty$. Thus, for every $m = 0, \ldots, k - 1$, we can
  write the summand
  \[ \frac{1}{\prod_{i < j} (x_i - x_j)} \sum_{i = 1}^N \binom{k}{m}
     \partial_i^{k - m} f_N  \left. \partial_i^m \left( \prod_{i < j} (x_i -
     x_j) \right) \right|_{x_i = 0} \]
  of $\mathbb{E} [\tmop{Tr} (A^k)]$ as a sum of terms of the form $N^{i_{m, k} \theta + j_{m,
  k}} h_{m, k, N} (\Psi)$, where $(h_{m, k, N} (\Psi))_{N \in \mathbb{N}}$
  converges. The summand that corresponds to the largest power of $N$ is
  \[ N^{(k - m) \theta + m + 1} \left( \frac{k!}{m! (m + 1) ! (k - m) !}
     \partial_1^m \left( \partial_1 \left( \frac{1}{N^{\theta}} \log \text{}
     f_N \right) \right)^{k - m} (0^N) + \omicron (1) \right) . \]
  Since $(k - m) \theta + m + 1 \leq \theta + k$, for every $m = 0, \ldots, k
  - 1$, we deduce that (\ref{salamalekum}) holds. On the other hand, in order
  to prove
  \begin{equation}
    \lim_{N \rightarrow \infty} \frac{1}{N^{2 \theta + 2 k}} \mathbb{E} \left[
    \sum_{i = 1}^N \lambda_i^k (A) \right]^2 = \left( \frac{\Psi^{(k)} (0)}{(k
    - 1) !} \right)^2, \label{metemfiliako}
  \end{equation}
  first note that $\lim_{N \rightarrow \infty} \frac{1}{N^{2 \theta + 2 k}}
  \mathcal{D}_{2 k} f_N \longdownminus_{x_i = 0}$ does not contribute to the
  above limit. Similarly, writing $$f_N = \exp \left( N^{\theta} \cdot
  \frac{1}{N^{\theta}} \log \text{} f_N \right)$$ and using Lemma
  \ref{SnikfeatFy} and (\ref{manytoomany}), the expression
  \begin{equation}
    \left. \frac{1}{\prod_{i < j} (x_i - x_j)} \sum_{\underset{n \neq m}{n, m
    = 1}}^N \binom{k}{\nu} \binom{k}{\mu} \partial_n^{k - \nu} \partial_m^{k -
    \mu} f_N \partial_n^{\nu} \partial_m^{\mu} \left( \prod_{i < j} (x_i -
    x_j) \right) \right|_{x_i = 0} \label{thelwex}
  \end{equation}
  can be written as a sum of terms of the form $N^{i_{\nu, \mu, k} \theta + j_{\nu, \mu, k}}
  h_{\nu, \mu, k, N} (\Psi)$, where $(h_{\nu, \mu, k, N} (\Psi))_{N \in
  \mathbb{N}}$ converges. The largest power of $N$ that appears as factor of a
  summand is $N^{(2 k - \nu - \mu) \theta + \nu + \mu + 2}$. In Theorem
  \ref{BEBT!} we have shown that for $(\nu, \mu) = (k - 1, k)$ the largest
  power of $N$ that appears as factor of a summand is $N^{\theta + 2 k}$.
  Moreover, we have shown that for $(\nu, \mu) = (k - 1, k - 1)$,
  (\ref{thelwex}) divided by $N^{2 \theta + 2 k}$ converges to $(\Psi^{(k)}
  (0) / (k - 1) !)^2$, as $N \rightarrow \infty$. Thus, due to
  (\ref{apopseimekodasou}) and the condition $(2 k - \mu - \nu) \theta + \nu + \mu + 2 \leq
  2 \theta + 2 k$, for every $0 \leq \nu, \mu \leq k - 1$, we deduce that
  (\ref{metemfiliako}) holds.
\end{proof}


\begin{remark}
Comparing Theorems \ref{SouvlakiA} and \ref{BEBT!}, we see that they produce Law of Large Numbers results for a matrix $A$ in two different regimes of growth. The first one is strongly related with free probability, and it is natural to ask whether the second one is related with it as well. The answer is positive. We show that the regime of growth from Theorem \ref{BEBT!} is related to \textit{infinitesimal freeness}. 

From the point of view of infinitesimal free probability, Theorem \ref{BEBT!} implies
that the empirical spectral distribution of $A / N$ converges to $\delta(0)$
and the $1 / N$ correction $\{m_k' \}$ is equal to
\[ m_k' = \frac{\Psi^{(k)} (0)}{(k - 1) !} = \sum_{\pi \in \tmop{NC} (k)}
   \sum_{V \in \pi} \frac{\Psi^{(|V|)} (0)}{(|V| - 1) !} \prod_{\underset{W
   \neq V}{W \in \pi}} \kappa_{|W|} (\delta(0)), \]
where $\{\kappa_n (\delta(0))\}$ are the free cumulants of $\delta(0)$. In other
words, $\Psi'$ is the infinitesimal $R$-transform. In Theorem \ref{Einaiiagapi} below, we combine the limit regimes of Theorems \ref{SouvlakiA} and \ref{BEBT!} (thus, generalizing both of them). 

\end{remark}

\section{Infinitesimal free probability}
\label{sec:inf-free}

The goal of this section is to study the first order correction to the Law of Large Numbers for the empirical measure of random matrices via a more detailed information about their Harish-Chandra transform.

\begin{theorem}
  \label{Einaiiagapi} Let $A$ be a random Hermitian matrix of size $N$. Assume that for every finite $r$ one has
  \begin{equation}
    \lim_{N \rightarrow \infty} N \left( \frac{1}{N} \log \mathbb{E}[H \text{}
    C (x_1, \ldots, x_r, 0^{N - r} ; \lambda_1 (A), \ldots, \lambda_N (A))] -
    \sum_{i = 1}^r \Psi (x_i) \right) = \sum_{i = 1}^r \Phi (x_i), \label{ola}
  \end{equation}
  where $\Psi, \Phi$ are smooth functions in a (complex) neighborhood of $0$ and the
  above convergence is uniform in a (complex) neighborhood of $0^r$. Then one has the following     limit:
  \begin{equation}
    \lim_{N \rightarrow \infty} N \left( \frac{1}{N^{k + 1}} \mathbb{E} \left[
    \sum_{i = 1}^N \lambda_i^k (A) \right] -\tmmathbf{\mu}_k \right)
    =\tmmathbf{\mu}'_k, \label{rin}
  \end{equation}
  where $(\tmmathbf{\mu}_k)_{k \in \mathbb{N}}$ are given by (\ref{topg}) and
  \begin{equation}
    \tmmathbf{\mu}'_k \assign \sum_{m = 0}^{k - 1} \frac{k!}{m! (m + 1) !
    (k - m - 1) !}  \left. \frac{\mathd^m}{\mathd x^m} ((\Psi' (x))^{k - m
    - 1} \Phi' (x)) \right|_{x = 0} . \label{G}
  \end{equation}
\end{theorem}

\begin{proof}
  Note that (\ref{ola}) implies that the relation (\ref{xereis}) is satisfied
  and the empirical distribution of $A / N$ converges in probability to a
  measure with moments $(\tmmathbf{\mu}_k)_{k \in \mathbb{N}}$. In the proof of Theorem \ref{SouvlakiA} we have
  shown the existence of an expansion
  \begin{equation}
    \frac{1}{N^{k + 1}} \mathbb{E} \left[ \sum_{i = 1}^N \lambda_i^k (A)
    \right] = \sum_{i = 0}^k \frac{1}{N^i} M_{i, N} (\Psi), \label{pits}
  \end{equation}
  where for every $i \in \{0, \ldots, k\}$ the sequence $(M_{i, N} (\Psi))_{N
  \in \mathbb{N}}$ converges and $M_{0, N} (\Psi)$ is a linear combination of
  derivatives
  \begin{equation}
    \partial_{i_m} \ldots \partial_{i_1} \left( \partial_{i_0} \left(
    \frac{1}{N} \log \text{} f \right) \right)^{k - m} (0^N), \label{viv}
  \end{equation}
  where $f (x_1, \ldots, x_N) =\mathbb{E} (H \text{} C (x_1, \ldots, x_N ;
  \lambda_1 (A), \ldots, \lambda_N (A)))$ and $m \in \{0, \ldots, k\}$. Due to
  (\ref{ola}), if $|\{i_0, \ldots, i_m \}|$>1, then the terms (\ref{viv})
  multiplied by $N$ will not contribute to the $1 / N$ correction as $N
  \rightarrow \infty$. Thus,
  \[ N (M_{0, N} (\Psi) -\tmmathbf{\mu}_k) = \]
  \[ = \sum_{m = 0}^{k - 1} \frac{k!}{m! (m + 1) ! (k - m) !} N \left(
     \partial_1^m \left( \partial_1 \left( \frac{1}{N} \log \text{} f \right)
     \right)^{k - m} (0^N) - \left. \frac{\mathd^m}{\mathd x^m} (\Psi'
     (x))^{k - m} \right|_{x = 0} \right) + \omicron (1), \]
  as $N \rightarrow \infty$. For $m \in \{0, \ldots, k - 1\}$, applying
  Leibniz rule, we have
  \[ N \left( \partial_1^m \left( \partial_1 \left( \frac{1}{N} \log \text{} f
     \right) \right)^{k - m} (0^N) - \left. \frac{\mathd^m}{\mathd x^m}
     (\Psi' (x))^{k - m} \right|_{x = 0} \right) = \]
  \begin{equation}
    = \sum_{l_1 + \cdots + l_{k - m} = m} \frac{m!}{l_1 ! \ldots l_{k - m} !}
    N \left( \prod_{i = 1}^{k - m} \partial_1^{l_i + 1} \left( \frac{1}{N}
    \log \text{} f \right) (0^N) - \prod_{i = 1}^{k - m} \Psi^{(l_i + 1)} (0)
    \right) \label{soli} .
  \end{equation}
  For any choice of $l_1, \ldots, l_{k - m}$, due to (\ref{ola}), we have
  \[ \lim_{N \rightarrow \infty} N \left( \partial_1^{l_i + 1} \left(
     \frac{1}{N} \log \text{} f \right) (0^N) - \Psi^{(l_i + 1)} (0) \right) =
     \Phi^{(l_i + 1)} (0), \text{\quad for every \ } i = 1, \ldots, k - m, \]
  and by induction on $k - m$
  \[ \lim_{N \rightarrow \infty} N \left( \prod_{i = 1}^{k - m}
     \partial_1^{l_i + 1} \left( \frac{1}{N} \log \text{} f \right) (0^N) -
     \prod_{i = 1}^{k - m} \Psi^{(l_i + 1)} (0) \right) = \]
  \[ = \Phi^{(l_1 + 1)} (0) \Psi^{(l_2 + 1)} (0) \ldots \Psi^{(l_{k - m} + 1)}
     (0) + \cdots + \Psi^{(l_1 + 1)} (0) \ldots \Psi^{(l_{k - m - 1} + 1)} (0)
     \Phi^{(l_{k - m} + 1)} (0) . \]
  Thus, as $N \rightarrow \infty$, the expression (\ref{soli}) is equal to
  \[ (k - m) \left. \frac{\mathd^m}{\mathd x^m} ((\Psi' (x))^{k - m - 1}
     \Phi' (x)) \right|_{x = 0}, \]
  which implies that $\lim_{N \rightarrow \infty} N (M_{0, N} (\Psi)
  -\tmmathbf{\mu}_k) =\tmmathbf{\mu}'_k$. As a corollary, the claim holds
  if $\lim_{N \rightarrow \infty} M_{1, N} (\Psi) = 0$. In order to determine
  $M_{1, N} (\Psi)$ we have to better understand the rule that gives expansion
  (\ref{pits}) or equivalently how the operator $\mathcal{D}_k$ acts on
  function $f$. Our approach from Theorem \ref{SouvlakiA} shows that the dependence of
  $\lim_{\varepsilon \rightarrow 0} \mathcal{D}_k f \longdownminus_{x_i = i
  \varepsilon}$ on $N$ emerges from powers of $N$ and the terms $(M_{i, N}
  (\Psi))_{i = 0}^k .$ The powers of $N$ appear from the differentiation of $f
  = \exp \left( N \cdot \frac{1}{N} \log \text{} f \right)$ (see formula
  (\ref{kefi})) and the fact that the symmetric terms (\ref{tralala}) do not
  depend on $b_0, \ldots, b_m \in \{1, \ldots, N\}$, where $x_i = i
  \varepsilon$, for $i = 1, \ldots, N$, and $\varepsilon \rightarrow 0$.
  Thus, $\lim_{\varepsilon \rightarrow 0} \mathcal{D}_k f \longdownminus_{x_{i
  = i \varepsilon}}$ can be written as a linear combination of terms
  \begin{equation}
    N^{l_1 + \cdots + l_{k - m}} \binom{N}{m + 1} c_N (m), \label{jij}
  \end{equation}
  where $m \in \{0, \ldots, k\}$, $l_1 + 2 l_2 + \cdots + (k - m) l_{k - m} =
  k - m$ and $c_N (m)$ are linear combinations of derivatives of the form
  (\ref{antk}). In order to determine $M_{1, N} (\Psi)$, we have to compute the
  summand with factor $N^k$. There are two cases that we have to consider. For
  $l_1 + \cdots + l_{k - m} = k - m - 1$, i.e. $l_1 = k - m - 2$, $l_2 = 1$
  and $l_3 = \cdots = l_{k - m} = 0$, the summands of the form (\ref{jij})
  with factor $N^k$ are
  \[ N^k \left( \frac{k!}{2 m! (m + 1) ! (k - m - 2) !} \partial_1^m \left(
     \left( \partial_1 \left( \frac{1}{N} \log \text{} f \right) \right)^{k -
     m - 2} \partial_1^2 \left( \frac{1}{N} \log \text{} f \right) \right)
     (0^N) + \omicron (1) \right), \]
  where $m \in \{0, \ldots, k - 2\}$. On the other hand, for $l_1 + \cdots +
  l_{k - m} = k - m$, i.e. $l_1 = k - m$ and $l_2 = \cdots = l_{k - m} = 0$,
  we also obtain appropriate summands of the form (\ref{jij}). These are
  \[ - N^k \left( \sum_{i = 1}^m i \right) \left( \frac{k!}{m! (m + 1) ! (k -
     m) !} \partial_1^m \left( \partial_1 \left( \frac{1}{N} \log \text{} f
     \right) \right)^{k - m} (0^N) + \omicron (1) \right) \]
  \[ \text{{\hspace{10em}}} = - N^k \left( \frac{k!}{2 (m - 1) !m! (k - m) !}
     \partial_1^m \left( \partial_1 \left( \frac{1}{N} \log \text{} f \right)
     \right)^{k - m} (0^N) + \omicron (1) \right), \]
  where $m \in \{0, \ldots, k - 1\}$. For $l_1 + \ldots l_{k - m} < k - m - 1$
  we obtain summands which have powers of $N$ of lower degree. By
  the above, we deduce that
  \[ \lim_{N \rightarrow \infty} M_{1, N} (\Psi) = \sum_{m = 0}^{k - 2}
     \frac{k!}{2 m! (m + 1) ! (k - m - 2) !}  \left.  \frac{\mathd^m}{\mathd
     x^m} ((\Psi' (x))^{k - m - 2} \Psi'' (x)) \right|_{x = 0} \]
  \[ \text{{\hspace{13em}}} - \sum_{m = 1}^{k - 1} \frac{k!}{2 (m - 1) !m! (k
     - m) !}  \left. \frac{\mathd^m}{\mathd x^m} ((\Psi' (x))^{k - m})
     \right|_{x = 0} = 0. \]
  This proves the claim.
\end{proof}

\begin{remark}
  Ergodic unitarily invariant matrices are included in the context of Theorem
  \ref{Einaiiagapi}. An example is a
  Hermitian random matrix $A$ of size $N$ that satisfies the
  relation
  \begin{equation}
    \mathbb{E} \left[ \exp (\tmop{Tr} (H \text{} A)) \right] = \prod_{b \in
    \tmop{Spec} (H)} \exp (N \Psi (b) + \Phi (b)), \text{\quad for every} \
    H \in H (N).
  \end{equation}
  In that case we have
  \[ \frac{1}{N}  \left( \log \mathbb{E} [H \text{} C (x_1, \ldots, x_N ; \lambda_1 (A),
     \ldots, \lambda_N (A))] - N \sum_{i = 1}^N \Psi (x_i) \right) = \sum_{i = 1}^N
     \Phi (x_i). \]
  In comparison with a more general setting of Theorem \ref{Einaiiagapi} , these matrices
  provide the simplest case since they give an expansion
  \[ \frac{1}{N^{k + 1}} \mathbb{E} \left[ \sum_{i = 1}^N \lambda_i^k (A)
     \right] = \sum_{i = 0}^k m_i \frac{1}{N^i}, \]
  where $m_0, \ldots, m_k$ do not depend from $N$. The numbers $m_0, m_1$ are
  given by (\ref{topg}) and (\ref{G}) respectively. Thus, for GUE and Wishart
  matrices the $1 / N$ correction will be equal to zero.
\end{remark}

In the context of Theorem \ref{Einaiiagapi}, we saw that function $\Psi'$ plays a crucial role
in the determination of the limit measure, since it gives us the
$R$-transform. Now we would like to understand the role of function $\Phi$
in the $1 / N$ correction. This role is quite similar from the
infinitesimal free probability side, in the following sense: For a random
matrix $A$ that satisfies the assumptions of Theorem \ref{Einaiiagapi}, we consider the
sequence of non-commutative probability spaces $(\mathbb{C} \langle \mathbf{x}
\rangle, \varphi_N)$, where
\begin{equation}
  \varphi_N (P) \assign \frac{1}{N} \mathbb{E} [\tmop{Tr} (P (N^{- 1}
  A))], \text{\quad for every} \ P \in \mathbb{C} \langle \mathbf{x} \rangle
  . \label{pornovioS}
\end{equation}
We also define a non-commutative probability space $(\mathbb{C} \langle
\mathbf{x} \rangle, \varphi)$ such that
\begin{equation}
  \varphi (P) = \int_{\mathbb{R}} P (t) \tmmathbf{\mu} (d \text{} t),
  \text{\quad for every} \ P \in \mathbb{C} \langle \mathbf{x} \rangle,
  \label{takitsan}
\end{equation}
where $\tmmathbf{\mu}$ is the probability measure with moments given by
(\ref{topg}). Theorem \ref{Einaiiagapi} implies the relation
\begin{equation}
  \varphi_N (P) = \varphi (P) + \frac{1}{N} \varphi' (P) + \omicron (N^{-
  1}), \label{geniamas}
\end{equation}
where $\varphi'$ is a linear functional which sends the identity of
$\mathbb{C} \langle \mathbf{x} \rangle$ to zero and $\varphi'
(\mathbf{x}^k)$ are given by (\ref{G}). The free cumulants
$(\tmmathbf{\kappa}_n (\mathbf{x}, \ldots, \mathbf{x}))_{n \in \mathbb{N}}$ of
$(\mathbb{C} \langle \mathbf{x} \rangle, \varphi)$ are given by
\begin{equation}
  \tmmathbf{\kappa}_n (\mathbf{x}, \ldots, \mathbf{x}) = \frac{\Psi^{(n)}
  (0)}{(n - 1) !}, \text{\quad for every} \ n \in \mathbb{N}. \label{chen}
\end{equation}
Similarly, function $\Phi$ allows us to compute explicitly the infinitesimal
free cumulants $(\tmmathbf{\kappa}'_n(\mathbf{x}, \ldots, \mathbf{x}))_{n
\in \mathbb{N}}$ of $(\mathbb{C} \langle \mathbf{x} \rangle, \varphi,
\varphi')$. In the next lemma we show that
\begin{equation}
  \tmmathbf{\kappa}'_n (\mathbf{x}, \ldots, \mathbf{x}) = \frac{\Phi^{(n)}
  (0)}{(n - 1) !}, \text{\quad for every} \ n \in \mathbb{N}. \label{alex}
\end{equation}

\begin{lemma}
  \label{24}Let $\Psi, \Phi$ be two infinitely differentiable at 0 functions
  and consider sequences $(c_n)_{n \in \mathbb{N}}, (c'_n)_{n \in
  \mathbb{N}}$, where
  \[ c_n \assign \frac{\Psi^{(n)} (0)}{(n - 1) !} \text{\quad and\quad}
     c'_n \assign \frac{\Phi^{(n)} (0)}{(n - 1) !}, \text{\quad for
     every \ } n \in \mathbb{N}. \]
  Then we have
  \begin{equation}
    \sum_{m = 0}^{k - 1} \frac{k!}{m! (m + 1) ! (k - m - 1) !}  \left.
    \frac{\mathd^m}{\mathd x^m} ((\Psi' (x))^{k - m - 1} \Phi' (x))
    \right|_{x = 0} = \sum_{\pi \in \tmop{NC} (k)} \sum_{V \in \pi}
    c'_{|V|} \prod_{W \in \pi; W \ne V} c_{|W|} . \label{v}
  \end{equation}
\end{lemma}

\begin{proof}
  By Leibniz rule we have that the left hand side of (\ref{v}) is equal to
  \begin{equation}
    \sum_{m = 0}^{k - 1} \sum_{l_1 + \cdots + l_{k - m} = m} \frac{k!}{(m + 1)
    ! (k - m - 1) ! \prod_{i = 1}^{k - m} l_i !} \prod_{i = 1}^{k - m - 1}
    \Psi^{(l_i + 1)} (0) \Phi^{(l_{k - m} + 1)} (0) \label{zwzw} .
  \end{equation}
  We will show that the correspondence that we made in the proof of Lemma \ref{lem:mom-cum}
  between ($k - m$) -tuples $(l_1, \ldots, l_{k - m})$ and non-crossing
  partitions $\pi \in \tmop{NC} (k)$ allows us to prove the claim. For fixed $m \in
  \{0, \ldots, k - 1\}$, consider non-negative integers $\lambda_1,
  \ldots, \lambda_{k - m}$ such that $\lambda_1 + \cdots + \lambda_{k -
  m} = m$, and assume that $\{\lambda_1, \ldots, \lambda_{k - m} \} = \{\nu_1,
  \ldots, \nu_a \}$, where $\nu_i \neq \nu_j$, for every $i \neq j$. We also
  define for every $i \in \{1, \ldots, a\}$ the quantity
  \[ r_i \assign |\{j = 1, \ldots, k - m \of \lambda_j = \nu_i \}|, \]
  which implies that
  \[ r_1 \nu_1 + \cdots + r_a \nu_a = m \text{\quad and\quad} r_1 + \cdots +
     r_a = k - m. \]
  For fixed $j \in \{1, \ldots, a\}$, the number of times that the term
  \[ \frac{k!}{(m + 1) ! (k - m - 1) ! \prod_{i = 1}^{k - m} \lambda_i !}
     \Phi^{(\nu_j + 1)} (0) (\Psi^{(\nu_j + 1)} (0))^{r_j - 1}
     \prod_{\underset{i \neq j}{i = 1}}^a (\Psi^{(\nu_i + 1)} (0))^{r_i} \]
  will contribute to the latter sum of (\ref{zwzw}) is equal to the number of
  different ways that we can cover $k - m - 1$ points on a line segment with
  $r_1$ elements $\nu_1$, {\textdots}, $r_{j - 1}$
  elements $\nu_{j - 1}$, $r_j - 1$ elements $\nu_j$, $r_{j + 1}$
  elements $\nu_{j + 1}$, {\textdots}, $r_a$ elements
  $\nu_a$. This is equal to 
  \[ \frac{(k - m - 1) !}{r_1 ! \ldots r_{j - 1} ! (r_j - 1) !r_{j + 1} !
     \ldots r_a !} . \]
  Thus, we see that the $(k - m)$-tuple $(\lambda_1, \ldots, \lambda_{k - m})$
  will contribute to the latter sum of (\ref{zwzw}) a term
  \[ \frac{k!}{(m + 1) ! \prod_{j = 1}^a (\nu_j !)^{r_j}} \sum_{i = 1}^a
     \frac{\Phi^{(\nu_i + 1)} (0) (\Psi^{(\nu_i + 1)} (0))^{r_i - 1}}{r_1 !
     \ldots r_{i - 1} ! (r_i - 1) !r_{i + 1} ! \ldots r_a !}
     \prod_{\underset{j \neq i}{j = 1}}^a (\Psi^{(\nu_j + 1)} (0))^{r_j} \]
  \[ \text{} = \frac{k!}{(m + 1) ! \prod_{i = 1}^a r_i !} \sum_{i = 1}^a r_i
     c'_{\nu_i + 1} (c_{\nu_i + 1})^{r_i - 1} \prod_{\underset{j \neq i}{j
     = 1}}^a (c_{\nu_j + 1})^{r_j} . \]
  Let $\pi = \{V_1, \ldots, V_{k - m} \} \in \tmop{NC} (k)$ be a partition
  which have $r_1$ blocks with $\nu_1 + 1$ elements,{\textdots}, $r_a$ blocks
  with $\nu_a + 1$ elements. Such a partition gives a following contribution to the
  right hand side of (\ref{v})
  \[ \sum_{i = 1}^{k - m} c'_{|V_i |} \prod_{\underset{j \neq i}{j = 1}}^{k
     - m} c_{|V_j |} = \sum_{i = 1}^a r_i c'_{\nu_i + 1} (c_{\nu_i +
     1})^{r_i - 1} \prod_{\underset{j \neq i}{j = 1}}^a (c_{\nu_j + 1})^{r_j}
     . \]
  Since the number of such partitions is equal to $\left( (m + 1) ! \prod_{i =
  1}^a r_i ! \right)^{- 1} k!$, the claim has been proven.
\end{proof}

In the rest of this section we will focus on specific examples of random
matrix models, which will allow us to better understand the functional
$\varphi'$ defined in (\ref{geniamas}). Our goal is to compute the signed
measure on $\mathbb{R}$ that determines it in the same way that the
functional $\varphi$ is determined by relation (\ref{takitsan}) from the
probability measure $\tmmathbf{\mu}$.

\begin{example}
  \label{ZNisxoli}We consider the Hermitian random matrix of size $N$, $A = N
  \text{} \cdot A_1 + \sqrt{N} \cdot A_2$, where $A_1, A_2$ are GUE matrices. 
  We also assume that $A_1, A_2$ are independent. Then we have
  for every $H \in H (N)$,
  \[ \mathbb{E} \exp (\tmop{Tr} \left( H \text{} A \right))=
     \prod_{x \in \tmop{Spec} (H)} \exp \left( N \frac{x^2}{2} + \frac{x^2}{2}
     \right), \]
  which implies that the $1 / N$ correction of the matrix $A_1 +
  \frac{1}{\sqrt{N}} A_2$ is given by
  \[ \lim_{N \rightarrow \infty} N \left( \frac{1}{N^{k + 1}} \mathbb{E} \left( 
     \sum_{i = 1}^N \lambda_i^k (A) \right) - \sum_{m = 0}^{k - 1}
     \frac{k!}{m! (m + 1) ! (k - m) !}  \left. \frac{\mathd^m}{\mathd x^m}
     (x^{k - m}) \right|_{x = 0} \right) \]
  \[ = \sum_{m = 0}^{k - 1} \frac{k!}{m! (m + 1) ! (k - m - 1) !}  \left.
     \frac{\mathd^m}{\mathd x^m} (x^{k - m}) \right|_{x = 0} . \]
  Using Cauchy formula, we deduce that
  \begin{multline*}
    \sum_{m = 0}^{k - 1} \frac{k!}{m! (m + 1) ! (k - m - 1) !}  \left.
    \frac{\mathd^m}{\mathd x^m} (x^{k - m}) \right|_{x = 0} = \frac{1}{2
    \mathpi \mathi} \sum_{m = 0}^{k - 1} \binom{k}{m + 1} \oint_{|z| = 1}
    \frac{z^{k - m}}{z^{m + 1}} d \text{} z\\
    = \frac{1}{2 \mathpi \mathi} \oint_{|z| = 1} z^{k + 1} \sum_{m = 1}^k
    \binom{k}{m} \left( \frac{1}{z^2} \right)^m d \text{} z
    = \frac{1}{2 \mathpi \mathi} \oint_{|z| = 1} z^{k + 1} \left( \left(
    1 + \frac{1}{z^2} \right)^k - 1 \right) d \text{} z
    = \frac{1}{2 \mathpi \mathi} \oint_{|z| = 1} z \left( z + \frac{1}{z}
    \right)^k d \text{} z.
  \end{multline*}
  Using the change of variables $u = 2 \cos \text{} t$ in order to compute the
  last contour integral, we obtain the formula
  \[ \sum_{m = 0}^{k - 1} \frac{k!}{m! (m + 1) ! (k - m - 1) !}  \left.
     \frac{\mathd^m}{\mathd x^m} (x^{k - m}) \right|_{x = 0} = \frac{1}{2
     \mathpi} \int_{- 2}^2 u^k  \frac{u^2 - 2}{\sqrt{4 - u^2}} d \text{} u, \]
  where
  \[ \tmmathbf{\mu}' (d \text{} u) = \frac{1}{2 \mathpi} \tmmathbf{1}_{(- 2,
     2)} (u) \frac{u^2 - 2}{\sqrt{4 - u^2}} d \text{} u \]
  is a signed measure of total mass $0$. 
\end{example}

\begin{example}
\label{ex:Wishart-1}
  We consider the Hermitian random matrix of size $N$, $A = N \cdot A_1 +
  \sqrt{N} \cdot A_2$, where $A_1 = \frac{1}{N} X \text{} X^{\ast}$ is a
  Wishart matrix. More precisely, $X$ is a $N \times M$ matrix with
  independent and standard complex Gaussian entries. We also assume that $A_2$
  is a GUE matrix and $A_1, A_2$ are independent. Then, for every $H
  \in H (N)$, we have
  \[ \mathbb{E} [\exp (\tmop{Tr} (H \text{} A))] = \exp \left( M \sum_{x \in
     \tmop{Spec} (H)} \log (1 - x)^{- 1} + \sum_{x \in \tmop{Spec} (H)}
     \frac{x^2}{2} \right), \]
  which implies that the $1 / N$ correction of the matrix $A_1 +
  \frac{1}{\sqrt{N}} A_2$ is given by
  \[ \lim_{M, N \rightarrow \infty \text{, } M - \lambda N \to 0} N
     \left( \frac{1}{N^{k + 1}} \mathbb{E} \left[ \sum_{i = 1}^N \lambda_i^k
     (A) \right] - \sum_{m = 0}^{k - 1} \frac{k!}{m! (m + 1) ! (k - m) !} 
     \left. \frac{\mathd^m}{\mathd x^m} \left( \frac{\lambda}{1 - x}
     \right)^{k - m} \right|_{x = 0} \right) \]
  \begin{equation}
    = \sum_{m = 0}^{k - 1} \frac{k!}{m! (m + 1) ! (k - m - 1) !}  \left.
    \frac{\mathd^m}{\mathd x^m} \left( \left( \frac{\lambda}{1 - x} \right)^{k
    - m - 1} x \right) \right|_{x = 0} \label{snik} .
  \end{equation}
  Similarly with the previous example, we calculate the correction measure.
  Using that
  \[ \left. \frac{\mathd^m}{\mathd x^m} \left( \left( \frac{1}{1 - x}
     \right)^{k - m - 1} x \right) \right|_{x = 0} \]
  \[ = \sum_{n = 0}^m \binom{m}{n} \left. \frac{\mathd^n}{\mathd x^n} \left(
     \frac{1}{1 - x} \right)^{k - m - 1} \frac{\mathd^{m - n}}{\mathd x^{m -
     n}} (x) \right|_{x = 0} = \binom{m}{m - 1} \left. \frac{\mathd^{m -
     1}}{\mathd x^{m - 1}} (1 + x)^{k - 3} \right|_{x = 0}, \]
  for $m \geq 1$ and $k \geq 3$, we obtain that (\ref{snik}) is equal to
  \[ \frac{1}{2 \mathpi \mathi} \sum_{m = 1}^{k - 1} \lambda^{k - m - 1}
     \binom{k}{m + 1} \oint_{|z| = \frac{1}{\sqrt{\lambda}}} \frac{(1 + z)^{k
     - 3}}{z^m} d \text{} z = \frac{1}{2 \mathpi \mathi} \oint_{|z| =
     \frac{1}{\sqrt{\lambda}}} z (1 + z)^{k - 3} \lambda^k \sum_{m = 2}^k
     \binom{k}{m} \left( \frac{1}{\lambda z} \right)^m d \text{} z \]
  \[ = \frac{1}{2 \mathpi \mathi} \oint_{|z| = \frac{1}{\sqrt{\lambda}}}
     \frac{z}{(1 + z)^3} \left( \lambda + 1 + \lambda z + \frac{1}{z}
     \right)^k d \text{} z \]
  for every $k \geq 3$. This implies that the corresponding linear functional
  $\varphi_{\lambda}'$, defined in (\ref{geniamas}), can be described by the
  measure
  \begin{equation}
    \tmmathbf{\mu}'_{\lambda} (d \text{} u) = \frac{1}{2 \mathpi}
    \tmmathbf{1}_{(\lambda + 1 - 2 \sqrt{\lambda}, \lambda + 1 + 2
    \sqrt{\lambda})} (u) \frac{\lambda u^2 - 2 \lambda^2 u + (\lambda + 1) u +
    (\lambda - 1)^3}{u^3 \sqrt{(\lambda + 1 + 2 \sqrt{\lambda} - u) (u -
    \lambda - 1 + 2 \sqrt{\lambda})}} d \text{} u \label{einaipioomorfi} .
  \end{equation}
  More precisely, for $\lambda > 1$, we have for every $P \in \mathbb{C}
  \langle \mathbf{x} \rangle$
  \[ \varphi'_{\lambda} (P) = \int_{\mathbb{R}} P (t)
     \tmmathbf{\mu}'_{\lambda} (d \text{} t) . \]
  For $\lambda \leq 1$ comparing the $k$-th moment of
  $\tmmathbf{\mu}'_{\lambda}$ with (\ref{snik}), for $k = 0, 1, 2$, we see
  that we also have to add derivatives of $\delta \left( 0 \right)$, in order to obtain a
  formula for $\varphi'_{\lambda}$ that will also give $\varphi'_{\lambda}
  (\tmmathbf{1}), \varphi'_{\lambda} (\mathbf{x})$ and $\varphi'_{\lambda}
  (\mathbf{x}^2)$. Note that for $\lambda = 1$ the total mass of
  $\tmmathbf{\mu}'_{\lambda}$ is infinity.
\end{example}

Other examples of interest are finite-rank perturbations of certain random
matrices. It has been shown that these random matrix models fit in the
framework of infinitesimal free probability, see \cite{Sh}. In parallel with the
previous works in this direction, we study finite-rank pertubations of ergodic
unitarily invariant matrices, incorporating them into the context of Theorem
\ref{Einaiiagapi}. For this purpose, we need the following result about their Harish-Chandra integral, which is due to \cite{OV}.

\begin{lemma}
  Let $\theta_1, \ldots, \theta_l \in \mathbb{R}$, and $l$ is fixed. Then the Harish-Chandra integral satisfies, for every fixed $r \in \mathbb{N}$,
  \[ \lim_{N \rightarrow \infty} \log \text{} H \text{} C (x_1, \ldots, x_r,
     0^{N - r} ; N \theta_1, \ldots, N \theta_l, 0^{N - l}) = \sum_{i = 1}^r
     \sum_{j = 1}^l \log (1 - \theta_j x_i)^{- 1}, \]
  where the above convergence is uniform in a neighborhood of $0^r$.
\end{lemma}

The previous lemma allows to obtain explicit formulas for the $1 / N$
correction of certain random matrices $A$ that satisfy the relation
(\ref{xereis}). More precisely, considering $E_{i, j}$ to be the matrix with
unit in the $(i, j)$-th coordinate and all the other entries zero, Theorem \ref{Einaiiagapi} 
implies for the matrix $B = A + \sum_{i = 1}^l N \theta_i E_{i, i}$ one has
\[ \lim_{N \rightarrow \infty} N \left( \mathbb{E} \left[ \frac{1}{N^{k + 1}}
   \sum_{i = 1}^N \lambda_i^k (B) \right] - \sum_{m = 0}^{k - 1} \frac{k!}{m!
   (m + 1) ! (k - m) !}  \left. \frac{\mathd^m}{\mathd x^m} (\Psi' (x))^{k
   - m} \right|_{x = 0} \right) \]
\begin{equation}
  = \sum_{i = 1}^l \sum_{m = 0}^{k - 1} \frac{k!}{m! (m + 1) ! (k - m - 1) !} 
  \left. \frac{\mathd^m}{\mathd x^m} \left( (\Psi' (x))^{k - m - 1}
  \frac{\theta_i}{1 - \theta_i x} \right) \right|_{x = 0}, \label{afrikane}
\end{equation}
where $\theta_1, \ldots, \theta_l \in \mathbb{R}$. We examine below specific
cases for the random matrix $A$ that allow to relate (\ref{afrikane}) with moments of a
signed measure. As we see in the examples below, the terms
(\ref{afrikane}) are explicit enough in order to determine the corresponding
signed measure for finite rank perturbations of ergodic unitarily invariant
matrices.

\begin{example}
  \label{Emeispamekorifi}Consider the Hermitian random matrix of size $N$, $B
  = N \cdot A + N \theta  E_{1, 1}$ where $A$ is a GUE matrix. Relation
  (\ref{afrikane}) implies that the $1 / N$ correction of $A + \theta E_{1,
  1}$ is given by
  \begin{equation}
    \sum_{m = 0}^{k - 1} \frac{k!}{m! (m + 1) ! (k - m - 1) !}  \left.
    \frac{\mathd^m}{\mathd x^m} \left( x^{k - m - 1} \frac{\theta}{1 - \theta
    x} \right) \right|_{x = 0} \label{gringo}
  \end{equation}
  \[ = \frac{1}{2 \mathpi \mathi} \sum_{m = 0}^{k - 1} \binom{k}{m + 1}
     \oint_{|z| = 1} \frac{\theta z^k}{1 - \theta z} \left( \frac{1}{z^2}
     \right)^{m + 1} d \text{} z = \frac{1}{2 \mathpi \mathi} \oint_{|z| = 1}
     \frac{\theta z^k}{1 - \theta z} \sum_{m = 1}^k \binom{k}{m} \left(
     \frac{1}{z^2} \right)^m d \text{}  z \]
  \[ = \frac{1}{2 \mathpi \mathi} \oint_{|z| = 1} \frac{\theta z^k}{1 - \theta
     z} \left( \left( 1 + \frac{1}{z^2} \right)^k - 1 \right) d \text{} z =
     \frac{1}{2 \mathpi \mathi} \oint_{|z| = 1} \frac{\theta}{1 - \theta z}
     \left( z + \frac{1}{z} \right)^k d \text{} z, \]
  for $| \theta | < 1$. Making the change of variables $u = 2 \cos \text{} t$
  on the last contour integral, we deduce that
  \[ \sum_{m = 0}^{k - 1} \frac{k!}{m! (m + 1) ! (k - m - 1) !}  \left.
     \frac{\mathd^m}{\mathd x^m} \left( x^{k - m - 1} \frac{\theta}{1 - \theta
     x} \right) \right|_{x = 0} = - \frac{1}{2 \pi} \int_{- 2}^2 \frac{\theta
     (u - 2 \theta)}{(\theta (u - \theta) - 1) \sqrt{4 - u^2}} u^k d \text{} u,
  \]
  for every $k \geq 1$. The above equality gives a characterization of the $1
  / N$ correction of $A + \theta E_{1, 1}$ via a signed measure for the case
  $| \theta | < 1$, in the sense that (\ref{gringo}) is given by the $k$-th
  moment of
  \[ \tmmathbf{\mu}'_{\theta} (d \text{} u) = - \frac{1}{2 \mathpi}
     \tmmathbf{1}_{(- 2, 2)} (u) \frac{\theta (u - 2 \theta)}{(\theta (u -
     \theta) - 1) \sqrt{4 - u^2}} d \text{} u, \]
  which has total mass zero. In order to get an integral representation that
  will lead to the signed measure, we write (\ref{gringo}) in the form
  \[ \sum_{m = 0}^{k - 1} \frac{1}{m!} \binom{k}{m + 1} \left.
     \frac{\mathd^m}{\mathd x^m} \left( \frac{\theta}{1 - \theta x} \left(
     x^{k - m - 1} - \left( \frac{1}{\theta} \right)^{k - m - 1} \right)
     \right) \right|_{x = 0} + \left( \theta + \frac{1}{\theta} \right)^k -
     \frac{1}{\theta^k}, \]
  where by Cauchy formula the sum is equal to
  \[ \frac{1}{2 \mathpi \mathi} \sum_{m = 0}^{k - 1} \binom{k}{m + 1}
     \oint_{|z| = 1} \frac{\theta}{z^{m + 1} (1 - \theta z)} \left( z^{k - m -
     1} - \frac{1}{\theta^{k - m - 1}} \right) d \text{} z \]
  \[ = \frac{1}{2 \mathpi \mathi} \oint_{|z| = 1} \left( \frac{\theta z^k}{1 -
     \theta z} \sum_{m = 1}^k \binom{k}{m} \frac{1}{z^{2 m}} -
     \frac{\theta}{\theta^k (1 - \theta z)} \sum_{m = 1}^k \binom{k}{m}
     \frac{\theta^{m + 1}}{z^{m + 1}} \right) d \text{} z \]
  \[ = \frac{1}{2 \mathpi \mathi} \oint_{|z| = 1} \left( \frac{\theta}{1 -
     \theta z} \left( z + \frac{1}{z} \right)^k - \frac{\theta z^k}{1 - \theta
     z} - \frac{\theta}{1 - \theta z} \left( \frac{1}{\theta} + \frac{1}{z}
     \right)^k + \frac{\theta}{\theta^k (1 - \theta z)} \right) d \text{} z \]
  \[ = \frac{1}{2 \mathpi \mathi} \oint_{|z| = 1} \frac{\theta}{1 - \theta z}
     \left( \left( z + \frac{1}{z} \right)^k - \left( \frac{1}{\theta} +
     \frac{1}{z} \right)^k \right) d \text{} z. \]
  In the last contour integral it is visible that the $k$-th moment of the
  signed measure will emerge from the function $\frac{\theta}{1 - \theta z}
  \left( z + \frac{1}{z} \right)^k$. Thus, using that
  \[ \oint_{|z| = 1} \frac{\theta}{1 - \theta z} \left( \left( \frac{1}{z} +
     \frac{1}{\theta} \right)^k - \left( \theta + \frac{1}{\theta} \right)^k
     \right) d \text{} z = \sum_{n = 0}^{k - 1} \frac{1}{2 \mathpi \mathi}
     \oint_{|z| = 1} \frac{\theta}{z} \left( \frac{1}{z} + \frac{1}{\theta}
     \right)^n \left( \theta + \frac{1}{\theta} \right)^{k - 1 - n} d \text{}
     z \]
  \[ = \sum_{n = 0}^{k - 1} \frac{1}{\theta^{n - 1}} \left( \theta +
     \frac{1}{\theta} \right)^{k - 1 - n} = \left( \theta + \frac{1}{\theta}
     \right)^k - \frac{1}{\theta^k}, \]
  we obtain, for $\theta \in \mathbb{R}$, that (\ref{gringo}) is equal to
  \[ \frac{1}{2 \mathpi \mathi} \oint_{|z| = 1} \frac{\theta}{1 - \theta z}
     \left( \left( z + \frac{1}{z} \right)^k - \left( \theta +
     \frac{1}{\theta} \right)^k \right) d \text{} z =\tmmathbf{1}_{| \theta |
     \geq 1} \left( \theta + \frac{1}{\theta} \right)^k - \frac{1}{2 \mathpi}
     \int_{- 2}^2 \frac{\theta (u - 2 \theta) u^k}{(\theta (u - \theta) - 1)
     \sqrt{4 - u^2}} d \text{} u. \]
  By the above, we deduce that the $1 / N$ correction of the matrix $A +
  \sum_{i = 1}^l \theta_i E_{i, i}$, in the sense of (\ref{geniamas}), is given
  by a linear functional $\varphi' : \mathbb{C} \langle \mathbf{x} \rangle
  \rightarrow \mathbb{C}$, where for every polynomial $P$,
  \[ \varphi' (P) = \sum_{i = 1}^l \int_{\mathbb{R}} P (t)
     \tmmathbf{\mu}'_{\theta_i} (d \text{} t), \text{} \]
  and for every $i = 1, \ldots, l$, $\tmmathbf{\mu}'_{\theta_i}$ are
  signed measures on $\mathbb{R}$ of total mass zero given by
  \[ \tmmathbf{\mu}'_{\theta_i} (d \text{} t) =\tmmathbf{1}_{| \theta_i |
     \geq 1} \delta \left( \theta_i + \theta_i^{- 1} \right) (d \text{} t) - \frac{1}{2
     \pi} \tmmathbf{1}_{(- 2, 2)} (t) \frac{\theta_i (t - 2
     \theta_i)}{(\theta_i (t - \theta_i) - 1) \sqrt{4 - t^2}} d \text{} t. \]
  This repeats a result of \cite{Sh}. The delta measures involved in the correction measure illustrate the Baik-Ben Arous-Peche phase transition, first studied in \cite{BBP}. 
\end{example}

\begin{example}
\label{ex:one-pertub}
  Consider the Hermitian random matrix of size $N$ defined by $A = N \cdot A_1 + N
  \theta E_{1, 1}$, where $A_1$ is a Wishart matrix as in Example \ref{ex:Wishart-1}. Relation
  (\ref{afrikane}) implies that the $k$-th moment of the $1 / N$ correction measure of $A_1 +     \theta E_{1,1}$ is given by
  \begin{equation}
    \sum_{m = 0}^{k - 1} \frac{k!}{m! (m + 1) ! (k - m - 1) !}  \left.
    \frac{\mathd^m}{\mathd x^m} \left( \left( \frac{\lambda}{1 - x} \right)^{k
    - m - 1} \frac{\theta}{1 - \theta x} \right) \right|_{x = 0} .
    \label{bonasera}
  \end{equation}
  In order to use an integral representation for the derivative that will
  simplify the computation of the signed measure, we write
  \[ \left. \frac{\mathd^m}{\mathd x^m} \left( \left( \frac{1}{1 - x}
     \right)^{k - m - 1} \frac{\theta}{1 - \theta x} \right) \right|_{x = 0} =
     \sum_{n = 0}^m \frac{m!}{n!} \theta^{m - n + 1} \left.
     \frac{\mathd^n}{\mathd x^n} (1 + x)^{k - m + n - 2} \right|_{x = 0} \]
  \[ = \sum_{n = 0}^m \frac{m! \theta^{m - n + 1}}{2 \mathpi \mathi}
     \oint_{|z| = \frac{1}{\sqrt{\lambda}}} \frac{(1 + z)^{k - m + n -
     2}}{z^{n + 1}} d \text{} z = \frac{m! \theta}{2 \mathpi \mathi}
     \oint_{|z| = \frac{1}{\sqrt{\lambda}}} \frac{(1 + z)^{k - m - 2}}{z}
     \sum_{n = 0}^m \left( 1 + \frac{1}{z} \right)^n \theta^{m - n} d \text{}
     z \]
  \[ = \frac{m! \theta}{2 \mathpi \mathi} \oint_{|z| =
     \frac{1}{\sqrt{\lambda}}} \frac{(1 + z)^{k - m - 2}}{(1 - \theta) z + 1}
     \left( \left( 1 + \frac{1}{z} \right)^{m + 1} - \theta^{m + 1} \right) d
     \text{} z \]
  \[ = \frac{m! \theta}{2 \mathpi \mathi} \oint_{|z| =
     \frac{1}{\sqrt{\lambda}}} \frac{(1 + z)^{k - 1}}{(1 - \theta) z + 1}
     \left( \frac{1}{z^{m + 1}} - \left( \frac{\theta}{1 + z} \right)^{m + 1}
     \right) d \text{} z, \]
  for $k \geq 2$. Thus, (\ref{bonasera}) is equal to
  \[ \frac{\theta \lambda}{2 \mathpi \mathi} \oint_{|z| =
     \frac{1}{\sqrt{\lambda}}} \frac{(\lambda + \lambda z)^{k - 1}}{(1 -
     \theta) z + 1} \sum_{m = 0}^{k - 2} \binom{k}{m + 1} \left( \left(
     \frac{1}{\lambda z} \right)^{m + 1} - \left( \frac{\theta}{\lambda (1 +
     z)} \right)^{m + 1} \right) d \text{} z + \theta^k \]
  \[ = \frac{\theta}{2 \mathpi \mathi} \oint_{|z| = \frac{1}{\sqrt{\lambda}}}
     \frac{(1 + z)^{- 1}}{(1 - \theta) z + 1} \left( \left( \lambda + 1 +
     \lambda z + \frac{1}{z} \right)^k - \left( 1 + \frac{1}{z} \right)^k -
     (\lambda + \theta + \lambda z)^k + \theta^k \right) d \text{} z +
     \theta^k . \]
  Similarly with Example \ref{ex:Wishart-1}, in the above contour integration the function
  \[ \frac{(1 + z)^{- 1}}{(1 - \theta) z + 1} \left( \lambda + 1 + \lambda z +
     \frac{1}{z} \right)^k \]
  will give the density of a signed measure while the remaining terms will
  contribute to point masses. We will only treat the case $\lambda = 1$
  because it illustrates the procedure sufficiently. By Cauchy formula, a
  straightforward computation shows that
  \[ \frac{\theta}{2 \mathpi \mathi} \oint_{|z| = 1} \frac{(1 + z)^{- 1}}{(1 -
     \theta) z + 1} \left( 1 + \frac{1}{z} \right)^k d \text{} z =
     \left\{\begin{array}{l}
       \theta^k \text{, for } | \theta - 1| < 1 \text{}\\
       0 \text{, for } | \theta - 1| > 1
     \end{array}\right. \]
  and
  \[ \frac{\theta}{2 \mathpi \mathi} \oint_{|z| = 1} \frac{(1 + z)^{- 1}}{(1 -
     \theta) z + 1} ((\theta + 1 + z)^k - \theta^k) d \text{} z =
     \left\{\begin{array}{l}
       0 \text{, for } | \theta - 1| < 1\\
       \left( \theta + 1 + \frac{1}{\theta - 1} \right)^k - \theta^k \text{,
       for } | \theta - 1| > 1.
     \end{array}\right. \]
  As a corollary, the corresponding linear functional $\varphi'_{\theta}$
  that gives the $1 / N$ correction satisfies the relation
  \[ \varphi'_{\theta} (\mathbf{x}^k) = \frac{1}{2 \mathpi} \int_0^4
     \frac{\theta (2 - \theta)}{(1 - \theta) t + \theta^2} 
     \frac{t^k}{\sqrt{(4 - t) t}} d \text{} t +\tmmathbf{1}_{| \theta - 1| >
     1} \left( \theta + 1 + \frac{1}{\theta - 1} \right)^k, \]
  for every $k \geq 1$ and $| \theta - 1| \neq 1$.
\end{example}

\section{Infinitesimal quantized freeness}
\label{sec:Schur-free}

In this section we deal with a different setup compared to the rest of the paper. 
However, this setup is a closely related one and leads to new classes of applications. 
Instead of the study of the Harish-Chandra transform, we will study characters of the
irreducible representations of the unitary group $U (N)$, as $N \rightarrow
\infty$.

We start by recalling relevant definitions and the main result of \cite{BG}. It is well-known that all irreducible representations of $U(N)$ are parametrized by signatures(=highest weights), that is, $N$-tuples of integers
$\mathlambda_1 \geq \mathlambda_2 \geq \cdots \geq \mathlambda_N$. We denote
by $\hat{U} (N)$ the set of all signatures and by $\pi^{\lambda}$ the
irreducible representation of $U(N)$ that corresponds to $\mathlambda \in \hat{U} (N)$.
Information about a signature $\lambda$ can be encoded by a discrete
probability measure on $\mathbb{R}$:
\begin{equation}
  \tmmathbf{m}_N [\mathlambda] \assign \frac{1}{N} \sum_{i = 1}^N
  \delta \left( \frac{\mathlambda_i + N - i}{N} \right) . \label{nono}
\end{equation}
For $\mathlambda \in \hat{U} (N)$ chosen at random with respect to some
probability measure $\varrho$ on $\hat{U} (N)$, we are interested in the
asymptotic behaviour of the random measure (\ref{nono}).

We recall that for a finite-dimensional representation $\pi$ of $U (N)$, the
character of $\pi$ is the function given by $\tmmathbf{\chi}^{\pi} (U) \assign
\tmop{Tr} (\pi (U))$, for every $U \in U (N)$. Note that for a matrix $U \in U (N)$
with eigenvalues $u_1, \ldots, u_N$, one has $\tmmathbf{\chi}^{\pi} (U)
=\tmmathbf{\chi}^{\pi} (\tmop{diag} (u_1, \ldots, u_N))$, and the character will be  denoted
by $\tmmathbf{\chi}^{\pi} (u_1, \ldots, u_N)$. Moreover, for $\mathlambda =
(\mathlambda_1 \geq \cdots \geq \mathlambda_N) \in \hat{U} (N)$ we will denote
the character of $\pi^{\mathlambda}$ by $\tmmathbf{\chi}^{\mathlambda}$. Due
to Weyl (see, e.g., \cite{W}), for $U \in U (N)$ with eigenvalues $u_1, \ldots, u_N
\in \mathbb{R}$, the value $\tmmathbf{\chi}^{\mathlambda} (u_1, \ldots, u_N)$
is given by the rational Schur function, i.e., we have
\begin{equation}
  \tmmathbf{\chi}^{\mathlambda} (u_1, \ldots, u_N) = \frac{\det
  (u_i^{\mathlambda_j + N - j})}{\det (u_i^{N - j})} = \frac{\det
  (u_i^{\mathlambda_j + N - j}) }{\prod_{i < j} (u_i - u_j)} . \label{Maroko} 
\end{equation}
Note that $\tmmathbf{\chi}^{\mathlambda} (1^N)$ is the dimension of
$\pi^{\mathlambda}$.

Let $\varrho$ be a probability measure on $\hat{U} (N)$. For a fixed $\mathlambda \in
\hat{U} (N)$, the symmetric polynomial 
$$
\tmmathbf{\chi}^{\mathlambda}
(u_1, \ldots, u_N) /\tmmathbf{\chi}^{\mathlambda} (1^N)
$$
should be thought of as
the analogue of the Harish-Chandra integral $H \text{} C (u_1, \ldots, u_N ;
\lambda_1 (A), \ldots, \lambda_N (A))$ of a Hermitian random matrix of size
$N$ with fixed spectrum $\{\lambda_i (A)\}_{i = 1}^N$. An analog of the Harish-Chandra transform is the following notion. 

\begin{definition}
  For a probability measure $\varrho$ on $\hat{U} (N)$, a \textbf{Schur generating function} $S_{\varrho}^{U (N)}$ is a symmetric Laurent
  power series in $(u_1, \ldots, u_N)$ given by
  \[ S_{\varrho}^{U (N)} (u_1, \ldots, u_N) := \sum_{\mathlambda \in \hat{U}
     (N)} \varrho (\mathlambda) \frac{\tmmathbf{\chi}^{\mathlambda} (u_1,
     \ldots, u_N)}{\tmmathbf{\chi}^{\mathlambda} (1^N)} . \]
\end{definition}
For every $k \in \mathbb{N}$ we consider the
following differential operator acting on smooth functions $f$ of $N$
variables:
\[ \mathcal{D}_k^{U (N)} (f) \assign \left( \prod_{i < j} (u_i - u_j)
   \right)^{- 1} \sum_{i = 1}^N (u_i \partial_i)^k \left( \prod_{i < j} (u_i -
   u_j) f \right) . \]

\begin{proposition}
  \label{Gkofretakia}For every $\mathlambda \in \hat{U} (N)$, the character
  $\tmmathbf{\chi}^{\mathlambda}$ satisfies the relation
  \[ \mathcal{D}_k^{U (N)} (\tmmathbf{\chi}^{\mathlambda}) (u_1, \ldots, u_N)
     = \sum_{i = 1}^N (\mathlambda_i + N - i)^k \tmmathbf{\chi}^{\mathlambda}
     (u_1, \ldots, u_N) . \]
\end{proposition}

\begin{proof}
  See, e.g., \cite{BG}.
\end{proof}
The operator $\mathcal{D}_k^{U (N)}$ is quite similar to the operator
$\mathcal{D}_k$ used in previous sections. 


\begin{theorem}[Bufetov-Gorin, Theorem 5.1]
  \label{KALOKAIRI}Let $\varrho (N)$ be a sequence of probability measures on
  $\hat{U} (N)$ such that for every $r \in \mathbb{N}$ one has
  \begin{equation}
    \lim_{N \rightarrow \infty} \frac{1}{N} \log \text{} S_{\varrho (N)}^{U
    (N)} (u_1, \ldots, u_r, 1^{N - r}) = \sum_{i = 1}^r \Psi (u_i),
    \label{Kailunsegamw}
  \end{equation}
  where $\Psi$ is an analytic function in a complex neighborhood of $1$ and the above
  convergence is uniform in a complex neighborhood of $1^r$. Then the random measure
  $N^{- 1} \sum_{i = 1}^N \delta \left( \frac{\mathlambda_i + N - i}{N} \right)$, converges,
  as $N \rightarrow \infty$, in probability, in the sense of moments to a
  deterministic measure $\mathbf{m}$ on $\mathbb{R}$, whose moments are given
  by
  \begin{equation}
    \int_{\mathbb{R}} t^k \mathbf{m} (d \text{} t) = \sum_{m = 0}^k
    \frac{k!}{m! (m + 1) ! (k - m) !}  \left. \frac{\mathd^m}{\mathd x^m} (x^k
    (\Psi' (x))^{k - m}) \right|_{x = 1} . \label{mporw?}
  \end{equation}
\end{theorem}



In the setting of
irreducible representations of $U (N)$ and Theorem \ref{KALOKAIRI}, the tensor product of representations can be viewed as a discrete generalization of a summation of (random) matrices. Let $\pi_1, \pi_2$ be two finite-dimensional representations of $U (N)$, their tensor product $(\pi_1 \otimes \pi_2) (U)$,
for $U \in U (N)$, is the Kronecker product of the matrices $\pi_1 (U), \pi_2
(U)$. This implies that $\tmmathbf{\chi}^{\pi_1 \otimes \pi_2} (U)
=\tmmathbf{\chi}^{\pi_1} (U) \tmmathbf{\chi}^{\pi_2} (U)$. Assume that $\pi_1,
\pi_2$ are irreducible representations of $U (N)$ that correspond to
signatures $\mathlambda_1 (N), \mathlambda_2 (N)$, respectively. Moreover,
assuming that for every fixed $r$ and $i = 1, 2$, one has
\begin{equation}
  \lim_{N \rightarrow \infty} \frac{1}{N} \log \left(
  \frac{\tmmathbf{\chi}^{\mathlambda_i (N)} (u_1, \ldots, u_r, 1^{N -
  r})}{\tmmathbf{\chi}^{\mathlambda_i (N)} (1^N)} \right) = \Psi_i (u_1) +
  \cdots + \Psi_i (u_r), \label{SpOtIfY}
\end{equation}
uniformly in a neighborhood of $1^r$, we can include the asymptotics of
$\tmmathbf{\chi}^{\pi_1 \otimes \pi_2}$ to the context of Theorem
\ref{KALOKAIRI}. Assumption (\ref{SpOtIfY}) should be thought of as the analogue
of (\ref{xereis}) for the matrices $A_1, A_2$.



The representation $\pi_1 \otimes \pi_2$ can be decomposed into irreducibles
\[ \pi_1 \otimes \pi_2 = \bigoplus_{\mathlambda \in \hat{U} (N)}
   c_{\mathlambda} \pi^{\mathlambda}, \]
where the non-negative integers $c_{\mathlambda}$ are multiplicities. This
implies that the character of $\pi = \pi_1 \otimes \pi_2$ is equal to
$S_{\varrho^{\pi}}^{U (N)}$, where the probability measure $\varrho^{\pi}$ on
$\hat{U} (N)$ is given by
\[ \varrho^{\pi} (\mathlambda) = \frac{c_{\mathlambda} \dim
   (\pi^{\mathlambda})}{\dim (\pi_1 \otimes \pi_2)} \text{, \quad for every }
   \mathlambda \in \hat{U} (N) . \]
Hence, we have the following result.

\begin{corollary}[Bufetov-Gorin]
  \label{Erntogan}Let $\mathlambda_1 (N), \mathlambda_2 (N) \in \hat{U} (N)$
  be two sequences of signatures such that for $i = 1, 2$,
  \[ \lim_{N \rightarrow \infty} \tmmathbf{m}_N [\mathlambda_i (N)]
     =\mathbf{m}_i, \]
  where $\mathbf{m}_1, \mathbf{m}_2$ are two probability measures and the
  above convergence is weak. Moreover, let $\pi (N) = \pi^{\mathlambda_1 (N)}
  \otimes \pi^{\mathlambda_2 (N)}$. Then, for $\mathlambda (N) \in \hat{U}
  (N)$ chosen at random with respect to $\varrho^{\pi (N)}$, the random
  measures $\tmmathbf{m}_N [\mathlambda (N)]$ converge, as $N \rightarrow
  \infty$, in the sense of moments, in probability to a deterministic
  probability measure which one denotes by $\mathbf{m}_1 \otimes \mathbf{m}_2$.
\end{corollary}

This corollary is an analogue of Voiculescu's theorem for random
matrices, see Theorem \ref{Voiculescu} above. The free
convolution $\mathbf{m}_1 \boxplus \mathbf{m}_2$ from Theorem \ref{Voiculescu} is replaced by $\mathbf{m}_1 \otimes \mathbf{m}_2$ in Corollary \ref{Erntogan}. In contrast with the random matrix framework, it was
noticed in \cite{BG} that the operation $(\mathbf{m}_1, \mathbf{m}_2) \mapsto
\mathbf{m}_1 \otimes \mathbf{m}_2$ is not linearized by the $R$-transform.
Instead, it is linearized by its deformation
\[ R_{\mathbf{m}}^{q \text{} u \text{} a \text{} n \text{} t} (z) \assign
   R_{\mathbf{m}} (z) - R_{u [0, 1]} (z), \]
where $u [0, 1]$ stands for the uniform measure on $[0, 1]$. This means that
$R^{q \text{} u \text{} a \text{} n \text{} t}_{\mathbf{m}_1 \otimes
\mathbf{m}_2} (z) = R^{q \text{} u \text{} a \text{} n \text{}
t}_{\mathbf{m}_1} (z) + R_{\mathbf{m}_2}^{q \text{} u \text{} a \text{} n
\text{} t} (z)$. 
The map $R_{\mathbf{m}}^{q \text{} u \text{} a \text{} n \text{} t}$ was
introduced in \cite{BG} and it is called the quantized $R$-transform, while the operation $\mathbf{m}_1 \otimes \mathbf{m}_2$ is called the quantized free convolution
of $\mathbf{m}_1$ and $\mathbf{m}_2$.

In the following we concentrate on the $1 / N$ correction of the random
measure (\ref{nono}), as $N \rightarrow \infty$. First we focus on its
description via an explicit formula (as it was done in Theorem
\ref{Einaiiagapi} for the empirical distribution of random matrices).

\begin{theorem}
  \label{Maria!!}Let $\varrho (N)$ be a sequence of probability measures on
  $\hat{U} (N)$ such that for every finite $r$ one has
  \begin{equation}
    \lim_{N \rightarrow \infty} N \left( \frac{1}{N} \log S_{\varrho (N)}^{U
    (N)} (u_1, \ldots, u_r, 1^{N - r}) - \sum_{i = 1}^r \Psi (u_i) \right) =
    \sum_{i = 1}^r \Phi (u_i), \label{thatonegamiswtonKinezo}
  \end{equation}
  where $\Psi, \Phi$ are analytic functions in a neighborhood of $1$ and the
  above convergence is uniform in a neighborhood of $1^r$. Then, for every $k
  \in \mathbb{N}$, we have
  \begin{equation}
    \lim_{N \rightarrow \infty} N \left( \frac{1}{N} \sum_{\mathlambda \in
    \hat{U} (N)} \varrho (N) [\mathlambda] \sum_{i = 1}^N \left(
    \frac{\mathlambda_i + N - i}{N} \right)^k -\mathbf{m}_k \right)
    =\mathbf{m}'_k, \label{magapaskiegwxezw}
  \end{equation}
  where $(\mathbf{m}_k)_{k \in \mathbb{N}}$ are given by (\ref{mporw?}) and
  \begin{equation}
    \mathbf{m}'_k = \sum_{m = 0}^{k - 1} \frac{k!}{m! (m + 1) ! (k - m - 1) !}
    \left. \frac{\mathd^m}{\mathd x^m} \left( x^k \left( \Phi' (x) -
    \frac{1}{2 x} \right) (\Psi' (x))^{k - m - 1} \right) \right|_{x = 1} .
    \label{poueinai}
  \end{equation}
\end{theorem}

\begin{proof}
  The proof strategy is similar to the proof of Theorem \ref{Einaiiagapi}. Assumption
  (\ref{thatonegamiswtonKinezo}) implies that (\ref{Kailunsegamw}) holds. This
  allows us to write
  \begin{equation}
    \sum_{\mathlambda \in \hat{U} (N)} \varrho (N) [\mathlambda] \sum_{i =
    1}^N (\mathlambda_i + N - i)^k \label{pussy!}
  \end{equation}
  in the form $N^{k + 1} m_{k, N} (\Psi) + N^k m_{k, N} (\Psi, \Phi) +
  \omicron (N^k)$, where $(m_{k, N} (\Psi))_{N \in \mathbb{N}}$ converges to \eqref{mporw?}
   and we need to prove that $(m_{k, N}(\Psi, \Phi))_{n \in \mathbb{N}}$ converges to
  (\ref{poueinai}). By Proposition \ref{Gkofretakia},
  (\ref{pussy!}) is equal to $\mathcal{D}_k^{U (N)} S_{\varrho (N)}^{U (N)}
  \longdownminus_{u_i = 1}$. Due to the definition of the differential
  operator, $\mathcal{D}_k^{U (N)} S_{\varrho (N)}^{U (N)}$ can be written as
  a linear combination of terms
  \begin{equation}
    \left( \prod_{i < j} (u_i - u_j) \right)^{- 1} \sum_{i = 1}^N u_i^n
    \partial_i^n \left( \prod_{i < j} (u_i - u_j) S_{\varrho (N)}^{U (N)}
    \right), \label{Seskeftomaiotanxezw!}
  \end{equation}
  for $n = 1, \ldots, k$. However, only for $n = k - 1$ and $n=k$ the terms
  (\ref{Seskeftomaiotanxezw!}) contribute to the left-hand side of
  (\ref{magapaskiegwxezw}) and their coefficients are $\frac{k (k - 1)}{2}$
  and $1$, respectively. For $n = k$, by writing $S_{\varrho (N)}^{U (N)} = \exp
  \left( N \cdot \frac{1}{N} \log \text{} S_{\varrho (N)}^{U (N)} \right)$ and
  using the chain rule in order to compute the derivatives of $S_{\varrho
  (N)}^{U (N)}$, we obtain the expansion of (\ref{Seskeftomaiotanxezw!}) as a
  linear combination of terms of the form
  \[ \frac{u_{b_0}^k \left( \partial_{b_0} \left( \frac{1}{N} \log \text{}
     S_{\varrho (N)}^{U (N)} \right) \right)^{l_1} \ldots \left(
     \partial_{b_0}^{k - m} \left( \frac{1}{N} \log \text{} S_{\varrho (N)}^{U
     (N)} \right) \right)^{l_{k - m}}}{(u_{b_0} - u_{b_1}) \ldots (u_{b_0} -
     u_{b_m})} + \text{{\hspace{14em}}} \]
  \begin{equation}
    \text{{\hspace{10em}}} \ldots + \frac{u_{b_m}^k \left( \partial_{b_m}
    \left( \frac{1}{N} \log \text{} S_{\varrho (N)}^{U (N)} \right)
    \right)^{l_1} \ldots \left( \partial_{b_m}^{k - m} \left( \frac{1}{N} \log
    \text{} S_{\varrho (N)}^{U (N)} \right) \right)^{l_{k - m}}}{(u_{b_m} -
    u_{b_0}) \ldots (u_{b_m} - u_{b_{m - 1}})}, \label{minmemalwneis}
  \end{equation}
  where $m \in \{0, \ldots, k\}$, $b_0, \ldots, b_m \in \{1, \ldots, N\}$ and
  $l_1, \ldots, l_{k-m}$ are non-negative integers such that $l_1 + 2 l_2 + \cdots
  + (k - m) l_{k-m} = k - m$. Similarly with the proof of Theorem
  \ref{SouvlakiA}, setting $u_i = 1 + i \varepsilon$ for every $i = 1, \ldots,
  N$, and considering the $m$-th order Taylor polynomial of
  \[ F_i (\varepsilon) = \left. u_i^k \left( \partial_i \left( \frac{1}{N}
     \log \text{} S_{\varrho (N)}^{U (N)} \right) \right)^{l_1} \ldots \left(
     \partial_i^{k - m} \left( \frac{1}{N} \log \text{} S_{\varrho (N)}^{U
     (N)} \right) \right)^{l_{k - m}} \right|_{u_j = 1 + j \varepsilon}, \]
  for $i = b_0, \ldots, b_m$, we obtain that, as $\varepsilon \rightarrow 0$,
  (\ref{minmemalwneis}) is a linear combination of derivatives
  \[ \partial_{i_m} \ldots \partial_{i_1} \left[ u_i^k \left( \partial_i
     \left( \frac{1}{N} \log \text{} S_{\varrho (N)}^{U (N)} \right)
     \right)^{l_1} \ldots \left( \partial_i^{k - m} \left( \frac{1}{N} \log
     \text{} S_{\varrho (N)}^{U (N)} \right) \right)^{l_{k - m}} \right] (1^N)
     . \]
  The coefficients of these derivatives do not depend on $N$ because
  (\ref{minmemalwneis}) emerges from differentiating $k - m$ times the function
  $S_{\varrho (N)}^{U (N)}$ in (\ref{Seskeftomaiotanxezw!}), for $n = k$.
  Since (\ref{minmemalwneis}) does not depend on $b_0, \ldots, b_m$ when $u_i
  = 1 + i \varepsilon$ and $\varepsilon \rightarrow 0$, we deduce that
  \[ m_{k, N} (\Psi) = \sum_{m = 0}^k \frac{k!}{m! (m + 1) ! (k - m) !}
     \partial_1^m \left( u_1^k \left( \partial_1 \left( \frac{1}{N} \log
     \text{} S_{\varrho (N)}^{U (N)} \right) \right)^{k - m} \right) (1^N) +
     \omicron (N^{- 1}) . \]
  Assumption (\ref{thatonegamiswtonKinezo}) allows us to control $\lim_{N
  \rightarrow \infty} N (m_{k, N} (\Psi) -\mathbf{m}_k)$. Namely, we start with the equality
  \[ N \left( \partial_1^m \left( u_1^k \left( \partial_1 \left( \frac{1}{N}
     \log \text{} S_{\varrho (N)}^{U (N)} \right) \right)^{k - m} \right)
     (1^N) - \left. \frac{\mathd^m}{\mathd u^m} (u^k (\Psi' (u))^{k - m})
     \right|_{u = 1} \right) \]
  \[ = \sum_{l_0 + \cdots + l_{k - m} = m} \frac{m!}{l_0 ! \ldots l_{k - m} !}
     \frac{k!}{(k - l_0) !} N \left( \prod_{i = 1}^{k - m} \partial_1^{l_i +
     1} \left( \frac{1}{N} \log \text{} S_{\varrho (N)}^{U (N)} \right) (1^N)
     - \prod_{i = 1}^{k - m} \Psi^{(l_i + 1)} (1) \right), \]
  and due to (\ref{thatonegamiswtonKinezo}) we have
  \[ \lim_{N \rightarrow \infty} N \left( \prod_{i = 1}^{k - m}
     \partial_1^{l_i + 1} \left( \frac{1}{N} \log \text{} S_{\varrho (N)}^{U
     (N)} \right) (1^N) - \prod_{i = 1}^{k - m} \Psi^{(l_i + 1)} (1) \right) =
     \sum_{i = 1}^{k - m} \Phi^{(l_i + 1)} (1) \prod_{\underset{j \neq i}{j =
     1}}^{k - m} \Psi^{(l_j + 1)} (1) . \]
  Therefore,
  \begin{equation}
    \lim_{N \rightarrow \infty} N (m_{k, N} (\Psi) -\mathbf{m}_k) = \sum_{m =
    0}^{k - 1} \frac{k!}{m! (m + 1) ! (k - m - 1) !}  \left.
    \frac{\mathd^m}{\mathd x^m} (x^k \Phi' (x) (\Psi' (x))^{k - m - 1})
    \right|_{x = 1} \label{iposxeseis} .
  \end{equation}
  It remains to show that
  \begin{equation}
    \lim_{N \rightarrow \infty} m_{k, N} (\Psi, \Phi) = - \frac{1}{2} \sum_{m
    = 0}^{k - 1} \frac{k!}{m! (m + 1) ! (k - m - 1) !}  \left.
    \frac{\mathd^m}{\mathd x^m} (x^{k - 1} (\Psi' (x))^{k - m - 1}) \right|_{x
    = 1} . \label{boadcast}
  \end{equation}
  Both for $n = k - 1$ and $n = k$, the terms (\ref{Seskeftomaiotanxezw!})
  contribute to $\lim_{N \rightarrow \infty} m_{k, N} (\Psi, \Phi)$. As we
  showed above, for $n = k - 1$ the contribution is
  \begin{equation}
    \frac{k (k - 1)}{2} \sum_{m = 0}^{k - 1} \frac{(k - 1) !}{m! (m + 1) ! (k
    - m - 1) !}  \left. \frac{\mathd^m}{\mathd x^m} (x^{k - 1} (\Psi' (x))^{k
    - m - 1}) \right|_{x = 1} . \label{Iagapisouvganeiapomesamousandiaria.}
  \end{equation}
  The case $n = k$ is quite similar to what we did in Theorem
  \ref{Einaiiagapi}. When $u_1 = \cdots = u_N = 1$, the sum with respect to
  $b_0, \ldots, b_m \in \{1, \ldots, N\}$ of (\ref{minmemalwneis})
  (multiplied by its coefficient, which is $O (N^{l_1 + \cdots + l_{k - m}})$)
  gives a term that contributes to $\lim_{N \rightarrow \infty} m_{k, N}
  (\Psi, \Phi)$ when $l_1 + \cdots + l_{k - m} \in \{k - m - 1, k - m\}$. For
  $l_1 + \cdots + l_{k - m} = k - m$, the contribution is
  \begin{equation}
    - \frac{1}{2} \sum_{m = 0}^{k - 1} \frac{k!}{m! (m + 1) ! (k - m - 1) !} 
    \left. \frac{\mathd^{m + 1}}{\mathd x^{m + 1}} (x^k (\Psi' (x))^{k - m -
    1}) \right|_{x = 1} . \label{WNIX}
  \end{equation}
  On the other hand, for $l_1 + \cdots + l_{k - m} = k - m - 1$, the
  contribution is
  \begin{equation}
    \frac{1}{2} \sum_{m = 0}^{k - 2} \frac{k!}{m! (m + 1) ! (k - m - 2) !} 
    \left. \frac{\mathd^m}{\mathd x^m} (x^k \Psi'' (x) (\Psi' (x))^{k - m -
    2}) \right|_{x = 1} . \label{Kifmarokino}
  \end{equation}
  Summing the expressions from (\ref{Iagapisouvganeiapomesamousandiaria.}),
  (\ref{WNIX}) and (\ref{Kifmarokino}), we see that (\ref{boadcast}) holds.
  This proves the claim.
\end{proof}

Our next goal is to give a description of the above $1 / N$ correction in terms of the
infinitesimal quantized free probability. Namely, we would like to
understand how to express moments (\ref{mporw?}) via the moment-cumulant formula \eqref{thelongestway}. This would allow to introduce the infinitesimal quantized free cumulants, which should be a non-trivial deformation of infinitesimal free cumulants.

Interestingly enough, we were unable to prove an analog of Lemma \ref{24} in the discrete setup directly. Instead, below we are essentially proving Theorem \ref{Maria!!} again in a slightly different way, which will produce a required formula. 
The key point is to deform the Schur generating
function.

\begin{definition}
  A deformed Schur generating function is a symmetric function
  $T_{\varrho}^{U (N)}$ given by
  \[ T_{\varrho}^{U (N)} (u_1, \ldots, u_N) = \frac{\prod_{i < j} (e^{u_i} -
     e^{u_j})}{\prod_{i < j} (u_i - u_j)} S_{\varrho}^{U (N)} (e^{u_1},
     \ldots, e^{u_N}) . \]
\end{definition}
For our purpose, the usefulness of $T_{\varrho}^{U (N)}$ comes from the fact
that the function
\[ (u_1, \ldots, u_N) \longmapsto \frac{\det (e^{u_i (\mathlambda_j + N -
   j)})_{1 \leq i, j \leq N}}{\prod_{i < j} (u_i - u_j)
   \tmmathbf{\chi}^{\mathlambda} (1^N)} \]
is an eigenfunction of $\mathcal{D}_k$ with corresponding eigenvalue $\sum_{i =
1}^N (\mathlambda_i + N - i)^k$. Moreover, an additive asymptotic behaviour for
the logarithm of $S_{\varrho (N)}^{U (N)}$ implies the same phenomenon for the
logarithm of $T_{\varrho (N)}^{U (N)}$.

\begin{theorem}
\label{th:cumulants}
  Let $\varrho (N)$ be a sequence of probability measures on $\hat{U} (N)$
  such that for every finite $r$ the condition (\ref{thatonegamiswtonKinezo})
  is satisfied. Then, for every $k \in \mathbb{N}$, we have
  \begin{equation}
    \lim_{N \rightarrow \infty} N \left( \frac{1}{N} \sum_{\mathlambda \in
    \hat{U} (N)} \varrho (N) [\mathlambda] \sum_{i = 1}^N \left(
    \frac{\mathlambda_i + N - i}{N} \right)^k -\mathbf{m}_k \right) =
    \sum_{\pi \in \tmop{NC} (k)} \sum_{V \in \pi}
    \tmmathbf{\kappa}_{|V|}^{(1)} \prod_{\underset{W \neq V}{W \in \pi}}
    \tmmathbf{\kappa}_{|W|}^{(0)}, \label{mpanio}
  \end{equation}
  where $(\mathbf{m}_k)_{k \in \mathbb{N}}$ are given by (\ref{mporw?}) and
  for every $n \in \mathbb{N}$,
  \[ \tmmathbf{\kappa}_n^{(0)} \assign \frac{1}{(n - 1) !}  \left.
     \frac{\mathd^n}{\mathd u^n} \left( \Psi (e^u) + \log \left( \frac{e^u -
     1}{u} \right) \right) \right|_{u = 0} \text{\quad and\quad}
     \tmmathbf{\kappa}_n^{(1)} \assign \frac{1}{(n - 1) !}  \left.
     \frac{\mathd^n}{\mathd u^n} \left( \Phi (e^u) - \frac{1}{2} u \right)
     \right|_{u = 0} . \]
\end{theorem}

\begin{proof}
  Note that condition (\ref{thatonegamiswtonKinezo}) determines the
  asymptotics of the logarithm of $T_{\varrho (N)}^{U (N)}$, since for every
  finite $r$ we have
  \begin{equation}
    \lim_{N \rightarrow \infty} \frac{1}{N} \log \text{} T_{\varrho (N)}^{U
    (N)} (u_1, \ldots, u_r, 0^{N - r}) = \sum_{i = 1}^r \Psi (e^{u_i}) + \log
    \left( \frac{e^{u_i} - 1}{u_i} \right), \label{tsalapetions}
  \end{equation}
  where the above convergence is uniform in a complex neighborhood of $0^r$. Thus, we
  have
  \begin{equation}
    \mathbf{m}_k = \sum_{m = 0}^k \frac{k!}{m! (m + 1) ! (k - m) !}  \left.
    \frac{\mathd^m}{\mathd u^m} ((A' (u) + B' (u))^{k - m}) \right|_{u = 0},
    \label{Italides}
  \end{equation}
  where $A (u) \assign \Psi (e^u)$ and $B (u) \assign \log (e^u - 1) - \log
  \text{} (u)$. This holds because both the left and the right hand side of
  (\ref{Italides}) are equal to $\lim_{N \rightarrow \infty} N^{- k - 1}
  \mathcal{D}_k T_{\varrho (N)}^{U (N)} \longdownminus_{u_i = 0}$. Moreover, as
  we showed, (\ref{tsalapetions}) implies that for $N^{- k - 1} \mathcal{D}_k
  T_{\varrho (N)}^{U (N)} \longdownminus_{u_i = 0}$ we can get an expansion of
  the form $M_{0, k, N} + N^{- 1} M_{1, k, N} + \ldots$, where the sequences
  $(M_{i, k, N})_{n \in \mathbb{N}}$ converge. In the proof of Theorem
  \ref{Propolemiko} we justified why only the sequence $(M_{0, k, N})_{N \in
  \mathbb{N}}$ will contribute to the limit (\ref{mpanio}). We recall that
  $M_{0, k, N}$ is a linear combination of derivatives
  \begin{equation}
    \partial_{i_m} \ldots \partial_{i_1} \left( \partial_{i_0} \left(
    \frac{1}{N} \log \text{} T_{\varrho (N)}^{U (N)} \right) \right)^{k - m}
    (0^N), \label{involves}
  \end{equation}
  for $m = 0, \ldots, k - 1$, where the coefficients do not depend on $N$. The
  summand $M_{0, k, N}^{(1)}$ of $M_{0, k, N}$, that involves derivatives
  (\ref{involves}) where $i_0 = \cdots = i_m$, is
  \[ M_{0, k, N}^{(1)} = \sum_{m = 0}^k \frac{k!}{m! (m + 1) ! (k - m!)}
     \partial_1^m \left( \left( \partial_1 \left( \frac{1}{N} \log \text{}
     {T^{U (N)}_{\varrho (N)}}  \right) \right)^{k - m} \right) (0^N) . \]
  Due to (\ref{thatonegamiswtonKinezo}), for every non-negative integer
  $l$, one has
  \[ \lim_{N \rightarrow \infty} N \left( \partial_1^l \left( \frac{1}{N} \log
     \text{} T_{\varrho (N)}^{U (N)} \right) (0^N) - A^{(l)} (0) - B^{(l)} (0)
     \right) = \Gamma^{(l)} (0) - B^{(l)} (0), \]
  where $\Gamma (u) \assign \Phi (e^u)$. We deduce that
  \[ \lim_{N \rightarrow \infty} N (M_{0, k, N}^{(1)} -\mathbf{m}_k) =
      \sum_{m = 0}^{k - 1} \frac{k!}{m! (m + 1) ! (k - m
     - 1) !}  \left. \frac{\mathd^m}{\mathd u^m} ((\Gamma' (u) - B' (u)) (A'
     (u) + B' (u))^{k - m - 1}) \right|_{u = 0} . \]
  Moreover, since
  \[ N \left( \frac{1}{N} \log \text{} T_{\varrho (N)}^{U (N)} (u_1, \ldots,
     u_r, 0^{N - r}) - \sum_{i = 1}^r \Psi (e^{u_i}) + \log \left(
     \frac{e^{u_i} - 1}{u_i} \right) \right) = \log \left( \prod_{1 \leq i < j
     \leq r} \frac{e^{u_i} - e^{u_j}}{u_i - u_j} \right) \]
  \[ - r \text{} \log \left( \prod_{i = 1}^r \frac{e^{u_i} - 1}{u_i} \right) +
     N \left( \frac{1}{N} \log \text{} S_{\varrho (N)}^{U (N)} (e^{u_1},
     \ldots, e^{u_r}, 1^{N - r}) - \sum_{i = 1}^r \Psi (e^{u_i}) \right), \]
  condition (\ref{thatonegamiswtonKinezo}) implies that only the summands
  $M_{0, k, N}^{(1)}, M_{0, k, N}^{(2)}$ of $M_{0, k, N}$ will contribute to
  the limit (\ref{mpanio}), where $M_{0, k, N}^{(2)}$ is the summand that
  involves all the derivatives (\ref{involves}), where $|\{i_0, \ldots, i_m
  \}| = 2$. From the procedure that we described in the proof of Theorem
  \ref{SouvlakiA} that provides the expansion for $N^{- k - 1} \mathcal{D}_k
  T_{\varrho (N)}^{U (N)} \longdownminus_{u_i = 0}$, we obtain that
  \[ N \cdot \text{} M_{0, k, N}^{(2)} = - \sum_{m = 1}^{k - 1} \frac{k!N}{m!
     (m + 1) ! (k - m) !} \sum_{n = 1}^m \binom{m}{n} \partial_2^n
     \partial_1^{m - n} \left( \partial_1 \left( \frac{1}{N} \log \text{}
     T_{\varrho (N)}^{U (N)} \right) \right)^{k - m} (0^N) . \]
  In order to compute its limit as $N \rightarrow \infty$, we write
  \[ \lim_{N \rightarrow \infty} N \partial_2^n \partial_1^{m - n} \left(
     \partial_1 \left( \frac{1}{N} \log \text{} T_{\varrho (N)}^{U (N)}
     \right) \right)^{k - m} (0^N) = \]
  \[ \lim_{N \rightarrow \infty} \sum_{l_1 + \cdots + l_{k - m} = m - n}
     \frac{(m - n) !N}{l_1 ! \ldots l_{k - m} !}  \sum_{i = 1}^{k - m}
     \partial_2^n \partial_1^{l_i + 1} \left( \frac{1}{N} \log \text{}
     T_{\varrho (N)}^{U (N)} \right) (0^N) \prod_{\underset{j \neq i}{j =
     1}}^{k - m} \partial_1^{l_j + 1} \left( \frac{1}{N} \log \text{}
     T_{\varrho (N)}^{U (N)} \right) (0^N) \]
  \[ = (- 1)^n (k - m) \left. \frac{\mathd^{m - n}}{\mathd u^{m - n}} (B^{(n +
     1)} (u) (A' (u) + B' (u))^{k - m - 1}) \right|_{u = 0}, \]
  where we used that $\lim_{N \rightarrow \infty} \partial_2^n \partial_1^l
  \left( \log \text{} T_{\varrho (N)}^{U (N)} \right) (0^N) = (- 1)^n B^{(l +
  n)} (0)$. As a corollary, $(N \cdot M_{0, k, N}^{(2)})_{N \in \mathbb{N}}$
  will converge to
  \[ \sum_{m = 1}^{k - 1} \frac{k!}{m! (m + 1) ! (k - m - 1) !} \sum_{n = 1}^m
     (- 1)^{n + 1} \binom{m}{n} \left. \frac{\mathd^{m - n}}{\mathd u^{m - n}}
     (B^{(n + 1)} (u) (A' (u) + B' (u))^{k - m - 1}) \right|_{u = 0} \]
  \[ = \sum_{m = 0}^{k - 1} \frac{k!}{m! (m + 1) ! (k - m - 1) !}  \left.
     \frac{\mathd^m}{\mathd u^m} \left( \left( B' (u) - \frac{1}{2} \right)
     (A' (u) + B' (u))^{k - m - 1} \right) \right|_{u = 0}, \]
  because $B' (0) = 1 / 2$. Thus, due to Lemma \ref{24}, $\lim_{N \rightarrow
  \infty} N (M_{0, k, N} -\mathbf{m}_k)$ is equal to the right hand side of
  (\ref{mpanio}), which proves the claim.
\end{proof}

\begin{remark}
\label{rem:quant-class-free}
  Theorem \ref{th:cumulants} describes the ''differentiation procedure'' of
  the moment-cumulant relations of the limit measure from Theorem
  \ref{KALOKAIRI} that gives the $1 / N$ correction. In particular, for the
  limit regime (\ref{thatonegamiswtonKinezo}) we showed that the ''derivative'' of
  $u [0, 1]$ is $\delta \left( - 1 / 2 \right)$.
\end{remark}

In the remainder of the section, we provide some examples. Our first
example comes from \textit{extreme characters} of $U (\infty)$. Their
multiplicativity incorporates them into the framework of Theorem \ref{KALOKAIRI}
and Theorem \ref{Maria!!}. In that sense, they are analogous to ergodic
unitarily invariant matrices on $H (\infty)$.

We recall that a character $\tmmathbf{\chi}: U (\infty) \rightarrow
\mathbb{C}$ has a form
\[ \tmmathbf{\chi} (\tmop{diag} (u_1, \ldots, u_N, 1, 1, \ldots)) =
   \sum_{\mathlambda \in \hat{U} (N)} \varrho^{\tmmathbf{\chi}} (N)
   [\mathlambda] \frac{s_{\mathlambda} (u_1, \ldots, u_N)}{s_{\mathlambda}
   (1^N)}, \]
where $s_{\mathlambda}$ is the rational Schur function, given by
(\ref{Maroko}), and $\varrho^{\tmmathbf{\chi}} (N)$ is a probability measure on
$\hat{U} (N)$. By extreme characters we refer to the extreme points of the
convex set of characters of $U (\infty)$. The classification of the extreme
characters is known as the Edrei-Voiculescu theorem. 
They are multiplicative functions, meaning that
\[ \tmmathbf{\chi} (U) = \prod_{u \in \tmop{Spec} (U)} F (u), \]
where $F$ is a function of one variable, parametrized by a set of non-negative
parameters. It is known as Voiculescu function and its general form can be
found in \cite{E}, \cite{V3}, or, e.g., \cite{BBO}. For our particular example, we consider
\begin{equation}
  F (u) = \exp (\delta (u - 1)) \frac{1}{1 - \alpha (u - 1)},
  \label{KifMaRoKiNo}
\end{equation}
where $\delta, \alpha \geq 0$.

\begin{example}
\label{ex:disc-BBP}
  Let $\chi^{\pi (N)}$ be the character of $U (N)$ that is the restriction on $U
  (N)$ of the character
  $\tmmathbf{\chi}^{\pi}$ given by the function (\ref{KifMaRoKiNo}), where
  $\delta = N \gamma$ and $\alpha, \gamma \geq 0$ do not depend on $N$. Its decomposition into irreducible characters define a probability measure on $\hat U(N)$.
  Its Schur generating function $\tmmathbf{\chi}^{\pi (N)}
  (\tmop{diag} (u_1, \ldots, u_N))$ will satisfy
  (\ref{thatonegamiswtonKinezo}) for $\Psi (u) = \gamma (u - 1)$ and $\Phi (u)
  = - \log (1 - \alpha (u - 1))$. Therefore, the limiting probability measure that
  emerges form Theorem \ref{KALOKAIRI} has $k$-th moment
  \[ \sum_{m = 0}^k \frac{k! \gamma^{k - m}}{m! (m + 1) ! (k - m) !}  \left.
     \frac{\mathd^m}{\mathd u^m} (u + 1)^k \right|_{u = 0} . \]
  By Cauchy formula, this expression equals
  \[ \frac{1}{2 \mathpi \mathi (k + 1)} \oint_{|z| = \frac{1}{\sqrt{\gamma}}}
     \frac{1}{z + 1} \left( \gamma z + \frac{1}{z} + \gamma + 1 \right)^{k +
     1} d \text{} z \]
  \[ = \frac{1}{2 \mathpi (k + 1)} \int_{\gamma + 1 - 2 \sqrt{\gamma}}^{\gamma
     + 1 + 2 \sqrt{\gamma}} \frac{t + 1 - \gamma}{t}  \frac{t^{k +
     1}}{\sqrt{(\gamma + 1 + 2 \sqrt{\gamma} - t) (t - \gamma - 1 + 2
     \sqrt{\gamma})}} d \text{} t. \]
  Integrating by parts, we obtain that the density function of
  the limit probability measure is
  \[ \frac{1}{\mathpi} \tmmathbf{1}_{[\gamma + 1 - 2 \sqrt{\gamma}, \gamma + 1
     + 2 \sqrt{\gamma}]} (x) \arccos \left( \frac{x - 1 + \gamma}{2
     \sqrt{\gamma x}} \right), \]
  for $\gamma \ge 1$ and
  \[ \frac{1}{\mathpi} \tmmathbf{1}_{[\gamma + 1 - 2 \sqrt{\gamma}, \gamma + 1
     + 2 \sqrt{\gamma}]} (x) \arccos \left( \frac{x - 1 + \gamma}{2
     \sqrt{\gamma x}} \right) +\tmmathbf{1}_{[0, \gamma + 1 - 2
     \sqrt{\gamma}]} (x), \]
     for $0 <\gamma <1$. 
  To simplify the formulas, let us assume that $\gamma>1$.     
  For the correction measure, by Theorem \ref{Maria!!} its $k$-th moment is given by
  \begin{equation}
    \sum_{m = 0}^{k - 1} \frac{k!}{m! (m + 1) ! (k - m - 1) !}  \left. \left(
    \frac{\mathd^m}{\mathd u^m} \left( u^k \frac{\alpha \gamma^{k - m - 1}}{1
    - \alpha (u - 1)} \right) - \frac{1}{2} \frac{\mathd^m}{\mathd u^m} (u^{k
    - 1} \gamma^{k - m - 1}) \right) \right|_{u = 1} . \label{ND}
  \end{equation}
  After opening the parenthesis, the first summand can be written in the form
  \[ \sum_{m = 0}^{k - 1} \frac{k!}{m! (m + 1) ! (k - m - 1) !}  \left.
     \frac{\mathd^m}{\mathd u^m} \left( (u + 1)^k  \frac{\alpha \gamma^{k - m
     - 1}}{1 - \alpha u} \right) \right|_{u = 0} = \left( \alpha + \gamma + 1
     + \frac{\gamma}{\alpha} \right)^k - \left( \gamma + \frac{\gamma}{\alpha}
     \right)^k + \]
  \[ \text{{\hspace{8em}}} + \sum_{m = 0}^{k - 1} \frac{k!}{m! (m + 1) ! (k -
     m - 1) !}  \left. \frac{\mathd^m}{\mathd u^m} \left( \left( (u + 1)^k -
     \left( 1 + \frac{1}{\alpha} \right)^k \right) \frac{\alpha \gamma^{k - m
     - 1}}{1 - \alpha u} \right) \right|_{u = 0} . \]
  Using the Cauchy formula and doing similar computations as in Example
  \ref{Emeispamekorifi}, we obtain that the above sum is equal to
  \[ \frac{\alpha}{2 \mathpi \mathi} \oint_{|z| = \frac{1}{\sqrt{\gamma}}}
     \frac{1}{1 - \alpha z} \left( \left( \gamma z + \gamma + 1 + \frac{1}{z}
     \right)^k - (\gamma z + \gamma)^k - \left( \frac{1}{z} + \frac{1}{\alpha
     z} + \gamma + \frac{\gamma}{\alpha} \right)^k + \left( \gamma +
     \frac{\gamma}{\alpha} \right)^k \right) d \text{} z \text{} \]
  \[ = \frac{\alpha}{2 \mathpi \mathi} \oint_{|z| = \frac{1}{\sqrt{\gamma}}}
     \frac{1}{1 - \alpha z} \left( \left( \gamma z + \gamma + 1 + \frac{1}{z}
     \right)^k - \left( \alpha + 1 + \gamma + \frac{\gamma}{\alpha} \right)^k
     \right) d \text{} z \text{} - \left( \alpha + 1 + \gamma +
     \frac{\gamma}{\alpha} \right)^k + \left( \gamma + \frac{\gamma}{\alpha}
     \right)^k . \]
  For $\sqrt{\gamma} < \alpha$, the above contour integral is equal to
  \begin{equation}
  \label{eq:ex-BPP}
  \frac{\alpha}{2 \mathpi} \int_{\gamma + 1 - 2 \sqrt{\gamma}}^{\gamma + 1
     + 2 \sqrt{\gamma}} \frac{t - 2 \alpha - \gamma - 1}{\gamma + \alpha^2 +
     \alpha \gamma + \alpha - \alpha t}  \frac{t^k}{\sqrt{(\gamma + 1 + 2
     \sqrt{\gamma} - t) (t - \gamma - 1 + 2 \sqrt{\gamma})}} d \text{} t +
     \left( \alpha + 1 + \gamma + \frac{\gamma}{\alpha} \right)^k.
     \end{equation}
  In the opposite case $\alpha < \sqrt{\gamma}$, equation \eqref{eq:ex-BPP} holds without the second term in its right-hand side, while in the case $\alpha = \sqrt{\gamma}$, we have a half of the second term.
  
  The remaining summand of (\ref{ND}) can be written as
  \[ \frac{1}{2} \sum_{m = 0}^{k - 1} \frac{k! \gamma^{k - m - 1}}{m! (m + 1)
     ! (k - m - 1) !}  \left. \frac{\mathd^m}{\mathd u^m} (u^{k - 1})
     \right|_{u = 1} = \frac{1}{4 \mathpi \mathi} \oint_{|z| =
     \frac{1}{\sqrt{\gamma}}} \frac{1}{1 + z} \left( \gamma z + \frac{1}{z} +
     \gamma + 1 \right)^k d \text{} z \]
  \[ = \frac{1}{4 \mathpi} \int_{\gamma + 1 - 2 \sqrt{\gamma}}^{\gamma + 1 + 2
     \sqrt{\gamma}} \frac{t + 1 - \gamma}{t}  \frac{t^k}{\sqrt{(\gamma + 1 + 2
     \sqrt{\gamma} - t) (t - \gamma - 1 + 2 \sqrt{\gamma})}} d \text{} t. \]

Combining the formulas, we obtain that the correction measure is equal to 
\begin{multline}
\label{eq:BBP-answer}
\left( \frac{ \alpha ( t - 2 \alpha - \gamma - 1)}{\gamma + \alpha^2 +
     \alpha \gamma + \alpha - \alpha t} - \frac{t + 1 - \gamma}{2t} \right) \frac{1}{2 \pi \sqrt{(\gamma + 1 + 2
     \sqrt{\gamma} - t) (t - \gamma - 1 + 2 \sqrt{\gamma})}} \mathbf{1}_{ t \in [(\sqrt{\gamma}-1)^2;(\sqrt{\gamma}+1)^2]} \\ + \delta \left( \alpha + 1 + \gamma + \frac{\gamma}{\alpha} \right) \mathbf{1}_{\alpha > \sqrt{\gamma}} + \frac12 \delta \left( \alpha + 1 + \gamma + \frac{\gamma}{\alpha} \right) \mathbf{1}_{\alpha = \sqrt{\gamma}}. 
\end{multline}
  
This answer (see also Figure \ref{fig:dBBP}) illustrates a discrete analog of the BBP phase transition for random matrices. We see that if one perturbs the large one-sided Plancherel representation (discrete analog of GUE) by multiplying its character to a character of one-dimensional representation, the resulting limiting measure has an outlier in the case $\alpha > \sqrt{\gamma}$, and does not have it otherwise. This is analogous to a random matrix phenomenon. 

\end{example}

\begin{figure}
\center
\includegraphics[height=5cm]{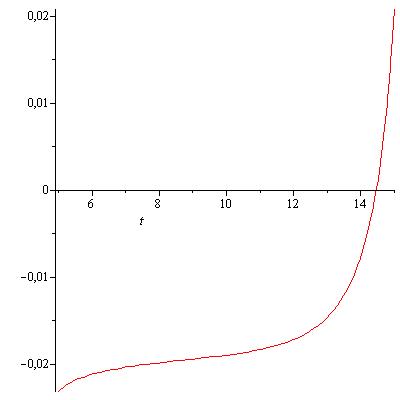}
\includegraphics[height=5cm]{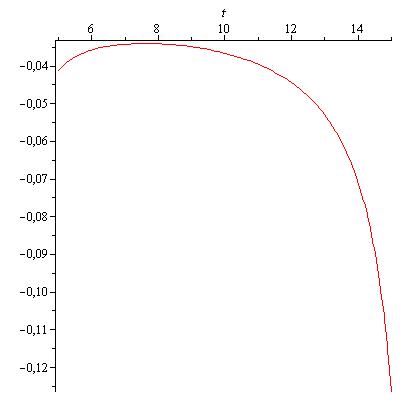}
\caption{The density from \eqref{eq:BBP-answer} for ($\gamma=9$,$\alpha=2$) and ($\gamma=9$,$\alpha=4$), respectively. It is visible that for $\alpha=2$ the correction measure created by a perturbation is positive at the right of its support and negative at the left. For $\alpha=4$, the outlier is created, thus the correction measure is negative inside the support of its continuous part. }
\label{fig:dBBP}

\end{figure}

\begin{example}
\label{ex:Aztec}

In this example we will address the perturbation of the model of uniformly random domino tilings of the Aztec diamond. Let us introduce very briefly the required notions. We refer to \cite{BK} for a detailed description of the construction below. 

A signature $\lambda$ with non-negative entries can be also thought of as a Young diagram and vice versa. Two Young diagrams $\lambda, \mu$ interlace (notation $\mu \prec \lambda$) if $\lambda_1 \ge \mu_1 \ge \lambda_2 \ge \mu_2 \ge \lambda_3 \ge \dots$ holds. Two Young diagrams interlace vertically
(notation $\mu \prec_v \lambda$) if the conjugations of $\lambda$ and $\mu$ interlace in the sense defined above. 
By $|\lambda|$ one denotes the sum $\lambda_1 + \lambda_2 + \dots$, which is the total number of boxes in the Young diagram $\lambda$.

It is known that domino tilings of the Aztec diamond of size $M$ are in a natural bijection with the collection of Young diagrams $(\lambda^1, \mu^1, \dots, \lambda^M)$ satisfying
$$
\varnothing \prec \lambda^1 \succ_v \mu^1 \prec \lambda^2 \succ_v \mu^2 \prec \dots \prec \lambda^M \succ_v \varnothing.
$$
Each signature $\lambda^r$ describes the positions of domino tilings along a one-dimensional slice of the Aztec diamond. 

One can introduce a probability measure on domino tilings of the Aztec diamond in various ways. The first one is to consider the uniform measure on all tilings. Uniformly random domino tilings were extensively studied in the last thirty years, in particular, in pioneering works \cite{JPS}, \cite{J}.

We will consider a perturbation of the uniform measure. Let us view the Aztec diamond as a square, and attach to horizontal dominoes along the first (top) up-right row (if the Aztec diamond is depicted as in Figure \ref{fig:til}) the weight $A>1$, while all the other dominoes have weight 1. Now, consider a measure on the space of domino tilings with the probability of a tiling being proportional to 
$$A^{\text{number of horizontal dominoes in the first up-right row}}.$$ 
In terms of interlacing Young diagrams above, the probability $Z A^{|\lambda^M|}$ is attached to a sequence $(\lambda^1, \mu^1, \dots, \lambda^M)$, where $Z$ is a normaizing constant. 

For $1 \le N \le M$, it follows from \cite[Sections 2 and 8]{BK} that the Schur generating function of the (marginal) distribution of the signature $\lambda^{N}$ under such a measure is given by
$$
\prod_{i=1}^N \left( \frac{1+x_i}{2} \right)^{M-N-1} \frac{1+A x_i}{1+A}.
$$
To make formulas below more readable, we assume that $N= \alpha M$ exactly (without rounding up or down), where $0<\alpha<1$ is fixed. Then, as $N \to \infty$, we can apply Theorem \ref{Maria!!} in order to study the asymptotic behavior of the signature $\lambda^{N}$: we see that the assumptions of the theorem hold with
$$
\Psi (x) = \left( \frac{1}{\alpha} - 1 \right) \log \left( \frac{1+x}{2} \right), \qquad 
\Phi (x) = \log \left( \frac{1+Ax}{1+x} \right).
$$
Let us assume additionally that $\alpha>\frac12$. Performing very similar to the previous example and somewhat lengthy computations (we omit them for brevity), one arrives to the following formula for the correction measure: 
\begin{multline}
\label{eq:Aztec-correction}
- \mathbf{1}_{\alpha > \frac{(A+1)^2}{2(A^2+1)}} \delta \left( \frac{-1 + 2\alpha A -A}{\alpha ( A^2 -1 )} \right) - \frac12 \mathbf{1}_{\alpha = \frac{(A+1)^2}{2(A^2+1)}} \delta \left( \frac{-1 + 2\alpha A -A}{\alpha ( A^2 -1 )} \right) \\ + \left( \alpha+\frac{-2 \alpha+1}{4y}-\frac{\alpha (2 \alpha-1)}{4( \alpha y-1)}+\frac{\alpha (2 A^2 \alpha-A^2-2A+2 \alpha-1)}{2(-\alpha y+A^2 \alpha y-2 \alpha A+1+A)} \right) \frac{\mathbf{1}_{1-(2 \alpha-1)^2-(2 \alpha y-1)^2 >0}}{\pi \sqrt{1-(2 \alpha-1)^2-(2 \alpha y-1)^2}} dy.
\end{multline}

As discussed in Section \ref{sec:intro}, the result explains geometric properties of a random tiling visible on Figure \ref{fig:til}, and also it is related to another setup known as the Tangent Method for models of statistical mechanics. Formula \eqref{intro-eq:Aztec-LLN} can be obtained from the formula for moments in a standard way. Note that in our assumptions (in particular, $\alpha >1/2$) outside of the arctic curve we have (the maximal possible) density 1 of particles encoded by a signature; therefore, the outlier is parameterized by a delta measure with a negative sign. 



\end{example}

\begin{remark}
All general theorems in the current paper provide formulas for moments of limiting (correction) measures. In order to find their density, in all examples we perform direct calculations for finding the integral representation for moments. In this remark we outline an another approach -- via Stieltjes transforms -- which is comparable in complexity and can be more convenient in some other setups. We will focus on the statement of Theorem \ref{Maria!!}, a similar computation can be performed for other theorems from this paper as well. 

In notations of Theorem \ref{Maria!!}, using the binomial theorem, one gets 
$$
\mathbf{m}'_k = \frac{1}{2 \pi \mathbf{i}} \oint \left( \Phi'(x) - \frac{1}{2x} \right) \mathcal{F} (x)^k dx,
$$
where $\mathcal{F} (x) := x \Psi'(x) + \frac{x}{x-1}$, and the contour of integration encircles 
$x=1$ and no other poles of the integrand. Therefore, we have
\begin{equation}
\label{eq:Stielt}
- \sum_{i=0}^{\infty} \mathbf{m}'_k y^{k+1} = - \frac{1}{2 \pi \mathbf{i}} \oint \frac{ \Phi'(x) - \frac{1}{2x} }{y} \frac{1}{1-\frac{\mathcal{F}(x)}{y}} dx.
\end{equation}
The equation $\mathcal{F}(x) = y$ has a unique solution $x_+ (y)$ for large $|y|$ (this follows from the formal power series technique). The solution can be analytically continued to the upper half-plane. In examples, this continuation often can be given explicitly, and it is also continuous on a real line (though we do not know how to justify these steps in full generality). In such a case, the contour integral in \eqref{eq:Stielt} can be calculated via computing the residue at the single pole inside the contour of integration for large $|y|$, and then analytically continued.
One obtains the following formula for the Stieltjes transform of the correction measure:
$$
St_{\mu} (y) = \frac{\Phi'(x_+ (y)) - \frac{1}{2 x_+ (y)}}{ \mathcal{F}' (x_+(y))}, \qquad \Im (y)>0.
$$
Thus, the formula for the density of the correction measure is given by
$$
d(y) = \frac{1}{\pi} \Im \left( \frac{\Phi'(x_+ (y)) - \frac{1}{2 x_+ (y)}}{ \mathcal{F}' (x_+(y))} \right), \qquad y \in \mathbb{R}, 
$$
while outliers can be found from the condition $\Phi'(x_+ (y)) = \infty$.

\end{remark}

\section{Second order infinitesimal free probability}
\label{sec:second-order}

In this section we analyze the third leading term in the expectation of the empirical measure of random matrices via its Harish-Chandra transform. We also connect it with notions coming from second order infinitesimal free probability. 

The main result of the section is the following theorem. 

\begin{theorem}
\label{th:sec-order-main}
  Let $A$ be a random Hermitian matrix of size $N$ and assume that for every
  finite $r$ one has
  \begin{equation} 
  \lim_{N \rightarrow \infty} N^2 \left( \frac{1}{N} \log \mathbb{E}[H
     \text{} C (x_1, \ldots, x_r, 0^{N - r} ; \lambda_1 (A), \ldots, \lambda_N
     (A))] - \sum_{i = 1}^r \Psi (x_i) - \frac{1}{N} \sum_{i = 1}^r \Phi (x_i)
     \right) = \sum_{i = 1}^r \Tau (x_i), \label{audikos}
     \end{equation} 
  where $\Psi, \Phi, \Tau$ are smooth functions in a complex neighborhood of $0$ and
  the above convergence is uniform in a complex neighborhood of $0^r$. Then the $1 /
  N^2$ correction of the average empirical distribution of $N^{- 1} \cdot A$
  is given by
  \[ \lim_{N \rightarrow \infty} N^2 \left( \frac{1}{N^{k + 1}}
     \mathbb{E}[\tmop{Tr} (A^k)] -\tmmathbf{\mu}_k - \frac{1}{N}
     \tmmathbf{\mu}'_k \right) =\tmmathbf{\mu}''_k
     +\tmmathbf{\nu}''_k, \]
  where $(\tmmathbf{\mu}_k)_{k \in \mathbb{N}}, (\tmmathbf{\mu}'_k)_{k \in
  \mathbb{N}}$ are given by (\ref{topg}),(\ref{G}) respectively. The sequences
  $(\tmmathbf{\mu}''_k)_{k \in \mathbb{N}}, (\tmmathbf{\nu}''_k)_{k \in
  \mathbb{N}}$ are given by
  \[ \tmmathbf{\mu}''_k = \sum_{m = 0}^{k - 1} \frac{k!}{m! (m + 1) ! (k -
     m - 1) !}  \left. \frac{\mathd^m}{\mathd x^m} ((\Psi' (x))^{k - m - 1}
     \Tau' (x)) \right|_{x = 0} \]
  \begin{equation}
    \text{{\hspace{10em}}} + \frac{1}{2} \sum_{m = 0}^{k - 2} \frac{k!}{m! (m
    + 1) ! (k - m - 2) !}  \left. \frac{\mathd^m}{\mathd x^m} ((\Psi'
    (x))^{k - m - 2} (\Phi' (x))^2) \right|_{x = 0} \label{Ss}
  \end{equation}
  and
  \[ \tmmathbf{\nu}''_k = \frac{1}{24} \sum_{m = 0}^{k - 3} \frac{k!}{m! (m
     + 1) ! (k - m - 3) !}  \left. \frac{\mathd^m}{\mathd x^m} ((\Psi'
     (x))^{k - m - 3} \Psi''' (x)) \right|_{x = 0} \]
  \[ \text{{\hspace{4em}}} + \frac{1}{12} \sum_{m = 0}^{k - 3} \frac{k!}{m! (m
     + 2) ! (k - m - 3) !}  \left. \frac{\mathd^m}{\mathd x^m} ((\Psi'
     (x))^{k - m - 3} \Psi''' (x)) \right|_{x = 0} \]
  \begin{equation}
    \text{{\hspace{10em}}} + \frac{1}{12} \sum_{m = 0}^{k - 4} \frac{k!}{m! (m
    + 2) ! (k - m - 4) !}  \left. \frac{\mathd^m}{\mathd x^m} ((\Psi'
    (x))^{k - m - 4} (\Psi'' (x))^2) \right|_{x = 0} . \label{sagan}
  \end{equation}
\end{theorem}

\begin{proof}
  For notation simplicity we write $f (x_1, \ldots, x_N) =\mathbb{E} [H
  \text{} C (x_1, \ldots, x_N ; \lambda_1 (A), \ldots, \lambda_N (A))]$. Note
  that \eqref{audikos} implies that the relations (\ref{xereis}),(\ref{ola})
  are satisfied. Thus, the empirical distribution of $N^{- 1} \cdot A$
  converges to a probability measure with moments
  $(\tmmathbf{\mu}_k)_{k \in \mathbb{N}}$ and its $1 / N$ correction is given
  by $(\tmmathbf{\mu}'_k)_{k \in \mathbb{N}}$. We have also shown an
  expansion
  \begin{equation}
    \frac{1}{N^{k + 1}} \mathbb{E} [\tmop{Tr} (A^k)] = \sum_{i = 0}^k
    \frac{1}{N^i} M_{i, N} (\Psi), \label{Drouvis}
  \end{equation}
  where $N^2 M_{0, N} (\Psi)$ is a linear combination of terms of the form
  \begin{equation}
    N^2 \partial_{i_m} \ldots \partial_{i_1} \left( \partial_{i_0} \left(
    \frac{1}{N} \log \text{} f \right) \right)^{k - m} (0^N) \label{gousaki} .
  \end{equation}
  Due to (\ref{audikos}), for $|\{i_0, \ldots, i_m \}| > 1$, (\ref{gousaki})
  converges to $0$ as $N \rightarrow \infty$, which implies that
  \[ N^2 \left( M_{0, N} (\Psi) -\tmmathbf{\mu}_k - \frac{1}{N}
     \tmmathbf{\mu}'_k \right) = \omicron (1) + N^2 \sum_{m = 0}^{k - 1}
     \frac{k!}{m! (m + 1) ! (k - m) !} \times \]
  \[ \left( \left. \partial_1^m \left( \partial_1 \left( \frac{1}{N} \log
     \text{} f \right) \right)^{k - m} (0^N) - \frac{\mathd^m}{\mathd x^m}
     ((\Psi' (x))^{k - m}) \right|_{x = 0} - \frac{k - m}{N}  \left. 
     \frac{\mathd^m}{\mathd x^m} ((\Psi' (x))^{k - m - 1} \Phi' (x))
     \right|_{x = 0} \right) . \]
  For fixed $m \in \{0, \ldots, k - 1\}$, by Leibniz rule, we get
  \[ N^2 \left( \partial_1^m \left( \partial_1 \left( \frac{1}{N} \log \text{}
     f \right) \right)^{k - m} (0^N) - \left. \frac{\mathd^m}{\mathd x^m}
     ((\Psi' (x))^{k - m}) \right|_{x = 0} - \frac{k - m}{N}  \left.
     \frac{\mathd^m}{\mathd x^m} ((\Psi' (x))^{k - m - 1} \Phi' (x))
     \right|_{x = 0} \right) \]
   \[
   = \sum_{l_1 + \cdots + l_{k - m} = m} \frac{m!}{l_1 ! \ldots l_{k - m} !}
     \] 
     \begin{equation}
    \times \left( N^2 \prod_{i = 1}^{k - m} \partial_1^{l_i + 1} \left( \frac{1}{N}
    \log \text{} f \right) (0^N) - N^2 \prod_{i = 1}^{k - m} \Psi^{(l_i + 1)}
    (0) - N \sum_{i = 1}^{k - m} \Phi^{(l_i + 1)} (0) \prod_{\underset{j \neq
    i}{j = 1}}^{k - m} \Psi^{(l_j + 1)} (0) \right) \label{9}
  \end{equation}
  For fixed indices $l_1, \ldots, l_{k - m}$, due to (\ref{audikos}), for
  every $i \in \{1, \ldots, k - m\}$, we have 
  \[ \lim_{N \rightarrow \infty} N^2 \left. \partial_1^{l_i + 1} \left(
     \frac{1}{N} \log \text{} f \right) (0^N) - N^2 \Psi^{(l_i + 1)} (0) - N
     \Phi^{(l_i + 1)} (0) \right) = \Tau^{(l_i + 1)} (0), \]
  and by induction on $k - m$ we obtain
  \[ \lim_{N \rightarrow \infty} \left( N^2 \prod_{i = 1}^{k - m}
     \partial_1^{l_i + 1} \left( \frac{1}{N} \log \text{} f \right) (0^N) -
     N^2 \prod_{i = 1}^{k - m} \Psi^{(l_i + 1)} (0) - N \sum_{i = 1}^{k - m}
     \Phi^{(l_i + 1)} (0) \prod_{\underset{j \neq i}{j = 1}}^{k - m}
     \Psi^{(l_j + 1)} (0) \right) \]
  \[ = \sum_{i = 1}^{k - m} \Tau^{(l_i + 1)} (0) \prod_{\underset{j \neq i}{j
     = 1}}^{k - m} \Psi^{(l_j + 1)} (0) + \frac{1}{2} \sum_{\underset{i \neq
     j}{i, j = 1}}^{k - m} \Phi^{(l_i + 1)} (0) \Phi^{(l_j + 1)} (0)
     \prod_{\underset{r \neq i, j}{r = 1}}^{k - m} \Psi^{(l_r + 1)} (0) . \]
  Therefore, as $N \rightarrow \infty$, (\ref{9}) converges to
  \[ (k - m) \left. \frac{\mathd^m}{\mathd x^m} ((\Psi' (x))^{k - m - 1}
     \Tau' (x)) \right|_{x = 0} + \frac{(k - m) \text{} (k - m - 1)}{2} 
     \left. \frac{\mathd^m}{\mathd x^m} ((\Psi' (x))^{k - m - 2} (\Phi'
     (x))^2) \right|_{x = 0}, \]
  which implies that
  \[ \lim_{N \rightarrow \infty} N^2 \left( M_{0, N} (\Psi) -\tmmathbf{\mu}_k
     - \frac{1}{N} \tmmathbf{\mu}'_k \right) =\tmmathbf{\mu}''_k . \]
  Thus, the claim holds if $\lim_{N \rightarrow \infty} (N \text{}
  \cdot M_{1, N} (\Psi) + M_{2, N} (\Psi)) =\tmmathbf{\nu}''_k$.
  The computation of $M_{1, N} (\Psi)$ was presented in the proof of Theorem \ref{Einaiiagapi}. 
  Taking also into account the assumption (\ref{audikos}), we have
  \[ N \cdot M_{1, N} (\Psi) = \omicron (1) + \frac{N}{2} \sum_{m = 0}^{k - 2}
     \frac{k!}{m! (m + 1) ! (k - m - 2) !} \partial_1^m \left( \left(
     \partial_1 \left( \frac{1}{N} \log \text{} f \right) \right)^{k - m - 2}
     \partial_1^2 \left( \frac{1}{N} \log \text{} f \right) \right) (0^N) \]
  \[ - \frac{N}{2} \sum_{m = 1}^{k - 1} \frac{k!}{(m - 1) !m! (k - m) !}
     \partial_1^m \left( \partial_1 \left( \frac{1}{N} \log \text{} f \right)
     \right)^{k - m} (0^N) = \omicron (1). \]
  Therefore, $\lim_{N \rightarrow \infty} N \cdot M_{1, N} (\Psi) = 0$. In order to
  determine $M_{2, N} (\Psi)$, note that relation
  (\ref{xereis}) implies that for every $k \in \mathbb{N}$, $\mathbb{E}
  [\tmop{Tr} (A^k)]$ can be written as a linear combination of terms
  \begin{equation}
    N^{l_1 + \cdots + l_{k - m}} \binom{N}{m + 1} c_N (m), \label{MAL}
  \end{equation}
  where $m \in \{0, \ldots, k - 1\}$, $l_1 + 2 l_2 + \cdots + (k - m) l_{k -
  m} = k - m$, and $c_N (m)$'s are linear combinations of derivatives of the form
  (\ref{antk}). Then, $M_{2, N} (\Psi)$ is given by the sum of the summands of
  terms (\ref{MAL}) with factor $N^{k - 1}$. There are three cases that we have to
  consider regarding $l_1, \ldots, l_{k-m}$ in order to compute these summands, since a term
  \[ N^{l_1 + \cdots + l_{k - m}} \binom{N}{m + 1} \]
  should be a polynomial in $N$ of degree at least $k - 1$. Thus, we must have $l_1
  + \cdots + l_{k - m} \in \{k - m, k - m - 1, k - m - 2\}$, since a
  binomial coefficient is a polynomial in $N$ of degree $m + 1$. For each
  case, we are interested in the coefficient of $N^{k - 1}$. For $l_1 + \cdots
  + l_{k - m} = k - m - 1$, i.e. $l_1 = k - m - 2$, $l_2 = 1$ and $l_3 =
  \cdots = l_{k - m} = 0$ the summands of the terms (\ref{MAL}) with factor
  $N^{k - 1}$ are
  \[ - N^{k - 1} \left( \sum_{i = 1}^m i \right) \left( \frac{k!}{2 m! (m + 1)
     ! (k - m - 2) !} \partial_1^m \left( \left( \partial_1 \left( \frac{1}{N}
     \log \text{} f \right) \right)^{k - m - 2} \partial_1^2 \left(
     \frac{1}{N} \log \text{} f \right) \right) (0^N) + \omicron (1) \right)
  \]
  \[ = - N^{k - 1} \left( \frac{k!}{4 (m - 1) !m! (k - m - 2) !} \partial_1^m
     \left( \left( \partial_1 \left( \frac{1}{N} \log \text{} f \right)
     \right)^{k - m - 2} \partial_1^2 \left( \frac{1}{N} \log \text{} f
     \right) \right) (0^N) + \omicron (1) \right), \]
  where $m \in \{1, \ldots, k - 2\}$. Thus, the contribution of these summands
  to $\lim_{N \rightarrow \infty} M_{2, N} (\Psi)$ will be
  \begin{equation}
    - \sum_{m = 0}^{k - 3} \frac{k!}{4 m! (m + 1) ! (k - m - 3) !}  \left.
    \frac{\mathd^{m + 1}}{\mathd x^{m + 1}} ((\Psi' (x))^{k - m - 3}
    \Psi'' (x)) \right|_{x = 0} \label{2222} .
  \end{equation}
  For the case $l_1 + \cdots + l_{k - m} = k - m - 2$ we have that $l_1 = k -
  m - 4$, $l_2 = 2$, $l_3 = \cdots = l_{k - m} = 0$ or $l_1 = k - m - 3$, $l_3
  = 1$ and $l_2 = l_4 = \cdots = l_{k - m} = 0$. For $l_1 = k - m - 4$, $l_2 =
  2$ and $l_3 = \cdots = l_{k - m} = 0$, the summands of the terms (\ref{MAL})
  with factor $N^{k - 1}$ are
  \[ N^{k - 1} \left( \frac{k!}{8 m! (m + 1) ! (k - m - 4) !} \partial_1^m
     \left( \left( \partial_1 \left( \frac{1}{N} \log \text{} f \right)
     \right)^{k - m - 4} \left( \partial_1^2 \left( \frac{1}{N} \log \text{} f
     \right) \right)^2 \right) (0^N) + \omicron (1) \right), \]
  where $m \in \{0, \ldots, k - 4\}$. On the other hand, for $l_1 = k - m -
  3$, $l_3 = 1$ and $l_2 = l_4 = \cdots = l_{k - m} = 0$, the summands of the
  terms (\ref{MAL}) with factor $N^{k - 1}$ are
  \[ N^{k - 1} \left( \frac{k!}{6 m! (m + 1) ! (k - m - 3) !} \partial_1^m
     \left( \left( \partial_1 \left( \frac{1}{N} \log \text{} f \right)
     \right)^{k - m - 3} \partial_1^3 \left( \frac{1}{N} \log \text{} f
     \right) \right) (0^N) + \omicron (1) \right), \]
  where $m \in \{0, \ldots, k - 3\}$. Hence, the case $l_1 + \cdots + l_{k -
  m} = k - m - 2$ contributes to $\lim_{N \rightarrow \infty} M_{2, N} (\Psi)$
  the expression
  \[ \sum_{m = 0}^{k - 4} \frac{k!}{8 m! (m + 1) ! (k - m - 4) !}  \left.
     \frac{\mathd^m}{\mathd x^m} ((\Psi' (x))^{k - m - 4} (\Psi''
     (x))^2) \right|_{x = 0} \]
  \begin{equation}
    \text{{\hspace{10em}}} + \sum_{m = 0}^{k - 3} \frac{k!}{6 m! (m + 1) ! (k
    - m - 3) !}  \left. \frac{\mathd^m}{\mathd x^m} ((\Psi' (x))^{k - m -
    3} \Psi''' (x)) \right|_{x = 0} \label{333} .
  \end{equation}
  For the last case $l_1 + \cdots +
  l_{k - m} = k - m$, we have $l_1 = k - m$ and $l_2 = \cdots = l_{k - m} =
  0$. Then, the desired summands will be
  \[ N^{k - 1} \left( \frac{1}{2} \sum_{\underset{i \neq j}{i, j = 0}}^m i
     \text{} j \right) \left( \frac{k!}{m! (m + 1) ! (k - m) !} \partial_1^m
     \left( \partial_1 \left( \frac{1}{N} \log \text{} f \right) \right)^{k -
     m} (0^N) + \omicron (1) \right) \]
  \[ \text{{\hspace{4em}}} = N^{k - 1} \frac{k!}{8 (m - 2) ! (m - 1) ! (k - m)
     !} \partial_1^m \left( \partial_1 \left( \frac{1}{N} \log \text{} f
     \right) \right)^{k - m} (0^N) \]
  \[ \text{{\hspace{8em}}} + N^{k - 1} \frac{k!}{12 (m - 2) !m! (k - m) !}
     \partial_1^m \left( \partial_1 \left( \frac{1}{N} \log \text{} f \right)
     \right)^{k - m} (0^N) + N^{k - 1} \omicron (1), \]
  where $m \in \{2, \ldots, k - 1\}$. Thus, the contribution of these summands
  to $\lim_{N \rightarrow \infty} M_{2, N} (\Psi)$ will be equal to
  \begin{equation}
    \sum_{m = 0}^{k - 3} \left( \frac{k!}{8 m! (m + 1) ! (k - m - 2) !} +
    \frac{k!}{12 m! (m + 2) ! (k - m - 2) !} \right)  \left. \frac{\mathd^{m +
    2}}{\mathd x^{m + 2}} ((\Psi' (x))^{k - m - 2}) \right|_{x = 0}. \label{l}
  \end{equation}
  Since $\lim_{N \rightarrow \infty} M_{2, N} (\Psi)$ is equal to the sum of
  (\ref{2222}),(\ref{333}) and (\ref{l}), a straightforward computation shows
  that the limit is equal to $\tmmathbf{\nu}''_k$. This proves the claim.
\end{proof}

We would like to better understand the terms
$\tmmathbf{\mu}''_k, \tmmathbf{\nu}''_k$ from  Theorem \ref{th:sec-order-main}. 
In the proof of
Theorem \ref{th:sec-order-main} we showed that $\tmmathbf{\nu}''_k = \lim_{N \rightarrow \infty} M_{2, N} (\Psi)$, and that $\tmmathbf{\mu}''_k$ contributes
to the $1 / N^2$ correction of $A / N$ due to the relation
\[ \tmmathbf{\mu}''_k = \lim_{N \rightarrow \infty} N^2 \left( M_{0, N} (\Psi)
   -\tmmathbf{\mu}_k - \frac{1}{N} \tmmathbf{\mu}'_k \right) . \]
However, we would like to provide also a different explanation of the term $\tmmathbf{\mu}''_k$,
which makes the connection with second order infinitesimal free probability.
This connection is a generalization of the fact that the formula (\ref{G}) for the
$1 / N$ correction of $A / N$ gives an explicit description for the
infinitesimal free cumulants of the infinitesimal non-commutative probability
space $(\mathbb{C} \langle \mathbf{x} \rangle, \varphi, \varphi')$, where
for every $P \in \mathbb{C} \langle \mathbf{x} \rangle$
\[ \varphi (P) = \lim_{N \rightarrow \infty} \frac{1}{N} \mathbb{E} [\tmop{Tr}
   \text{} P (N^{- 1} \cdot A)] \text{\quad and\quad} \varphi' (P) = \lim_{N
   \rightarrow \infty} N \left( \frac{1}{N} \mathbb{E}[\tmop{Tr} \text{} P
   (N^{- 1} \cdot A)] - \varphi (P) \right) . \]
By the above equalities, $\mu_k = \varphi (\mathbf{x}^k)$ and
$\tmmathbf{\mu}'_k = \varphi' (\mathbf{x}^k)$, for every $k \in \mathbb{N}$.
The role of functions $\Psi, \Phi$ in the computation of infinitesimal free
cumulants can be seen from (\ref{chen}) and (\ref{alex}). Thus, thinking of
$\tmmathbf{\mu}_k, \tmmathbf{\mu}'_k$ as moments, we have that the
relations (\ref{topg}),(\ref{G}) describe moment-cumulant relations for an infinitesimal non-commutative probability space of order 1. The fact that $\tmmathbf{\mu}''_k$ is determined by (\ref{Ss}) allows us to go a step further and see the relations
(\ref{topg}),(\ref{G}),(\ref{Ss}) as moment-cumulant relations, in the sense
of Definition \ref{def:inf-free-mom-cum}, for an infinitesimal non-commutative probability space of order
2. This is shown by the next lemma.

\begin{lemma}
\label{lem:mu-doublePrime}
  Let $\Psi, \Phi, \Tau$ be infinitely differentiable functions at $0$ and
  consider the sequences $(c_n)_{n \in \mathbb{N}}, (c'_n)_{n \in \mathbb{N}},
  (c''_n)_{n \in \mathbb{N}}$, where for every $n \in \mathbb{N}$
  \[ c_n \assign \frac{\Psi^{(n)} (0)}{(n - 1) !}, \text{\quad} c'_n \assign
     \frac{\Phi^{(n)} (0)}{(n - 1) !} \text{\quad and\quad} c''_n \assign 2
     \frac{\Tau^{(n)} (0)}{(n - 1) !} . \]
  Then for the sequence $(\tmmathbf{\mu}''_k)_{k \in \mathbb{N}}$ defined by
  (\ref{Ss}), for every $k \in \mathbb{N}$,  we have
  \[ \tmmathbf{\mu}''_k = \frac{1}{2} \left( \sum_{\pi \in \tmop{NC} (k)}
     \sum_{V \in \pi} c''_{|V|} \prod_{\underset{W \neq V}{W \in \pi}} c_{|W|}
     + \sum_{\pi \in \tmop{NC} (k)} \sum_{\underset{W \neq V}{V, W \in \pi}}
     c'_{|V|} c'_{|W|} \prod_{\underset{U \neq V, W}{U \in \pi}} c_{|U|}
     \right) . \]
\end{lemma}

\begin{proof}
  In Lemma \ref{24} we proved, for every $k \in \mathbb{N}$, that
  \[ \sum_{m = 0}^{k - 1} \frac{k!}{m! (m + 1) ! (k - m - 1) !}  \left.
     \frac{\mathd^m}{\mathd x^m} ((\Psi' (x))^{k - m - 1} \Tau' (x))
     \right|_{x = 0} = \frac{1}{2} \sum_{\pi \in \tmop{NC} (k)} \sum_{V \in
     \pi} c''_{|V|} \prod_{\underset{W \neq V}{W \in \pi}} c_{|W|} . \]
  Thus, it suffices to show that
   \begin{equation}   
   \sum_{m = 0}^{k - 2} \frac{k!}{m! (m + 1) ! (k - m - 2) !}  \left.
     \frac{\mathd^m}{\mathd x^m} ((\Psi' (x))^{k - m - 2} (\Phi' (x))^2)
     \right|_{x = 0} = \sum_{\pi \in \tmop{NC} (k)} \sum_{\underset{W
    \neq V}{V, W \in \pi}} c'_{|V|} c'_{|W|} \prod_{\underset{U \neq V, W}{U
    \in \pi}} c_{|U|}. \label{thatsmagic} 
    \end{equation}
  By Leibniz rule, the left hand side of (\ref{thatsmagic}) is equal to
  \begin{equation} 
  \sum_{m = 0}^{k - 2} \sum_{l_{1 + \cdots + l_{k - m} = m}} \frac{k!}{(m +
     1) ! (k - m - 2) ! \prod_{i = 1}^{k - m} l_i !} \Phi^{(l_{k - m} + 1)} (0) \Phi^{(l_{k - m-1} + 1)} (0) \prod_{i = 1}^{k - m - 2} \Psi^{(l_i + 1)} (0).
    \label{jewdogs} 
    \end{equation}
  The way that we related $(k - m)$-tuples $(l_1, \ldots, l_{k - m})$ such
  that $l_1 + \cdots + l_{k - m} = m$ with non-crossing partitions $\pi =
  \{V_1, \ldots, V_{k - m} \} \in \tmop{NC} (k)$ in Lemma \ref{24}
  will lead to the equality (\ref{thatsmagic}). In more detail, for $m \in
  \{0, \ldots, k - 2\}$, consider non-negative integers $\lambda_1,
  \ldots, \lambda_{k - m}$ such that $\lambda_1 + \cdots + \lambda_{k -
  m} = m$. We also assume that $\{\lambda_1, \ldots, \lambda_{k - m} \} =
  \{\nu_1, \ldots, \nu_a \}$, where $\nu_i \neq \nu_j$, for every $i \neq j$, and
  we define for every $i \in \{1, \ldots, a\}$
  \[ r_i \assign |\{j = 1, \ldots, k - m \of \lambda_j = \nu_i \}| . \]
  We would like to calculate the contribution of $(\lambda_1, \ldots,
  \lambda_{k - m})$ to the latter sum of (\ref{jewdogs}). In order to do so, we
  consider two cases. If we assume $l_{k - m} \neq l_{k - m - 1}$, the $(k - m)$-tuple $(\lambda_1, \ldots, \lambda_{k - m})$
  contributes terms of the form
  \[ \frac{k!}{(m + 1) ! (k - m - 2) ! \prod_{t = 1}^{k - m} \lambda_t !}
     \Phi^{(\nu_i + 1)} (0) \Phi^{(\nu_j + 1)} (0) (\Psi^{(\nu_i + 1)} (0))^{r_i - 1}
     (\Psi^{(\nu_j + 1)} (0))^{r_j - 1} \prod_{\underset{t \neq i, j}{t =
     1}}^a (\Psi^{(\nu_t + 1)} (0))^{r_t}, \]
  where $i, j \in \{1, \ldots, a\}$, with $i < j$. The number of times that
  such a term will appear in the latter sum of (\ref{jewdogs}) is equal to the
  number of different ways that we can cover $k - m - 2$ points on a line
  segment with $r_t$ number of elements $\nu_t$ ($t \in \{1, \ldots, a\}
  \backslash \{i, j\}$), $r_i - 1$ number of elements $\nu_i$ and $r_j - 1$
  number of elements $\nu_j$. This is equal to
  \[ C_{i, j} \assign \frac{(k - m - 2) !}{r_1 ! \ldots r_{i - 1} ! (r_i - 1)
     !r_{i + 1} ! \ldots r_{j - 1} ! (r_j - 1) !r_{j + 1} ! \ldots r_a !} . \]
  Thus, we see that the $(k - m)$-tuple $(\lambda_1, \ldots, \lambda_{k - m})$
  contributes to the latter sum of (\ref{jewdogs}) the expression
  \[ 2 \frac{k!}{(m + 1) ! (k - m - 2) ! \prod_{t = 1}^a (\nu_t !)^{r_t}}
     \sum_{\underset{i < j}{i, j = 1}}^a C_{i, j} \Phi^{(\nu_i + 1)} (0)
     \Phi^{(\nu_j + 1)} (0)
   (\Psi^{(\nu_i + 1)} (0))^{r_i - 1}
     (\Psi^{(\nu_j + 1)} (0))^{r_j - 1} \prod_{\underset{t \neq i, j}{t =
     1}}^a (\Psi^{(\nu_t + 1)} (0))^{r_t} . \]
  By the definition of $(c_n)_{n \in \mathbb{N}}, (c'_n)_{n \in \mathbb{N}}$
  this is equal to
  \begin{equation}
    \frac{k!}{(m + 1) ! \prod_{i = 1}^a r_i !} \sum_{\underset{i \neq j}{i, j
    = 1}}^a r_i r_j c'_{\nu_i + 1} c'_{\nu_j + 1} (c_{\nu_i + 1})^{r_i - 1}
    (c_{\nu_j + 1})^{r_j - 1} \prod_{\underset{t \neq i, j}{t = 1}}^a
    (c_{\nu_t + 1})^{r_t} \label{monkeys} .
  \end{equation}
  In the second case, for $l_{k - m} = l_{k - m - 1}$, the contributions of
  $(\lambda_1, \ldots, \lambda_{k - m})$ to the latter sum of (\ref{jewdogs})
  will be of the form
  \[ \frac{k!}{(m + 1) ! (k - m - 2) ! \prod_{i = 1}^{k - m} \lambda_i !}
     (\Phi^{(\nu_j + 1)} (0))^2 (\Psi^{(\nu_j + 1)} (0))^{r_j - 2}
     \prod_{\underset{i \neq j}{i = 1}}^a (\Psi^{(\nu_i + 1)} (0))^{r_i}, \]
  for some $j \in \{1, \ldots, a\}$. The number of times that this term will
  appear in the latter sum of (\ref{jewdogs}) is equal to the number of
  different ways that we can cover $k - m - 2$ points on a line segment with
  $r_i$ number of elements $\nu_i$ ($i \in \{1, \ldots, a\} \backslash \{j\}$)
  and $r_j - 2$ number of elements $\nu_j$, i.e.
  \[ \frac{(k - m - 2) !}{r_1 ! \ldots r_{j - 1} ! (r_j - 2) !r_{j + 1} !
     \ldots r_a !} . \]
  Thus, $(\lambda_1, \ldots, \lambda_{k - m})$ will further contribute to the
  latter sum of (\ref{jewdogs}) the expression
  \[ \frac{k!}{(m + 1) ! \prod_{i = 1}^a (\nu_i !)^{r_i}}  \sum_{i = 1}^a
     \frac{(\Phi^{(\nu_i + 1)} (0))^2 (\Psi^{(\nu_i + 1)} (0))^{r_i - 2}}{r_1
     ! \ldots r_{i - 1} ! (r_i - 2) !r_{i + 1} ! \ldots r_a !}
     \prod_{\underset{j \neq i}{j = 1}}^a (\Psi^{(\nu_j + 1)} (0))^{r_j} \]
  \begin{equation}
    = \frac{k!}{(m + 1) ! \prod_{i = 1}^a r_i !} \sum_{i = 1}^a r_i (r_i - 1)
    (c'_{\nu_i + 1})^2 (c_{\nu_i + 1})^{r_i - 2} \prod_{\underset{j \neq i}{j
    = 1}}^a (c_{\nu_j + 1})^{r_j} \label{Ziz} .
  \end{equation}
  The total contribution of $(\lambda_1, \ldots, \lambda_{k - m})$ is given by
  the sum of (\ref{monkeys}) and (\ref{Ziz}). 
  
  Regarding the right hand
  side of (\ref{thatsmagic}), a partition $\pi = \{V_1, \ldots, V_{k - m} \}
  \in \tmop{NC} (k)$ which has $r_1$ blocks with $\nu_1 + 1$
  elements,{\textdots}, $r_a$ blocks with $\nu_a + 1$ elements, contributes
  \[ \sum_{\underset{i \neq j}{i, j = 1}}^{k - m} c'_{|V_i |} c'_{|V_j |}
     \prod_{\underset{t \neq i, j}{t = 1}}^{k - m} c_{|V_t |} =
     \sum_{\underset{i \neq j \text{ and } |V_i | = |V_j |}{i, j = 1}}^{k - m}
     c'_{|V_i |} c'_{|V_j |} \prod_{\underset{t \neq i, j}{t = 1}}^{k - m}
     c_{|V_t |} + \sum_{\underset{|V_i | \neq |V_j |}{i, j = 1}}^{k - m}
     c'_{|V_i |} c'_{|V_j |} \prod_{\underset{t \neq i, j}{t = 1}}^{k - m}
     c_{|V_t |} \]
  \[ = \sum_{i = 1}^a r_i (r_i - 1) (c'_{\nu_i + 1})^2 (c_{\nu_i + 1})^{r_i -
     2} \prod_{\underset{j \neq i}{j = 1}}^a (c_{\nu_j + 1})^{r_j}
   + \sum_{\underset{i \neq j}{i, j = 1}}^a r_i r_j
     c'_{\nu_i + 1} c'_{\nu_j + 1} (c_{\nu_i + 1})^{r_i - 1} (c_{\nu_j +
     1})^{r_j - 1} \prod_{\underset{t \neq i, j}{t = 1}}^a (c_{\nu_t +
     1})^{r_t} . \]
  Since the number of such partitions is equal to $\left( (m + 1) ! \prod_{i =
  1}^a r_i ! \right)^{- 1} k!$, this proves (\ref{thatsmagic}) and the claim
  holds.
\end{proof}

In the remainder of this section we make some explicit computations, which
illustrate how second order fluctuations of random matrix models are described by certain signed measures on $\mathbb{R}$. Our models consist of additive
perturbations of ergodic unitarily invariant matrices.

\begin{example}
  \label{olesoigeitonies}We consider the Hermitian random matrix of size $N$,
  $A = N \cdot A_1 + \sqrt{N} A_2 + A_3$, where $A_1, A_2, A_3$ are
  independent GUE matrices. As a consequence, for every $H \in H (N)$, one has
  \[ \mathbb{E} [\exp (\tmop{Tr} (H \text{} A))] = \prod_{x \in \tmop{Spec}
     (H)} \exp \left( N \frac{x^2}{2} + \frac{x^2}{2} + \frac{1}{N} 
     \frac{x^2}{2} \right) . \]
  The $1 / N^2$ correction of the matrix $A_1$ is given by
  \[ \frac{1}{12} \sum_{m = 0}^{k - 4} \frac{k!}{m! (m + 2) ! (k - m - 4) !} 
     \left. \frac{\mathd^m}{\mathd x^m} (x^{k - m - 4}) \right|_{x = 0} =
     \frac{k (k - 1)}{24 \mathpi \mathi} \sum_{m = 0}^{k - 4} \binom{k - 2}{m
     + 2} \oint_{|z| = 1} \frac{z^{k - m - 4}}{z^{m + 1}} d \text{} z \]
  \[ = \frac{k (k - 1)}{24 \mathpi \mathi} \oint_{|z| = 1} z^{k - 1} \sum_{m =
     2}^{k - 2} \binom{k - 2}{m} \left( \frac{1}{z^2} \right)^m d \text{} z =
     \frac{k (k - 1)}{24 \mathpi \mathi} \oint_{|z| = 1} z \left( z +
     \frac{1}{z} \right)^{k - 2} d \text{} z \]
  for every $k \geq 4$. Using the change of variables $u = 2 \cos \text{} t$
  in order to compute the last contour integral, we see that the second order
  correction functional of $A_1$
  \[ \phi'' (P) \assign \lim_{N \rightarrow \infty} N^2 \left( \frac{1}{N}
     \mathbb{E}[\tmop{Tr} \text{} P (A_1)] - \frac{1}{2 \mathpi} \int_{- 2}^2
     P (t) \sqrt{4 - t^2} d \text{} t \right) \text{, \quad} P \in \mathbb{C}
     \langle \mathbf{x} \rangle, \]
  is equal to
  \[ \phi'' (P) = \frac{1}{24 \mathpi} \int_{- 2}^2 P'' (t) \frac{t^2 -
     2}{\sqrt{4 - t^2}} d \text{} t, \text{\quad for every } P \in
     \mathbb{C} \langle \mathbf{x} \rangle . \]
  A similar formula for the second order correction functional of $A_1 +
  \frac{1}{\sqrt{N}} A_2 + \frac{1}{N} A_3$ can be derived if we consider an
  integral representation for the terms that correspond to the
  infinitesimal non-commutative probability space of order $2$. One of this
  terms is
  \[ \frac{1}{2} \sum_{m = 0}^{k - 2} \frac{k!}{m! (m + 1) ! (k - m - 2) !} 
     \left. \frac{\mathd^m}{\mathd x^m} (x^{k - m}) \right|_{x = 0} =
     \frac{k}{4 \mathpi \mathi} \sum_{m = 0}^{k - 2} \binom{k - 1}{m + 1}
     \oint_{|z| = 1} \frac{z^{k - m}}{z^{m + 1}} d \text{} z \]
  \[ = \frac{k}{4 \mathpi \mathi} \oint_{|z| = 1} z^{k + 1} \sum_{m = 1}^{k -
     1} \binom{k - 1}{m} \left( \frac{1}{z^2} \right)^m d \text{} z =
     \frac{k}{4 \mathpi \mathi} \oint_{|z| = 1} z^2 \left( z + \frac{1}{z}
     \right)^{k - 1} d \text{} z, \]
  for every $k \geq 2$. The remaining term is the first order correction of
  $A_1 + \frac{1}{\sqrt{N}} A_2$ and it was studied in Example \ref{ZNisxoli}. Concluding,
  the second order correction functional $\varphi''$ of $A_1 +
  \frac{1}{\sqrt{N}} A_2 + \frac{1}{N} A_3$ is given by
  \[ \varphi'' (P) = \frac{1}{2 \mathpi} \int_{- 2}^2 \left( P'' (t) \frac{t^2
     - 2}{12} + P' (t) \frac{t^3 - 3 t}{2} + P (t) (t^2 - 2) \right)
     \frac{1}{\sqrt{4 - t^2}} d \text{} t, \text{\quad for every } P \in
     \mathbb{C} \langle \mathbf{x} \rangle . \]
\end{example}

\begin{example}
  We consider the random matrix
  \begin{equation}
    A = N \cdot G + \theta_1 X \text{} X^{\ast} + \frac{\theta_2}{N} Y \text{}
    Y^{\ast}, \label{1111}
  \end{equation}
  where the matrices $G, X, Y$ are independent, $G$ is a GUE of size $N$ and
  $X, Y$ are independent $N \times 1$ vectors with standard complex Gaussian i.i.d. entries. Then, in notation of Theorem \ref{th:sec-order-main},
  we have that
  \[ \Psi (x) = \frac{x^2}{2}, \text{{\hspace{3em}}} \Phi (x) = \log (1 -
     \theta_1 x)^{- 1}, \text{{\hspace{3em}}} \Tau (x) = \theta_2 x. \]
     Applying the theorem and performing very similar to previous examples calculations, one arrives to the following answer for the corrections. 
  Consider the infinitesimal non-commutative probability space of order 2,
  $(\mathbb{C} \langle X \rangle, \varphi^{(0)}, \varphi^{(1)},
  \varphi^{(2)})$, where $\varphi^{(0)}$ is the limit, $\varphi^{(1)}$ is the
  first order correction and $\varphi^{(2)}$ is the second order correction.
  Then we have that
  \[ \varphi^{(1)} (P) =\tmmathbf{1}_{| \theta_1 | \geq 1} \cdot P (\theta_1 +
     \theta_1^{- 1}) - \frac{1}{2 \pi} \int_{- 2}^2 \frac{t - 2 \theta_1}{(t -
     \theta_1 - \theta_1^{- 1}) \sqrt{4 - t^2}} P (t) d \text{} t, \]
  and
  \[ \varphi^{(2)} (P) = \frac{\theta_2}{2 \mathpi} \int_{- 2}^2 \sqrt{4 -
     t^2} P' (t) d \text{} t + \frac{1}{2 \mathpi} \int_{- 2}^2
     \frac{\theta_1^2}{\sqrt{4 - t^2}}  \frac{P' (t) - P' (\theta_1 +
     \theta^{- 1}_1)}{t - \theta_1 - \theta_1^{- 1}} d \text{} t \]
  \[ - \frac{1}{4 \mathpi} \int_{- 2}^2 \frac{(\theta_1^2 - 1) (t - 2
     \theta_1)}{(t - \theta_1 - \theta_1^{- 1}) \sqrt{4 - t^2}} \left(
     \frac{P' (t) - P' (\theta_1 + \theta_1^{- 1})}{t - \theta_1 - \theta_1^{-
     1}} - P'' (\theta_1 + \theta_1^{- 1}) \right) d \text{} t + \frac{1}{24
     \mathpi} \int_{- 2}^2 P'' (t) \frac{t^2 - 2}{\sqrt{4 - t^2}} d \text{} t.
  \]
  
  Instead of $\frac{\theta_2}{N} Y \text{} Y^{\ast}$ in (\ref{1111}), one can
  consider the deterministic Hermitian matrix with one non-zero eigenvalue
  $\theta_2$. This would lead to the same result for corrections.
\end{example}

\section{Higher order infinitesimal free probability}
\label{sec:higher-order}

In this section we present our most general result, focusing on corrections of arbitrary order
of certain random matrices. Our goal is to provide explicit
computations that allow us to make the connection with non-crossing cumulant
functionals of $k$-th order.
As before, we are interested in an asymptotic expansion
\begin{equation}
  \frac{1}{N^{k + 1}} \mathbb{E} [\tmop{Tr} (A^k)] = \sum_{i = 0}^k
  \frac{1}{N^i} M_{i, N} (\Psi), \label{GOVA}
\end{equation}
where $(M_{i, N} (\Psi))_{N \in \mathbb{N}}$ converges for every $i$. We would like to understand under which further conditions on $A$ one can control the terms
$M_{i, N} (\Psi)$, for $i \in \{0, \ldots, k\}$, as $N \rightarrow
\infty$. 

\begin{theorem}
  \label{fredi}
  Let $n$ be a fixed integer. Let $A$ be a random Hermitian matrix of size $N$ and assume
  that for every finite $r$ one has
  \begin{equation}
    \lim_{N \rightarrow \infty} N^{n \varepsilon} \left( \frac{1}{N} \log
    \mathbb{E}[H \text{} C (x_1, \ldots, x_r, 0^{N - r} ; \lambda_1 (A),
    \ldots, \lambda_n (A))] - \sum_{j = 0}^{n - 1} \sum_{i = 0}^r
    \frac{1}{N^{j \varepsilon}} \Psi_j (x_i) \right) = \sum_{i = 1}^r \Psi_n
    (x_i), \label{100}
  \end{equation}
  where $\Psi_0, \ldots, \Psi_n$ are smooth functions in a complex neighborhood of
  $0$, the above convergence is uniform in a complex neighborhood of $0^r$ and $0 <
  \varepsilon < n^{- 1}$. Then the $1 / N^{n \varepsilon}$ correction of the
  average empirical distribution of $N^{- 1} A$ is given by
  \begin{equation}
    \lim_{N \rightarrow \infty} N^{n \varepsilon} \left( \frac{1}{N^{k + 1}}
    \mathbb{E}[\tmop{Tr} (A^k)] - \sum_{j = 0}^{n - 1} \frac{1}{N^{j
    \varepsilon}} \tmmathbf{\mu}^{(j)}_k \right) =\tmmathbf{\mu}_k^{(n)},
    \label{PORDOKOFTIS}
  \end{equation}
  where for every $j \in \{0, \ldots, n\}$
  \[ \tmmathbf{\mu}^{(j)}_k \assign \sum_{b_1 + 2 b_2 + \cdots + j \text{} b_j
     = j} \sum_{m = 0}^{k - b_1 - \cdots - b_j} \frac{k!}{m! (m + 1) ! (k - m
     - b_1 - \cdots - b_j) !b_1 ! \ldots b_j !} \text{{\hspace{8em}}} \]
  \begin{equation}
    \text{{\hspace{14em}}} \times \left. \frac{\mathd^m}{\mathd x^m} ((\Psi'_0
    (x))^{k - m - b_1 - \cdots - b_j} (\Psi'_1 (x))^{b_1} \ldots (\Psi'_j
    (x))^{b_j}) \right|_{x = 0} . \label{tavradiamouleipis}
  \end{equation}
\end{theorem}

\begin{proof}
  First, we consider an easier case, in which we additionally assume that 
  \begin{equation}
    N^{n \varepsilon} \left( \frac{1}{N} \log \mathbb{E}[H \text{} C (x_1,
    \ldots, x_N ; \lambda_1 (A), \ldots, \lambda_N (\Alpha))] - \sum_{j =
    0}^{n - 1} \sum_{i = 0}^N \frac{1}{N^{j \varepsilon}} \Psi_j (x_i) \right)
    = \sum_{i = 1}^N \Psi_n (x_i), \label{rithmo}
  \end{equation}
  for every $N \in \mathbb{N}$ and $x_1, \ldots, x_N \in \mathbb{R}$. Thus,
  denoting the average Harish-Chandra integral of $A$ by $f_N$, we have that
  $f_N = f_0 f_1 \ldots f_n$, where
  \[ f_j (x_1, \ldots, x_N) = \exp \left( N^{1 - j \varepsilon} \sum_{i = 1}^N
     \Psi_j (x_i) \right), \]
  for $j = 0, \ldots, n$. Similarly to previous theorems, we will pass from the
  Harish-Chandra integral of $A$ to the $k$-th moment of the average empirical
  distribution of $A / N$ using the differential operator $\mathcal{D}_k$.
  We have that $\mathbb{E} (\tmop{Tr}
  (A^k)) =\mathcal{D}_k f_N |_{x_i = 0}$ and
  \begin{equation}
    \mathcal{D}_k f_N = \sum_{m = 0}^{k - 1} \sum_{\underset{l_i \neq l_j
    \text{ for } i \neq j}{l_0, \ldots, l_m = 1}}^N \sum_{b_0 + \cdots + b_n =
    k - m} \binom{k}{m} \frac{(k - m) !}{b_0 ! \ldots b_n !} 
    \frac{\partial_{l_0}^{b_0} f_0 \partial_{l_0}^{b_1} f_1 \ldots
    \partial^{b_n}_{l_0} f_n}{(x_{l_0} - x_{l_1}) \ldots (x_{l_0} - x_{l_m})}
    . \label{anemoi}
  \end{equation}
  As a consequence, $\mathcal{D}_k f_N$ is a linear combination of terms
  \begin{equation} 
  \text{{\hspace{3em}}} f_N \frac{\prod_{i = 0}^n \prod_{j = 1}^{b_i}
     (\Psi_i^{(j)} (x_{l_0}))^{\lambda_{i, j}}}{(x_{l_0} - x_{l_1}) \ldots
     (x_{l_0} - x_{l_m})} + f_N \frac{\prod_{i = 0}^n \prod_{j = 1}^{b_i}
     (\Psi_i^{(j)} (x_{l_1}))^{\lambda_{i, j}}}{(x_{l_1} - x_{l_0}) (x_{l_1} -
     x_{l_2}) \ldots (x_{l_1} - x_{l_m})} + \ldots + f_N \frac{\prod_{i = 0}^n \prod_{j =
    1}^{b_i} (\Psi_i^{(j)} (x_{l_m}))^{\lambda_{i, j}}}{(x_{l_m} - x_{l_0})
    \ldots (x_{l_m} - x_{l_{m - 1}})}, \label{polakis}
     \end{equation}
  where these terms have a limit for $x_1 = \cdots = x_N = 0$, due to Lemma \ref{SnikfeatFy}.
  Note that all the coefficients in $\mathcal{D}_k f_N |_{x_i = 0}$ will
  be of order $O (N^{k + 1})$. In order to prove the claim, for $x_1 = \cdots = x_N =
  0$ we have to compute the terms (\ref{polakis}), where their coefficients are
  not $\omicron (N^{k + 1 - n \varepsilon})$. Since $0 <
  \varepsilon < n^{- 1}$ and $b_0 + \cdots + b_n = k - m$, note that
  differentiating the functions $f_0, \ldots, f_n$ in the expression
  \[ \binom{k}{m} \frac{(k - m) !}{b_0 ! \ldots b_n !} 
     \frac{\partial_{l_0}^{b_0} f_0 \partial_{l_0}^{b_1} f_1 \ldots
     \partial_{l_0}^{b_n} f_n}{(x_{l_0} - x_{l_1}) \ldots (x_{l_0} - x_{l_m})}
  \]
  can be written as a sum, where only the summand
  \[ \frac{k!}{m!b_0 ! \ldots b_n !} N^{b_0 + b_1 (1 - \varepsilon) + \cdots +
     b_n (1 - n \varepsilon)} f_N \frac{(\Psi'_0 (x_{l_0}))^{b_0} \ldots
     (\Psi'_n (x_{l_0}))^{b_n}}{(x_{l_0} - x_{l_1}) \ldots (x_{l_0} -
     x_{l_m})} \]
  may contribute to a term (\ref{polakis}) with a coefficient that is not
  $\omicron (N^{k + 1 - n \varepsilon})$. Indeed, it follows from the equality
  \[ \lim_{x_1, \ldots, x_N \rightarrow 0} \left( \frac{(\Psi'_0
     (x_{l_0}))^{b_0} \ldots (\Psi'_n (x_{l_0}))^{b_n}}{(x_{l_0} - x_{l_1})
     \ldots (x_{l_0} - x_{l_m})} + \frac{(\Psi'_0 (x_{l_1}))^{b_0} \ldots
     (\Psi'_n (x_{l_1}))^{b_n}}{(x_{l_1} - x_{l_0}) (x_{l_1} - x_{l_2}) \ldots
     (x_{l_1} - x_{l_m})} + \right. \text{{\hspace{8em}}} \]
  \[ \text{{\hspace{8em}}} \left. \ldots + \frac{(\Psi'_0 (x_{l_m}))^{b_0}
     \ldots (\Psi'_n (x_{l_m}))^{b_n}}{(x_{l_m} - x_{l_0}) \ldots (x_{l_m} -
     x_{l_{m - 1}})} \right) = \frac{1}{m!}  \left. \frac{\mathd^m}{\mathd
     x^m} ((\Psi'_0 (x))^{b_0} \ldots (\Psi'_n (x))^{b_n}) \right|_{x = 0}, \]
  in which the right hand side does not depend on $l_0, \ldots, l_m \in \{1, \ldots, N\}$. Thus, the above limit appears in the right hand side of (\ref{anemoi}) exactly $m! \binom{N}{m
  + 1}$ times. Therefore, the summands (\ref{polakis}) of
  $\mathcal{D}_k f_N |_{x_i = 0}$ with a coefficient that it is not $\omicron (N^{k +
  1 - n \varepsilon})$ have the form
  \begin{equation}
    \frac{k!}{m! (m + 1) !b_0 ! \ldots b_n !} N^{k + 1 - b_1 \varepsilon -
    \cdots - n \text{} b_n \varepsilon} \left. \frac{\mathd^m}{\mathd x^m}
    ((\Psi'_0 (x))^{b_0} \ldots (\Psi'_n (x))^{b_n}) \right|_{x = 0},
    \label{VERAMAN}
  \end{equation}
  where $b_0 = k - m - b_1 - \cdots - b_n$ and $b_1 + 2 b_2 + \cdots + n
  \text{} b_n \leq n$. Thus, the summands that we are interested in have a factor $N^{k
  + 1 - j \varepsilon}$, for $j = 0, \ldots, n$. Taking the sum of those of the form
  (\ref{VERAMAN}) with the same factor $N^{k + 1 - j \varepsilon}$ (i.e., for
  $b_1 + 2 b_2 + \cdots + jb_j = j$, $b_{j + 1} = \cdots = b_n = 0$, and $m =
  0, \ldots, k - b_1 - \cdots - b_j$), we get that
  \begin{equation}
    \frac{1}{N^{k + 1}} \mathbb{E} [\tmop{Tr} (A^k)] = \frac{1}{N^{k + 1}}
    \mathcal{D}_k (f_0 \ldots f_n) |_{x_i = 0} =\tmmathbf{\mu}_k^{(0)} +
    \frac{1}{N^{\varepsilon}} \tmmathbf{\mu}_k^{(1)} + \cdots + \frac{1}{N^{n
    \varepsilon}} \tmmathbf{\mu}_k^{(n)} + o (N^{- n \varepsilon}), \label{DUB}
  \end{equation}
  and the claim holds.
  
  Now we deal with a more general case, in which only relation (\ref{100}) is
  known for $A$. Again, we will use the differential operator
  $\mathcal{D}_k$ in order to pass from the Harish-Chandra integral of $A$ to
  the $k$-th moment of the average empirical distribution of $A / N$. Taking
  into account (\ref{100}), it will be useful for our analysis to write the
  average Harish-Chandra integral as $f_N = F_N \cdot G_N$,
  where $F_N \assign \exp (N \cdot \widetilde{F_N})$, where
  \[ \widetilde{F_N} (x_1, \ldots, x_N) \assign \frac{1}{N} \log \text{} f_N
     (x_1, \ldots, x_N) - \sum_{j = 0}^n \sum_{i = 1}^N \frac{1}{N^{j
     \varepsilon}} \Psi_j (x_i) \]
  and
  \[ G_N (x_1, \ldots, x_N) \assign \prod_{j = 0}^n \exp \left( N^{1 - j
     \varepsilon} \sum_{i = 1}^N \Psi_j (x_i) \right) . \]
  We write $G_N = f_0 f_1 \ldots f_n$, where $f_j \assign \exp (N^{1 - j
  \varepsilon} g_j)$ and $g_j (x_1, \ldots, x_N) \assign \Psi_j (x_1) + \cdots
  + \Psi_j (x_N)$, for every $j = 0, \ldots, n$. Then, applying the operator
  $\mathcal{D}_k$ to $f_N$, by Proposition \ref{prop:1} and Lemma \ref{lem:2} we have
  \[ \mathbb{E} \left[ \left( \sum_{i = 1}^N \lambda_i^k (A) \right) H \text{}
     C (x_1, \ldots, x_N ; \lambda_1 (A), \ldots, \lambda_N (A)) \right] \]
  \begin{equation}
    = \left. \sum_{m = 0}^{k - 1} \sum_{\underset{l_i \neq l_j \text{ for } i \neq
    j}{l_0, \ldots, l_m = 1}}^N \sum_{\nu = 0}^{k - m} \sum_{b_0 + \cdots +
    b_n = k - m - \nu} \binom{k}{m} \binom{k - m}{\nu} \frac{(k - m - \nu)
    !}{b_0 ! \ldots b_n !} \frac{\partial_{l_0}^{\nu} F_N \partial_{l_0}^{b_0}
    f_0 \ldots \partial_{l_0}^{b_n} f_n}{(x_{l_0} - x_{l_1}) \ldots (x_{l_0} -
    x_{l_m})} \right|_{x_1,x_2, \dots, x_N=0} \label{AKRIVA}.
  \end{equation}
  Using the chain rule, we obtain
  that $\partial_{l_0}^{\nu} F_N, \partial_{l_0}^{b_0} f_0, \ldots,
  \partial_{l_0}^{b_n} f_n$ can be written as a linear combination of products
  of factors $F_N, f_0, \ldots, f_n$ and derivatives of $\widetilde{F_N},
  g_0, \ldots, g_n$ with respect to $x_{l_0}$, where the
  coefficients depend on $N$. Writing in such a way $\partial_{l_0}^{\nu} F_N
  \partial_{l_0}^{b_0} f_0 \ldots \partial_{l_0}^{b_n} f_n$ as a sum, we will
  show that the contribution of each summand to (\ref{AKRIVA}) is equal to
  zero, for $\nu > 0$, in the limit $x_1, \ldots, x_N \rightarrow 0$ and $N
  \rightarrow \infty$. We consider the summands 
  $$
  N^{\nu} F_N (\partial _{l_0}
  \widetilde{F_N})^{\nu}, N^{b_0} f_0  (\partial_{l_0} g_0)^{b_0}, N^{b_1 (1 -
  \varepsilon)} f_1  (\partial_{l_0} g_1)^{b_1}, \ldots, N^{b_n (1 - n
  \varepsilon)} f_n  (\partial_{l_0} g_n)^{b_n}
  $$ of 
  $\partial_{l_0}^{\nu} F_N,
  \partial_{l_0}^{b_0} f_0, \ldots, \partial_{l_0}^{b_n} f_n.
  $
  We will show that that the terms
  \[ f_N \frac{(\partial_{l_0} \widetilde{F_N})^{\nu} (\partial_{l_0}
     g_0)^{b_0} \ldots (\partial_{l_0} g_n)^{b_n}}{(x_{l_0} - x_{l_1}) \ldots
     (x_{l_0} - x_{l_m})} + f_N \frac{(\partial_{l_1} \widetilde{F_N})^{\nu}
     (\partial_{l_1} g_0)^{b_0} \ldots (\partial_{l_1} g_n)^{b_n}}{(x_{l_1} -
     x_{l_0}) (x_{l_1} - x_{l_2}) \ldots (x_{l_1} - x_{l_m})} + \]
  \begin{equation}
    \text{\qquad} \ldots + f_N \frac{(\partial_{l_m} \widetilde{F_N})^{\nu}
    (\partial_{l_m} g_0)^{b_0} \ldots (\partial_{l_m} g_n)^{b_n}}{(x_{l_m} -
    x_{l_0}) \ldots (x_{l_m} - x_{l_{m - 1}})} \label{iperboleS}
  \end{equation}
  that contribute to (\ref{AKRIVA}) can be controlled when $x_1, \ldots, x_N
  \rightarrow 0$. In order to do so, we can use a similar
  to the proof of Theorem \ref{SouvlakiA} argument. The first step is to consider $x_i = i \varepsilon$, for every $i = 1, \ldots, N$. For every $i = 1, \ldots, N$, we consider the
  functions
  \[ h_i (x_1, \ldots, x_N) \assign [(\partial_i \widetilde{F_N})^{\nu}
     (\partial_i g_0)^{b_0} \ldots (\partial_i g_n)^{b_n}] (x_1, \ldots, x_N)
  \]
  and $H_i (\varepsilon) \assign h_i (\varepsilon, 2 \varepsilon, \ldots, N
  \varepsilon)$. Then, if we consider for every $i = 1, \ldots, N$ the Taylor
  expansion
  \[ H_i (\varepsilon) = H_i (0) + H'_i (0) \varepsilon + \cdots +
     \frac{\varepsilon^m}{m!} H_i^{(m)} (0) + \varepsilon^m \widetilde{H_i}
     (0), \]
  where $\lim_{\varepsilon \rightarrow 0} \widetilde{H_i} (\varepsilon) = 0$,
  we can control (\ref{iperboleS}) as $\varepsilon \rightarrow 0$, because
  \[ \frac{H_{l_0}^{(l)} (0)}{(l_0 - l_1) \ldots (l_0 - l_m)} +
     \frac{H_{l_1}^{(l)} (0)}{(l_1 - l_0) (l_1 - l_2) \ldots (l_1 - l_m)} +
     \cdots + \frac{H_{l_m}^{(l)} (0)}{(l_m - l_0) \ldots (l_m - l_{m - 1})} =
     0, \]
  for $l < m$. Indeed, one has
  \[ H_i^{(l)} (0) = \sum_{i_1, \ldots, i_l = 1}^N i_1 \ldots i_l
     \partial_{i_l} \ldots \partial_{i_1} h_i (0^N), \]
  and since the functions $\widetilde{F_N}, g_0, \ldots, g_n$ are symmetric
  for two $(l + 1)$-tuples $(\alpha_1, \ldots, \alpha_{l + 1})$ and $(\beta_1,
  \ldots, \beta_{l + 1})$ of elements of $\{1, \ldots, N\}$, we have
  \[ \partial_{\alpha_1} \ldots \partial_{\alpha_l} h_{\alpha_{l + 1}} (0^N) =
     \partial_{\beta_1} \ldots \partial_{\beta_l} h_{\beta_{l + 1}} (0^N), \]
  where $\alpha_i = \alpha_j \Leftrightarrow \beta_i = \beta_j$, for every $i,
  j$. Thus, $H_i^{(l)} (0)$ is a polynomial in $i$ of degree $l$ and in the limit
  $\varepsilon \rightarrow 0$ the expression (\ref{iperboleS}) is equal to a linear
  combination of derivatives
  \[ \partial_{i_m} \ldots \partial_{i_1} [(\partial_i \widetilde{F_N})^{\nu}
     (\partial_i g_0)^{b_0} \ldots (\partial_i g_n)^{b_n}] (0^N), \]
  where $i, i_1, \ldots, i_m \in \{1, \ldots, m + 1\}$, and the coefficients
  do not depend on $N$. This shows that as $\varepsilon \rightarrow 0$,
  the expression (\ref{iperboleS}) does not depend on $l_0, \ldots, l_m \in \{1, \ldots,
  N\}$. Therefore, the above linear combinations of derivatives will
  appear $m! \binom{N}{m + 1}$ times in (\ref{AKRIVA}), as $\varepsilon
  \rightarrow 0$. Thus, the summands
  \begin{equation}
    N^{\nu + b_0 + b_1 (1 - \varepsilon) + \cdots + b_n (1 - n \varepsilon)}
    f_N  (\partial_{l_0} \widetilde{F_N})^{\nu}  (\partial_{l_0} g_0)^{b_0}
    \ldots (\partial_{l_0} g_n)^{b_n} \label{ILEO}
  \end{equation}
  of $\partial_{l_0}^{\nu} F_N \partial_{l_0}^{b_0} f_0 \ldots
  \partial_{l_0}^{b_n} f_n$ will contribute to (\ref{AKRIVA}) by a factor
  $N^{n \varepsilon - (k + 1)}$, as $\varepsilon \rightarrow 0$, with a linear
  combination of terms
  \[ \left( \frac{N^{n \varepsilon}}{N^{\varepsilon b_1 + 2 \varepsilon b_2 +
     \cdots + n \varepsilon b_n}} + \omicron (1) \right) \partial_{i_m} \ldots
     \partial_{i_1} [(\partial_i \widetilde{F_N})^{\nu} (\partial_i g_0)^{b_0}
     \ldots (\partial_i g_n)^{b_n}] (0^N), \]
  for $b_0 + \cdots + b_n = k - m - \nu$ and $0 < \varepsilon < n^{- 1}$.
  For $\nu > 0$, due to (\ref{100}), we have
  \[ \lim_{N \rightarrow \infty} N^{n \varepsilon} \partial_{i_l} \ldots
     \partial_{i_1} \widetilde{F_N} (0^N) = 0, \]
  for all $l$, and we obtain that these linear combinations will not
  contribute to $N^{n \varepsilon - (k + 1)} \mathbb{E} [\tmop{Tr} (A^k)]$, as
  $N \rightarrow \infty$. Let $\psi_{l_0} (x_1, \ldots, x_N)$ be one of the
  remaining summands of $\partial_{l_0}^{\nu} F_N \partial_{l_0}^{b_0} f_0
  \ldots \partial_{l_0}^{b_n} f_n$, where they are products and their factors
  that depend from $x_1, \ldots, x_N$ are $f_N$ and derivarives of
  $\widetilde{F_N}, g_0, \ldots, g_n$ with respect to $x_{l_0}$. Then, we can
  use the same argument to control the terms
  \begin{equation}
    \frac{\psi_{l_0} (x_1, \ldots, x_N)}{(x_{l_0} - x_{l_1}) \ldots (x_{l_0} -
    x_{l_m})} + \cdots + \frac{\psi_{l_m} (x_1, \ldots, x_N)}{(x_{l_m} -
    x_{l_0}) \ldots (x_{l_m} - x_{l_{m - 1}})} \label{TRAFICANTE}
  \end{equation}
  that contribute to (\ref{AKRIVA}), when $x_i = i \varepsilon$ and
  $\varepsilon \rightarrow 0$. Similarly to the case that we analyzed above,
  where $\psi_{l_0} (x_1, \ldots, x_N)$ has a form (\ref{ILEO}), for the same reasons
  (\ref{TRAFICANTE}) will converge to a linear combination of products of
  derivatives of $\widetilde{F_N}, g_0, \ldots, g_n$ at $0^N$. For $\nu > 0$,
  we have that each of its summands will have as factor a derivative of
  $\widetilde{F_N}$ at $0^N$. Taking into account that as $\varepsilon
  \rightarrow 0$, (\ref{TRAFICANTE}) does not depend on $l_0, \ldots, l_m \in
  \{1, \ldots, N\}$ and that the factor of $\psi_{l_0} (x_1, \ldots, x_N)$
  that does not depend on $x_1, \ldots, x_N$ is $O (N^{k - m - 1})$, we obtain
  that this factor multiplied by $N^{n \varepsilon - (k + 1)} \binom{N}{m +
  1}$ will converge to zero as $N \rightarrow \infty$. Moreover, since all
  the derivatives of $\widetilde{F_N}$ at $0^N$ converge (to zero), due to
  (\ref{100}), the expression (\ref{TRAFICANTE}) will not contribute to $N^{n \varepsilon -
  (k + 1)} \mathbb{E} [\tmop{Tr} (A^k)]$, as $N \rightarrow \infty$. Thus,
  considering the case $\nu = 0$ in (\ref{AKRIVA}), as $N \rightarrow \infty$, we obtain
  that
  \[ \frac{N^{n \varepsilon}}{N^{k + 1}} \mathbb{E} [\tmop{Tr} (A^k)] =
     \frac{N^{n \varepsilon}}{N^{k + 1}} \mathcal{D}_k G_N |_{x_i = 0} +
     \omicron (1) . \]
  But since in (\ref{DUB}) we proved that
  \[ \frac{1}{N^{k + 1}} \mathcal{D}_k G_N |_{x_i = 0} =\tmmathbf{\mu}_k^{(0)}
     + \frac{1}{N^{\varepsilon}} \tmmathbf{\mu}_k^{(1)} + \cdots +
     \frac{1}{N^{n \varepsilon}} \tmmathbf{\mu}_k^{(n)} + \omicron (N^{- n
     \varepsilon}), \]
  we arrive at the claim of the theorem.
\end{proof}

\begin{remark}
  In the context of Theorem \ref{fredi}, note that for $\Psi \assign \Psi_0, \Phi
  \assign \Psi_1$ and $\Tau \assign \Psi_2$, we have that
  $\tmmathbf{\mu}_k^{(0)} =\tmmathbf{\mu}_k, \tmmathbf{\mu}_k^{(1)}
  =\tmmathbf{\mu}'_k$ and $\tmmathbf{\mu}^{(2)}_k \assign \tmmathbf{\mu}''_k$,
  for every $k \in \mathbb{N}$, where $\tmmathbf{\mu}_k, \tmmathbf{\mu}'_k,
  \tmmathbf{\mu}''_k$ were defined in (\ref{topg}), (\ref{G}), (\ref{Ss}),
  repsectively. Thus, for the average empirical distribution functional of $A
  / N$ we have, for every $P \in \mathbb{C} \langle \mathbf{x} \rangle$, that
  \[ \varphi_N (P) = \varphi (P) + \frac{1}{N^{\varepsilon}} \varphi' (P) +
     \frac{1}{N^{2 \varepsilon}}  \frac{\varphi'' (P)}{2} + \omicron (N^{- 2
     \varepsilon}),  \]
  where $(\mathbb{C} \langle \mathbf{x} \rangle, \varphi, \varphi',
  \varphi'')$ is the infinitesimal non-commutative probability space of order
  $2$. Hence we see that compared to Theorem \ref{th:sec-order-main}, the
  limit regime (\ref{100}) for the logarithm of the Harish-Chandra integral
  of $A$ gives an expansion for $\varphi_N$, where the rate of convergence
  that gives the correction to the non-commutative distribution of $A$ is
  $N^{\varepsilon}$ instead of $N$. As a corollary, the functional
  $\varphi''$ determines completely the second order correction since the term
  (\ref{sagan}) that contributes to the limit regime (\ref{audikos}) will not
  affect the $1 / N^{2 \varepsilon}$ correction. This holds because assumption
  (\ref{100}) implies the expansion (\ref{GOVA}) and the $1 / N^{n
  \varepsilon}$ correction (\ref{PORDOKOFTIS}) of $A / N$ will be given by
  \begin{equation}
    \lim_{N \rightarrow \infty} N^{n \varepsilon} \left( M_{0, N} (\Psi)
    -\tmmathbf{\mu}_k^{(0)} - \frac{1}{N^{\varepsilon}} \tmmathbf{\mu}_k^{(1)}
    - \cdots - \frac{1}{N^{(n - 1) \varepsilon}} \tmmathbf{\mu}_k^{(n - 1)}
    \right), \label{PORNIDIO!}
  \end{equation}
  since $0 < \varepsilon < n^{- 1}$. On the other hand, assuming the limit
  regime (\ref{100}) for $A$ with $\varepsilon = 1$, in order to compute the
  $1 / N^n$ correction, besides of (\ref{PORNIDIO!})
  there are extra limits that we have to consider. For example, for $n = 3$,
  the limits
  \[ \lim_{N \rightarrow \infty} N (M_{2, N} (\Psi) -\tmmathbf{\nu}''_k),
     \text{\quad} \lim_{N \rightarrow \infty} M_{3, N} (\Psi) . \]
  will also contribute non-trivially to the $1 / N^3$ correction
  \[ \lim_{N \rightarrow \infty} N^3 \left( \frac{1}{N^{k + 1}}
     \mathbb{E}[\tmop{Tr} (A^k)] -\tmmathbf{\mu}_k - \frac{1}{N}
     \tmmathbf{\mu}'_k - \frac{1}{N^2} (\tmmathbf{\mu}''_k
     +\tmmathbf{\nu}''_k) \right). \]
  Note also that for $n = 1$ the
  first order fluctuations that we get from the limit regime of Theorem \ref{Einaiiagapi}
  coincide with those of the $\varepsilon < 1$ regime because $\lim_{N
  \rightarrow \infty} M_{1, N} (\Psi) = 0$.
\end{remark}

\begin{remark}

Theorem \ref{fredi} (for $n = 1$) and Theorem \ref{Einaiiagapi} give the
same formula for the $1 / N^{\varepsilon}$ and the $1 / N$ correction,
respectively. However, this is not true for the second order correction, if we
compare Theorem \ref{fredi} (for $n = 2$) and Theorem \ref{th:sec-order-main}. 
A (technical) reason why this holds for
the first order corrections is that in the limit regime
\begin{equation}
  \frac{1}{N} \log \mathbb{E} [H \text{} C (x_1, \ldots, x_r, 0^{N - r} ;
  \mathlambda_1 (A), \ldots, \mathlambda_N (A))] \longrightarrow \Psi (x_1) +
  \cdots + \Psi (x_r), \label{1a1}
\end{equation}
applying the differential operator to the Harish-Chandra transform, one obtains (see the proofs of the theorems) an
asymptotic expansion
\begin{equation}
  \frac{1}{N^{k + 1}} \mathbb{E} [\tmop{Tr} (A^k)] = \sum_{i = 0}^k M_{i, k,
  N} (\Psi) \frac{1}{N^i}, \label{2a1}
\end{equation}
where $M_{1, k, N} (\Psi) = 0$. Therefore $M_{1, k, N} (\Psi)$ does not
contribute to the $1 / N$ correction of Theorem \ref{Einaiiagapi}. On the other hand the extra
term $\nu''$ of Theorem \ref{th:sec-order-main}, is the $N \rightarrow \infty$ limit of \{$M_{2, k,N} (\Psi)$\}.

Using similar arguments, one can check that $M_{3, k, N} (\Psi)$=0. More
generally, it is believed by the authors that in the asymptotic expansion
(\ref{2a1}) all the odd terms vanish. Therefore, (\ref{1a1}) leads to an asymptotic
expansion of a similar form as in, e.g., \cite{Par}.

\end{remark}

\begin{remark}
  One source of examples of random matrix models that satisfy the limit regime (\ref{100})
  are sums of ergodic unitarily invariant matrices. In more detail, for
  functions $\Psi_0, \ldots, \Psi_n$ of the form \eqref{eq:OV-classification} and $0<\varepsilon < n^{-1}$, we consider $n$ independent Hermitian random matrices $A_0, \ldots, A_n$ of size $N$, such that
  \[ \mathbb{E} [\tmop{Tr} (H \text{} A_j)] = \prod_{b \in \tmop{Spec} (H)}
     \exp (N^{1 - j \varepsilon} \Psi_j (b)), \text{\quad for every } H
     \in H (N), \]
  and $j = 0, \ldots, n$. Then $A \assign A_0 + \cdots + A_n$ satisfy the
  relation (\ref{rithmo}). Note that the formula (\ref{tavradiamouleipis}) for
  the $1 / N^{j \varepsilon}$ correction of $A / N$, for $j = 0, \ldots, n$,
  implies that the matrices $A_{j + 1}, \ldots, A_n$ do not affect the $1 /
  N^{i \varepsilon}$ correction of $A / N$, for $i = 0, \ldots, j$.
\end{remark}

Our next result provides an interpretation of the moment formulas \eqref{tavradiamouleipis} via infinitesimal free probability of order $n$.

\begin{theorem}
\label{th:mom-cum-high}
  Let $\Psi_0, \ldots, \Psi_n$ be infinitely differentiable functions at $0$
  and consider the sequences $(c_m^{(0)})_{m \in \mathbb{N}}$, $\ldots$,
  $(c^{(n)}_m)_{m \in \mathbb{N}}$, where for every $i = 0, \ldots, n$,
  \[ c_m^{(i)} \assign \frac{i! \Psi_i^{(m)} (0)}{(m - 1) !}, \text{\quad for
     every } m \in \mathbb{N}. \]
  Then, for the sequence $(\tmmathbf{\mu}_k^{(n)})_{k \in \mathbb{N}}$ defined
  in (\ref{tavradiamouleipis}) we have, for every $k \in \mathbb{N}$, the equality
  \begin{equation}
    \tmmathbf{\mu}_k^{(n)} = \sum_{\underset{\pi = \{V_1, \ldots, V_l \}}{\pi
    \in \tmop{NC} (k)}} \sum_{\lambda_1 + \cdots + \lambda_l = n}
    \frac{1}{\lambda_1 ! \ldots \lambda_l !} c_{|V_1 |}^{(\lambda_1)} \ldots
    c_{|V_l |}^{(\lambda_l)}. \label{Zouganeli}
  \end{equation}
\end{theorem}

\begin{proof}
  In order to prove the claim, we describe how the indices $b_1, \ldots, b_n,
  m$ from the sum
  \[ \sum_{b_1 + 2 b_2 + \cdots + n \text{} b_n = n} \sum_{m = 0}^{k - b_1 -
     \cdots - b_n} \frac{k!}{m! (m + 1) ! (k - m - b_1 - \cdots - b_n) !b_1 !
     \ldots b_n !} \text{{\hspace{8em}}} \]
  \begin{equation}
    \text{{\hspace{11em}}} \times \left. \frac{\mathd^m}{\mathd x^m} ((\Psi'_0
    (x))^{k - m - b_1 - \cdots - b_n} (\Psi'_1 (x))^{b_1} \ldots (\Psi'_n
    (x))^{b_n}) \right|_{x = 0} \label{GIANAMPEISmestoclub}
  \end{equation}
  can be matched with those from the right hand side of
  (\ref{Zouganeli}). Note that for
  $\pi = \{V_1, \ldots, V_{k - m} \} \in \tmop{NC} (k)$, with $m = 0,
  \ldots, k - 1$, the sum
  \begin{equation}
    \sum_{\lambda_1 + \cdots + \lambda_{k - m} = n} \frac{n!}{\lambda_1 !
    \ldots \lambda_{k - m} !} c_{|V_1 |}^{(\lambda_1)} \ldots c_{|V_{k - m}
    |}^{(\lambda_{k - m})} \label{ipovrixieskatastrofes}
  \end{equation}
  looks like a formal $n$-th derivative of the product $c^{(0)}_{|V_1 |}
  \ldots c^{(0)}_{|V_{k - m} |}$, if we view $c_l^{(i)}$ as a formal $i$-th
  derivative of $c_l^{(0)}$. For fixed $b_1, \ldots, b_n$ with $b_1 + 2 b_2 +
  \cdots + n \text{} b_n = n$ and $m \in \{0, \ldots, k - b_1 - \cdots - b_n
  \}$, let $b_0 \assign k - m - b_1 - \cdots - b_n$. Then, by Leibniz
  rule,
  \[ \left. \frac{\mathd^m}{\mathd x^m} ((\Psi'_0 (x))^{b_0} \ldots (\Psi'_n
     (x))^{b_n}) \right|_{x = 0} = \sum \frac{m!}{\prod_{i = 1}^{b_0} l_{i, 0}
     ! \ldots \prod_{i = 1}^{b_n} l_{i, n} !}  \prod_{i = 1}^{b_0}
     \Psi_0^{(l_{i, 0} + 1)} (0) \ldots \prod_{i = 1}^{b_n} \Psi_n^{(l_{i, n}
     + 1)} (0) \]
  \begin{equation}
    \text{{\hspace{10em}}} = \sum m! \left( \prod_{i = 0}^n (i!)^{b_i}
    \right)^{- 1} \prod_{i = 1}^{b_0} c^{(0)}_{l_{i, 0} + 1} \ldots \prod_{i =
    1}^{b_n} c^{(n)}_{l_{i, n} + 1}, \label{koritsakimoupanourgoo}
  \end{equation}
  where the sum is over all tuples of non-negative integers $((l_{i, 0})_{i =
  1}^{b_0}, \ldots, (l_{i, n})_{i = 1}^{b_n})$ satisfying the constraint
  \[ \sum_{i = 1}^{b_0} l_{i, 0} + \cdots + \sum_{i = 1}^{b_n} l_{i, n} = m.
  \]
  Thus, in order to pass from (\ref{GIANAMPEISmestoclub}) to (\ref{Zouganeli}),
  we will provide a correspondence between summands of
  (\ref{koritsakimoupanourgoo}) and summands of (\ref{ipovrixieskatastrofes}),
  where we differentiate one time $b_1$ factors, two times $b_2$
  factors,{\textdots}, $n$ times $b_n$ factors of $c^{(0)}_{|V_1 |} \ldots
  c^{(0)}_{|V_{k - m} |}$, i.e.,
  \[ |\{i = 1, \ldots, k - m \of \lambda_i = j\}| = b_j, \]
  for every $j = 1, \ldots, n$. For this reason we need to have an equality $m \leq k - b_1 -
  \cdots - b_n$, since such a partition must have at least $b_1 + \cdots +
  b_n$ blocks. On the other hand, the integers $l_{i, j} + 1$ will express the
  cardinalities of the blocks of $\pi$. We consider a tuple $((\lambda_{i, 0})_{i =
  1}^{b_0}, \ldots, (\lambda_{i, n})_{i = 1}^{b_n})$ such that their sum
  is equal to $m$. Previously, in the cases $n = 0, 1, 2$, we proved the claim by showing
  that the contribution of all $\sigma = \{W_1, \ldots, W_{k - m} \} \in
  \tmop{NC} (k)$ that have the same number of blocks with $i$ elements, for
  every $i = 1, \ldots, k$, to the right hand side of (\ref{Zouganeli}), is
  equal to the contribution of $(|W_1 | - 1, \ldots, |W_{k - m} | - 1)$ to
  $\tmmathbf{\mu}_k^{(n)}$. For arbitrary $n \in \mathbb{N}$, we will show that the contribution of
  \begin{equation}
    \prod_{i = 1}^{b_0} c_{\lambda_{i, 0} + 1}^{(0)} \ldots \prod_{i =
    1}^{b_n} c_{\lambda_{i, n} + 1}^{(n)} \label{ginaikespolles}
  \end{equation}
  to $\tmmathbf{\mu}_k^{(n)}$ and to the right hand side of (\ref{Zouganeli})
  will be the same. This suffices to prove the claim. For this purpose, for $j
  = 0, \ldots, n$, we assume that $\{\lambda_{1, j}, \ldots, \lambda_{b_j, j}
  \} = \{\mu_{1, j}, \ldots, \mu_{a_j, j} \}$, where $\mu_{p, j} \neq \mu_{q,
  j}$ for every $p \neq q$. We also define
  \[ r_{i, j} \assign |\{p = 1, \ldots, b_j \of \lambda_{p, j} = \mu_{i, j}
     \}|, \]
  for every $i = 1, \ldots, a_j$ and $j = 0, \ldots, n$. Then, one has
  \[ \sum_{i = 1}^{a_j} r_{i, j} = b_j \text{\quad and\quad} \sum_{i =
     1}^{a_j} r_{i, j} \mu_{i, j} = \sum_{i = 1}^{b_j} \lambda_{i, j} . \]
  Therefore, the number of different ways for
  \[ \prod_{i = 1}^{b_j} c_{\lambda_{i, j} + 1}^{(j)} = \prod_{i = 1}^{a_j}
     (c_{\mu_{i, j} + 1}^{(j)})^{r_{i, j}} \]
  to become a factor of a summand of (\ref{koritsakimoupanourgoo}) is equal
  to the number of different ways that we can cover $b_j$ points on a line
  segment with $r_{1, j}$ number of elements $\mu_{1, j}$, {\textdots},
  $r_{a_j, j}$ number of elements $\mu_{a_j, j}$. This is equal to $(\prod_{i
  = 1}^{a_j} r_{i, j} !)^{- 1} b_j !$, and it implies that
  (\ref{ginaikespolles}) will contribute
  \[ \left( (m + 1) ! \cdot \prod_{j = 0}^n \prod_{i = 1}^{a_j} r_{i, j} !
     \cdot \prod_{i = 0}^n (i!)^{b_i} \right)^{- 1} k! \]
  times to $\tmmathbf{\mu}_k^{(n)}$. 
  
  Now let us analyze the right hand side of
  (\ref{Zouganeli}). Its summands with factor (\ref{ginaikespolles})
  correspond to non-crossing partitions $\{\{V_{i, 0} \}_{i = 1}^{b_0},
  \ldots, \{V_{i, n} \}_{i = 1}^{b_n} \} \in \tmop{NC} (k)$ such that $|V_{i,
  j} | = \lambda_{i, j} + 1$, for every $i = 1, \ldots, b_j$ and $j = 0,
  \ldots, n$. For such partitions the corresponding products
  \begin{equation}
    \prod_{i = 1}^{b_0} c_{|V_{i, 0} |}^{(0)} \ldots \prod_{i = 1}^{b_n}
    c_{|V_{i, n} |}^{(0)} \label{paragwgos}
  \end{equation}
  have the same formal $n$-th derivative. We assume that such a partition has
  $r_1$ blocks with $\mu_1 + 1$ elements,{\textdots}, $r_a$ blocks with $\mu_a
  + 1$ elements. Then
  \[ r_1 + \cdots + r_a = k - m \text{\quad and\quad} r_1 (\mu_1 + 1) + \cdots
     + r_a (\mu_a + 1) = k. \]
  These $r_i$'s depend on whether some of the $\mu_{i, j}$'s are equal to each other. By
  definition, for every $p = 1, \ldots, a$, $r_p$ is a finite sum of $r_{i,
  j}$'s and every $r_{i, j}$ will be a summand of a unique $r_p$. The $r_{i,
  j}$ will be a summand of $r_p$ if and only if $\mu_{i, j} = \mu_p$. Thus,
  for every $p = 1, \ldots, a$, one has
  \[ r_p = \sum_{j = 0}^n \sum_{i = 1}^{a_j} r_{i, j} \tmmathbf{1}_{\{\mu_{i,
     j} = \mu_p \}} . \]
  There are $((m + 1) ! \prod_{i = 1}^a r_i !)^{- 1} k!$ such non-crossing
  partitions. Moreover, for such partitions we are interested in summands inside
  \begin{equation}
    \sum \frac{1}{\prod_{i = 1}^{r_1} l_{i, 1} ! \ldots \prod_{i = 1}^{r_a}
    l_{i, a} !} \prod_{i = 1}^{r_1} c_{\mu_1 + 1}^{(l_{i, 1})} \ldots \prod_{i
    = 1}^{r_a} c_{\mu_a + 1}^{(l_{i, a})} \label{Whatelse}
  \end{equation}
  which have factor (\ref{ginaikespolles}). The above sum is over all tuples
  of nonnegative integers $((l_{i, 1})_{i = 1}^{r_1}, \ldots, (l_{i, a})_{i =
  1}^{r_a})$, such that their sum is equal to $n$. These terms appear when the $l_{i,
  j}$'s have been chosen in order to get a product where we differentiate one
  time $b_1$ terms $c_{\mu_i + 1}^{(0)}$, {\textdots}, $n$ times $b_n$ terms
  $c_{\mu_i + 1}^{(0)}$. Thus, for these summands, the product of the
  corresponding $l_{i, j} !$ is equal to $\prod_{i = 0}^n (i!)^{b_i}$.
  We have to find the number of ways we have to choose the $l_{i, j}$'s in
  order to get such a product with factor (\ref{ginaikespolles}). For this
  purpose, for some $p \in \{1, \ldots, a\}$, we assume that $r_p = r_{i_1,
  j_1} + \cdots + r_{i_l, j_l}$, where $0 \leq j_1 < \cdots < j_l \leq n$.
  This implies that $\mu_{i_1, j_1} = \cdots = \mu_{i_l, j_l} = \mu_p$. Therefore, 
  the number of different ways for
  \[ \left( c_{\mu_{i_1, j_1} + 1}^{(j_1)} \right)^{r_{i_1, j_1}} \ldots
     \left( c_{\mu_{i_l, j_l} + 1}^{(j_l)} \right)^{r_{i_l, j_l}} \]
  to become a factor of a summand of (\ref{Whatelse}) is equal to the number
  of different ways that we can cover $r_p$ points on a line segment with
  $r_{i_1, j_1}$ number of elements $j_1$, {\textdots}, $r_{i_l, j_l}$ number
  of elements $j_l$. This is equal to $(r_{i_1, j_1} ! \ldots r_{i_l, j_l}
  !)^{- 1} r_p !$. Therefore, we deduce that the term (\ref{ginaikespolles})
  contributes
  \[ \frac{k!}{(m + 1) ! \prod_{i = 1}^a r_i !}  \frac{\prod_{i = 1}^a r_i
     !}{\prod_{j = 0}^n \prod_{i = 1}^{a_j} r_{i, j} ! \cdot \prod_{i = 0}^n
     (i!)^{b_i}} \]
  times to the right hand side of (\ref{Zouganeli}). This proves the claim.
\end{proof}

\begin{remark}
  Theorem \ref{fredi} and Theorem \ref{th:mom-cum-high} show that the limit regime
  (\ref{100}) for a random Hermitian matrix $A$ provides us a version of Taylor
  expansion for the average empirical distribution functional of $A / N$. An
  infinitesimal non-commutative probability space of higher order plays the
  role of the derivatives. More precisely, for the functional
  $\varphi_N : \mathbb{C} \langle \mathbf{x} \rangle \rightarrow \mathbb{C}$,
  defined in (\ref{pornovioS}), we have the expansion
  \begin{equation}
    \varphi_N (P) = \varphi^{(0)} (P) + \frac{1}{N^{\varepsilon}}
    \varphi^{(1)} (P) + \cdots + \frac{1}{N^{n \varepsilon}} 
    \frac{\varphi^{(n)} (P)}{n!} + \omicron (N^{- n \varepsilon}), \text{\quad
    for every } P \in \mathbb{C} \langle \mathbf{x} \rangle, \label{FANIS}
  \end{equation}
  where $(\mathbb{C} \langle \mathbf{x} \rangle, \varphi^{(0)}, \ldots,
  \varphi^{(n)})$ is an infinitesimal non-commutative probability space of
  order $n$ with infinitesimal non-crossing cumulant functionals given by
  $(\tmmathbf{\kappa}^{(0)}_m)_{m \in \mathbb{N}}, \ldots,
  (\tmmathbf{\kappa}_m^{(n)})_{m \in \mathbb{N}}$, where
  \[ \tmmathbf{\kappa}_m^{(i)} (\mathbf{x}, \ldots, \mathbf{x}) = \frac{i!
     \Psi_i^{(m)} (0)}{(m - 1) !}, \]
  for every $i = 0, \ldots, n$ and $m \in \mathbb{N}$. 
\end{remark}

\begin{example}
  We compute the corresponding infinitesimal non-commutative probability space
  $$(\mathbb{C} \langle \mathbf{x} \rangle, \varphi^{(0)}, \varphi^{(1)},
  \varphi^{(2)}, \varphi^{(3)}),$$ described in (\ref{FANIS}), for the random
  matrix $A = N \cdot A_1 + N^{1 - \varepsilon / 2} A_2 + N^{1 - \varepsilon}
  A_3 + N^{1 - 3 \varepsilon / 2} A_4$, where $A_1, A_2, A_3, A_4$ are
  independent GUE matrices. Then $\varphi_0$ is given by the semicircle distribution
  and the linear functionals $\varphi_1, \varphi_2$ have been computed in
  Example \ref{ZNisxoli} and Example \ref{olesoigeitonies}, respectively.
  Theorem \ref{fredi} implies that the values $\varphi_3 (\mathbf{x}^k)$ are
  equal to
  \[ \sum_{m = 0}^{k - 3} \frac{k!}{m! (m + 1) ! (k - m - 3) !}  \left.
     \frac{\mathd^m}{\mathd x^m} (x^{k - m}) \right|_{x = 0} + 6 \sum_{m =
     0}^{k - 2} \frac{k!}{m! (m + 1) ! (k - m - 2) !}  \left.
     \frac{\mathd^m}{\mathd x^m} (x^{k - m}) \right|_{x = 0} \]
  \[ + 6 \sum_{m = 0}^{k - 1} \frac{k!}{m! (m + 1) ! (k - m - 1) !}  \left.
     \frac{\mathd^m}{\mathd x^m} (x^{k - m}) \right|_{x = 0} . \]
  Using the Cauchy formula, we see that this is equal to
  \[ \frac{k (k - 1)}{2 \mathpi \mathi} \oint_{|z| = 1} z^3 \left( z +
     \frac{1}{z} \right)^{k - 2} d \text{} z + \frac{3 k}{\mathpi \mathi}
     \oint_{|z| = 1} z^2 \left( z + \frac{1}{z} \right)^{k - 1} d \text{} z +
     \frac{3}{\mathpi \mathi} \oint_{|z| = 1} z \left( z + \frac{1}{z}
     \right)^k d \text{} z. \]
  This implies that the linear functional $\varphi^{(3)}$ is given by
  \[ \varphi^{(3)} (P) = \frac{1}{2 \mathpi} \int_{- 2}^2 \left( \frac{t^4 - 4
     t^2 + 2}{\sqrt{4 - t^2}} P'' (t) + \frac{6 t^3 - 18 t}{\sqrt{4 - t^2}} P'
     (t) + \frac{6 t^2 - 12}{\sqrt{4 - t^2}} P (t) \right) d \text{} t, \]
  for every $P \in \mathbb{C} \langle \mathbf{x} \rangle$.
\end{example}

\begin{example}
\label{ex:BBP-highOrder}
  We consider the random matrix
  \[ A = N \cdot G + \theta_1 \sum_{i = 1}^{\lfloor N^{1 - \varepsilon}
     \rfloor} X_i X_i^{\ast} + \theta_2 \sum_{i = 1}^{\lfloor N^{1 - 2
     \varepsilon} \rfloor} Y_i Y_i^{\ast}, \]
  where the matrices $G, \{X_i \}, \{Y_i \}$ are independent, $G$ is a GUE of
  size $N$ and $X_i, Y_i$ are independent standard complex Gaussian $N \times
  1$ vectors. 
  
  Then, in notation of Theorem \ref{fredi}, we have that
  \[ \Psi_0 (x) = \frac{x^2}{2} \text{, \qquad} \Psi_1 (x) = \log (1 -
     \theta_1 x)^{- 1} \text{, \qquad} \Psi_2 (x) = \log (1 - \theta_2 x)^{-
     1} . \]
  
  Applying the theorem and performing similar to previous examples computations, 
  we arrive to the following answer.   
  Consider the infinitesimal non-commutative probability space of
  order 2, $(\mathbb{C} \langle X \rangle, \varphi^{(0)}, \varphi^{(1)},
  \varphi^{(2)})$, where $\varphi^{(0)}$ is the limit, $\varphi^{(1)}$ is the
  first order correction and $\varphi^{(2)}$ is the second order correction.     
     
  This implies that
  \[ \varphi^{(1)} (P) =\tmmathbf{1}_{| \theta_1 | \geq 1} \cdot P (\theta_1 +
     \theta_1^{- 1}) - \frac{1}{2 \pi} \int_{- 2}^2 \frac{t - 2 \theta_1}{(t -
     \theta_1 - \theta_1^{- 1}) \sqrt{4 - t^2}} P (t) d \text{} t, \]
  and
  \[ \varphi^{(2)} (P) =\tmmathbf{1}_{| \theta_2 | \geq 1} \cdot P (\theta_2
     + \theta_2^{- 1}) - \frac{1}{2 \pi} \int_{- 2}^2 \frac{t - 2 \theta_2}{(t
     - \theta_2 - \theta_2^{- 1}) \sqrt{4 - t^2}} P (t) d \text{} t \text{
     {\hspace{12em}}} \]
  \[ \text{{\hspace{6em}}} + \frac{1}{2 \mathpi} \int_{- 2}^2
     \frac{\theta_1^2}{\sqrt{4 - t^2}}  \frac{P' (t) - P' (\theta_1 +
     \theta^{- 1}_1)}{t - \theta_1 - \theta_1^{- 1}} d \text{} t \hspace{16em}
  \]
  \[ \text{{\hspace{7em}}} \quad - \frac{1}{4 \mathpi} \int_{- 2}^2
     \frac{(\theta_1^2 - 1) (t - 2 \theta_1)}{(t - \theta_1 - \theta_1^{- 1})
     \sqrt{4 - t^2}} \left( \frac{P' (t) - P' (\theta_1 + \theta_1^{- 1})}{t -
     \theta_1 - \theta_1^{- 1}} - P'' (\theta_1 + \theta_1^{- 1}) \right) d
     \text{} t. \]

Note that the second correction lives on the scale $N^{1-2 \varepsilon}$, for $\varepsilon <\frac12$. Therefore, the presence of the delta measure in $\varphi^{(2)} (P)$ means the existence of $N^{1-2 \varepsilon}$ outliers at $\theta_2+\theta_2^{- 1}$ in case $| \theta_2 | \geq 1$. This can be viewed as a \textit{higher order} BBP phase transition --- it is not visible from the first order infinitesimal freeness (which governs $N^{1 - \varepsilon}$ scale), so in order to see these outliers, one needs to consider higher orders of infinitesimal freeness. 
     
\end{example}

\end{document}